\documentclass[10pt,a4paper]{amsart}
\usepackage{amsthm, amsfonts, amssymb, amsmath, latexsym, enumerate}
\usepackage[all]{xy} \usepackage[latin1]{inputenc}
\usepackage{array}
\usepackage[naturalnames=true]{hyperref}

 \newcommand*\patchAmsMathEnvironmentForLineno[1]{%
   \expandafter\let\csname old#1\expandafter\endcsname\csname #1\endcsname
   \expandafter\let\csname oldend#1\expandafter\endcsname\csname end#1\endcsname
   \renewenvironment{#1}%
      {\linenomath\csname old#1\endcsname}%
      {\csname oldend#1\endcsname\endlinenomath}}%
 \newcommand*\patchBothAmsMathEnvironmentsForLineno[1]{%
   \patchAmsMathEnvironmentForLineno{#1}%
   \patchAmsMathEnvironmentForLineno{#1*}}%
 \AtBeginDocument{%
 \patchBothAmsMathEnvironmentsForLineno{equation}%
 \patchBothAmsMathEnvironmentsForLineno{align}%
 \patchBothAmsMathEnvironmentsForLineno{flalign}%
 \patchBothAmsMathEnvironmentsForLineno{alignat}%
 \patchBothAmsMathEnvironmentsForLineno{gather}%
 \patchBothAmsMathEnvironmentsForLineno{multline}%
 }
\usepackage{anysize}
\marginsize{2.5cm}{2.5cm}{2.9cm}{2.9cm}

\usepackage[pagewise]{lineno}

\usepackage{mathrsfs}
\newcommand{\mathbold}{\mathbf}

\newtheorem{thm}{Theorem}[section] \newtheorem{lem}[thm]{Lemma}
\newtheorem{corol}[thm]{Corollary} \newtheorem{prop}[thm]{Proposition}
\newtheorem*{thm*}{Theorem}

\theoremstyle{definition} \newtheorem{rmk}[thm]{Remark}
 \newtheorem{dfn}[thm]{Definition}

\newtheorem{step}{Step}
\newtheorem{steppe}{Step}
\newtheorem*{ack}{Acknowledgments}

\newcommand{\OO}{\mathscr{O}}
\newcommand{\FF}{\mathscr{F}}
\newcommand{\EE}{\mathscr{E}}
\newcommand{\GG}{\mathscr{G}}

\newcommand{\EExt}{\mathscr{E}xt}
\newcommand{\HHom}{\mathscr{H}om}
\newcommand{\TTor}{\mathscr{T}or}
\newcommand{\LL}{\mathscr{L}}
\newcommand{\PP}{\mathscr{P}}

\newcommand{\sH}{\mathscr{H}}
\newcommand{\sK}{\mathscr{K}}

\newcommand{\cH}{\mathcal{H}}
\newcommand{\cE}{\mathcal{E}}
\newcommand{\cK}{\mathcal{K}}
\newcommand{\cL}{\mathcal{L}}
\newcommand{\cF}{\mathcal{F}}
\newcommand{\cN}{\mathcal{N}}
\newcommand{\cP}{\mathcal{P}}
\newcommand{\cM}{\mathcal{M}}
\newcommand{\HHH}{\mathcal{H}}
\newcommand{\UU}{\mathcal{U}}

\newcommand{\cI}{\mathcal{I}}

\newcommand{\Mo}{{\sf M}}
\newcommand{\Wo}{{\sf W}}
\newcommand{\Po}{{\sf P}}
\newcommand{\No}{{\sf N}}
\newcommand{\Mos}{{\sf M^s}}

\DeclareMathOperator{\GL}{{\sf GL}}
\DeclareMathOperator{\Spin}{{\sf Spin}}
\DeclareMathOperator{\s}{{\sf S}}

\DeclareMathOperator{\rk}{rk}

\DeclareMathOperator{\Ext}{Ext} 
\DeclareMathOperator{\Sym}{Sym}
\DeclareMathOperator{\Hom}{Hom} 
\DeclareMathOperator{\im}{Im} \DeclareMathOperator{\cok}{cok}
\DeclareMathOperator{\HH}{H} \DeclareMathOperator{\hh}{h}
\DeclareMathOperator{\ext}{ext}
\DeclareMathOperator{\len}{len}
\DeclareMathOperator{\Pic}{Pic}
\DeclareMathOperator{\supp}{supp}

\newcommand{\Z}{\mathbb Z} \newcommand{\C}{\mathbb C}
 \newcommand{\p}{\mathbb P}
 \newcommand{\G}{\mathbb G}

\DeclareMathOperator{\D}{\mathbf{D^b}}

\newcommand{\RR}{\mathbold{R}}
\newcommand{\LLb}{\mathbold{L}}

\newcommand{\ph}{\mathbf{\Phi}}
\newcommand{\phs}{\mathbf{\Phi^{*}}}
\newcommand{\phx}{\mathbf{\Phi^{!}}}

\DeclareMathOperator{\ts}{\otimes}
\newcommand{\rr}{\rightarrow}

\newcommand{\mono}{\hookrightarrow}

\newcommand{\xr}{\xrightarrow}
\newcommand{\f}{F}
\newcommand{\e}{E}

\SelectTips{cm}{} \newdir{ >}{{}*!/-5pt/@{>}}
\numberwithin{equation}{section}

\allowdisplaybreaks[1]

\begin{document}


\title[Bundles on Fano threefolds of genus $7$]{Vector bundles on Fano threefolds \\
of Genus $7$ and Brill-Noether Loci}

\author{Maria Chiara Brambilla}
\email{{\tt brambilla@dipmat.univpm.it}}
\address{
Dipartimento di Scienze Matematiche,  Universit\`a Politecnica delle
Marche, Ancona, Italy}
\urladdr{\tt{\url{http://www.dipmat.univpm.it/~brambilla}}}

\author{Daniele Faenzi}
\email{{\tt daniele.faenzi@univ-pau.fr}}
\address{Universit\'e de Pau et des Pays de l'Adour \\
  Av. de l'universit\'e - BP 576 - 64012 PAU Cedex - France}
\urladdr{\tt{\url{http://www.univ-pau.fr/~dfaenzi1}}}

\thanks{Both authors were partially supported by Italian MIUR funds.
The second author was partially supported by GRIFGA
and by ANR contract Interlow ANR-09-JCJC-0097-01}


\keywords{Fano threefolds of genus 7.
  Moduli space of vector bundles.
  Brill-Noether theory for vector bundles on curves.
  Semiorthogonal decomposition.
  Lagrangian submanifold of symplectic manifolds.}

\subjclass[2000]{Primary 14J60. Secondary 14H30, 14F05, 14D20.}


\sloppy

\begin{abstract}
Given a smooth prime Fano threefold $X$ of genus $7$ we consider
its homologically projectively dual curve $\Gamma$ and the natural
integral functor $\phx:\D(X) \to \D(\Gamma)$.

We prove that, for $d\geq 6$, $\phx$ gives a birational map from a component of the
moduli scheme $\Mo_X(2,1,d)$ of rank $2$ stable sheaves on $X$ with
$c_1=1$, $c_2=d$ to a
generically smooth $(2\,d-9)$-dimensional component of the
Brill-Noether variety $W^{2\,d-11}_{d-5,5\,d-24}$
of stable vector bundles on
$\Gamma$ of rank $d-5$ and degree $5\,d-24$ with at least $2\,d-10$
sections.

This map turns out to be an isomorphism for $d=6$, and the moduli
space $\Mo_X(2,1,6)$ is fine. For general $X$, this moduli space
is a smooth irreducible threefold.
\end{abstract}

\maketitle

\section{Introduction}

Let $X$ be a smooth complex projective variety of dimension $3$,
with Picard number one, and assume that the anticanonical divisor
$K_{X}$ is ample. Then $X$ is called a Fano threefold, and one
defines the index $i_X$ of $X$ as the greatest integer $i$ such that
$-K_X/i$ lies in $\Pic(X)$. We are interested in the Maruyama moduli
scheme $\Mo_X(2,c_1,c_2)$ of semistable sheaves of rank $2$ and
Chern classes $c_1$, $c_2$, and with $c_3=0$, defined on a Fano
threefold $X$.

The maximum value of $i_X$ is $4$, and in this case $X$ must be
isomorphic to $\p^3$.
The study of the moduli space $\Mo_{\p^3}(2,c_1,c_2)$ was
pioneered by Barth in \cite{barth:some-properties},
and pursued later by several authors.
Roughly speaking, the main questions concern rationality,
irreducibility and smoothness of these moduli spaces;
many of them are still open.
Among the main tools to study the problem, we recall monads
and Beilinson's theorem, see \cite{barth-hulek},
\cite{beilinson:derived-and-linear} and \cite{okonek-schneider-spindler}.

The next case is $i_X = 3$. Then $X$ has to be isomorphic to a quadric
hypersurface.
This case was considered by Ein and Sols (\cite{ein-sols}) and later by
Ottaviani and Szurek, see \cite{ottaviani-szurek}.

In the case $i_X=2$, there are $5$ deformation classes of Fano
threefolds as it results from Iskovskikh's classification, see
\cite{fano-encyclo}. Perhaps the most studied among them is the
cubic hypersurface $V_3$ in $\p^4$. The geometry of these
threefolds is deeply linked to the properties of the families of
curves they contain. A cornerstone in this sense is the paper
\cite{clemens-griffiths} of Clemens and Griffiths on $V_3$. For a
survey of results about moduli spaces of vector bundles on $V_3$
we refer to \cite{beauville:cubic}. In particular we mention
\cite{druel:cubic-3-fold}, \cite{markushevich-tikhomirov} and \cite{ballico-miro-roig:cohomology}.

In the case $i_X = 1$, we say that $X$ is a {\it prime} Fano
threefold. Then, one defines the genus of $X$ as the integer $g =
-K_X^3/2+1$. The genus satisfies $2\leq g \leq 12$, $g\neq 11$,
and there are $10$ deformation classes of prime Fano threefolds.
Their birational geometry has been extensively studied as well,
see \cite{fano-encyclo}. The geometry of the moduli spaces of rank
$2$ vector bundles on $X$ has been more recently investigated by
several authors, for instance in the papers
\cite{iliev-markushevich:quartic} (for genus $3$),
\cite{iliev-markushevich:genus-7},
\cite{iliev-markushevich:sing-theta:asian} (for genus $7$),
\cite{iliev-manivel:genus-8}, \cite{iliev-markushevich:cubic} (for
genus $8$), \cite{iliev-ranestad} (for genus $9$),
\cite{enrique-dani:v22} (for genus $12$). Among the main tools we
mention the Abel-Jacobi map and Serre's correspondence between
rank $2$ vector bundles and curves contained in $X$.
\medskip

The purpose of the present paper is to investigate the properties of the
moduli spaces of rank $2$ bundles on a smooth prime Fano threefold
$X$, making use of homological methods.
We first observe that (under a mild generality assumption on $X$),
given any integer $d\geq g/2+1$, the moduli space $\Mo_X(2,1,d)$
contains a generically smooth
component $\Mo(d)$ of dimension $2\,d-g-2$, such that its
general
element $F$ is a stable
locally free sheaf with $\HH^1(X,F(-1))=0$, see Theorem
\ref{add-line}.

Then, we focus on genus $7$, where an analogue of Beilinson's theorem
is provided by the semiorthogonal decomposition of the bounded derived category
$\D(X)$ obtained by Kuznetsov in \cite{kuznetsov:v12}.
We use this decomposition to study the component $\Mo(d)$.
More precisely, we consider the
homologically projectively dual curve $\Gamma$ in the sense of
\cite{kuznetsov:hyperplane}, and the corresponding integral
functor $\phx : \D(X) \to \D(\Gamma)$.
Making use of the canonical resolution of a general element
of $\Mo(d)$,
we show that $\phx$ gives a birational map $\varphi$ from
$\Mo(d)$
to a component of $W^{2\,d-11}_{d-5,5\,d-24}$ (Theorem
\ref{thm:moduli-brill}),
where we denote by $W^{s}_{r,c}$ the
Brill-Noether variety of stable vector bundles on
$\Gamma$ of rank $r$ and degree $c$ with at least $s+1$
independent global sections.

We prove that the map $\varphi$ is in fact an isomorphism in the
case $d=6$. In particular the moduli space $\Mo_X(2,1,6)$ is fine
and isomorphic to a
connected threefold (Theorem \ref{thm:lune}, part \ref{lune-A}).
If $X$ is general
enough, the moduli space $\Mo_X(2,1,6)$ is actually smooth and
irreducible (Theorem \ref{thm:lune}, part \ref{lune-B}).
We also exhibit an involution of $\Mo_X(2,1,6)$
which interchanges the set of sheaves which are not globally
generated with the one of those which are not locally free.
Finally we show that, if $S$ is a general hyperplane section
surface, the space $\Mo_X(2,1,6)$ embeds as a Lagrangian
subvariety of $\Mo_S(2,1,6)$ with respect to the Mukai form, away
from finitely many double points (Theorem \ref{ty}).

The paper is organized as follows. In Section \ref{sec:pre} we
review the geometry of Fano threefolds $X$ of genus $7$ and the
structure of their derived category. In Section \ref{sec:3} we
construct (under mild generality assumptions) a generically smooth
component $\Mo(d)$ of $\Mo_{Y}(2,1,d)$, over a smooth prime Fano
threefold $Y$ (of any genus), and we recall some basic facts concerning bundles
with minimal $c_{2}$. Then again we work on genus $7$: in Section \ref{sezione:cubics}, we prove
that the functor $\phx$ provides an isomorphism between the
Hilbert scheme $\sH^0_3(X)$ and the symmetric cube $\Gamma^{(3)}$,
(see Theorem \ref{thm:cubic-iso}). Section \ref{sec:moduli-bn}
contains our main results.

\begin{ack}
We are grateful to Kieran O'Grady for pointing out the paper by
Tyurin, and to Dimitri Markushevich for helping us in making more
precise the generality assumption in Theorem \ref{thm:lune}.
We would like to thank the referee for the extremely accurate reading
of our manuscript, and for several comments that
greatly helped us to improve our paper.
\end{ack}

\section{Preliminaries} \label{sec:pre}

Let us introduce some basic material. The main notions we will need concern moduli spaces of
semistable sheaves, smooth Fano threefolds, Brill-Noether varieties and a bit of homological algebra.
Throughout the paper we work over the field of complex numbers.

\subsection{Notation and preliminary results}

Given a smooth connected complex projective $n$-dimensional polarized
variety $(X,H_X)$, and a sheaf $F$ on $X$, we write $F(t)$ for
$F\ts \OO_X(t H_X)$. Given a subscheme $Z$ of $X$, we write $F_Z$
for $F\ts \OO_Z$ and we denote by $\cI_{Z,X}$ the ideal sheaf of
$Z$ in $X$, and by $N_{Z,X}$ its normal sheaf. We will frequently
drop the second subscript. Given a pair of sheaves $(F,E)$ on $X$,
we will write $\ext_X^k(F,E)$ for the dimension of the vector space
$\Ext_X^k(F,E)$, and similarly $\hh^k(X,F) = \dim \HH^k(X,F)$. The
Euler characteristic of $(F,E)$ is defined as $\chi(F,E)=\sum_k
(-1)^k \ext_X^k(F,E)$ and $\chi(F)$ is defined as $\chi(\OO_X,F)$.
We denote by $p(F,t)$ the Hilbert polynomial $\chi(F(t))$ of the
sheaf $F$. The dualizing sheaf of $X$ is denoted by $\omega_{X}$.
We define also the natural evaluation map:
 \[
 e_{E,F}: \Hom_Y(E,F) \ts E \to F.
 \]

If the divisor $H_X$ embeds $X$ in $\p^m$,
we say that $X$ is
{\it ACM} (for {\it arithmetically Cohen-Macaulay}) if its coordinate ring is Cohen-Macaulay.
If $X$ is ACM and $n \ge 1$, a locally free sheaf $F$ on $X$ is called {\em ACM}
if it has no intermediate
cohomology, i.e. if $\HH^k(X,F(t))=0$ for all integer $t$ and for any
$0<k<n$. This is equivalent to $\oplus_{t\in \Z} \HH^0(X,F(t))$ being a
Cohen-Macaulay module over the coordinate ring of $X$.

\subsubsection{Coherent sheaves}

Let again $X$ be a smooth connected complex projective $n$-dimensional variety.
The Chern classes $c_k(F)$ are defined for a vector bundle $F$ on
$X$, and take values in $\HH^{k,k}(X)$.
In the sequel, the Chern classes will be denoted by integers as soon as
$\HH^{k,k}(X)$ has dimension $1$.
The Chern polynomial of a rank $r$ vector bundle $F$ is defined to be
$c_F(t)=1+c_1(F)t+\ldots+c_r(F)t^r$ and it is multiplicative on exact
sequences of vector bundles.
Hence the Chern classes can be defined on the Grothendieck group of vector bundles on $X$,
and in fact on the Grothendieck group of all coherent sheaves (they
are isomorphic on a smooth variety).
For more details see e.g. \cite[Appendix A]{hartshorne:ag}.

We denote by $F^*=\HHom_X(F,\OO_X)$ the dual of a coherent sheaf
$F$ on $X$. Recall that a coherent sheaf $F$ on $X$ is
{\it reflexive} if the natural map $F\to F^{**}$ of $F$ to its
double dual is an isomorphism. We recall here some basic facts on
reflexive sheaves, which will be useful in the sequel, and we refer to
\cite{hartshorne:stable-reflexive} for more details.
Any locally free sheaf is reflexive, and any reflexive sheaf is
torsion-free.
A coherent sheaf $F$ on $X$ is
reflexive if and only if it can be included into a locally free sheaf
$E$ with $E/F$ torsion-free, see
\cite[Proposition 1.1]{hartshorne:stable-reflexive}. Moreover, by
\cite[Proposition 1.9]{hartshorne:stable-reflexive},
any reflexive rank-$1$ sheaf is invertible (recall that $X$ is
smooth and irreducible).
Finally, we will often use a straightforward generalization of \cite[Proposition
2.6]{hartshorne:stable-reflexive} which implies that the third Chern
class $c_3(F)$ of a rank $2$ reflexive sheaf $F$ on a smooth projective
threefold satisfies $c_3(F)\ge0$, and vanishes if and only if
$F$ is locally free.

\subsubsection{Semistable sheaves and their moduli spaces}

We refer to the book \cite{huybrechts-lehn:moduli} for a detailed account of all the notions introduced here.
We recall that a torsion-free coherent sheaf $F$ on $X$ is {\it Gieseker-semistable},
(shortly, {\it semistable})
if for any coherent subsheaf $E$, with $\rk(E)<\rk(F)$,
one has $p(E,t)/\rk(E) \leq p(F,t)/\rk(F)$ for $t\gg 0$.
The sheaf $F$ is called {\it stable} if the inequality above is strict for all $E$ and $t\gg0$.
It is straightforward to see that a sheaf $F$ is semistable (respectively stable) if and only if
for any torsion-free quotient $Q$, with $\rk(Q)<\rk(F)$,
one has $p(Q,t)/\rk(Q) \geq p(F,t)/\rk(F)$
(respectively $p(Q,t)/\rk(Q) > p(F,t)/\rk(F)$)
for $t\gg 0$.

The {\it slope} of a sheaf $F$ of positive rank is defined as
$\mu(F) = \deg(c_1(F)\cdot H_X^{n-1})/\rk(F)$.
We recall that a torsion-free coherent sheaf $F$ is  {\it $\mu$-semistable} if for
any coherent subsheaf $E$, with $\rk(E) < \rk(F)$,
one has $\mu(E) \leq \mu(F)$.
The sheaf $F$ is called {\it $\mu$-stable} if the above inequality is strict for all $E$.
We recall that the {\it discriminant} of a sheaf $F$ is:
\begin{equation}\label{Delta}
\Delta(F) = 2 r c_2(F) - (r-1) c_1(F)^2.
\end{equation}
Bogomolov's inequality, see e.g. \cite[Theorem 3.4.1]{huybrechts-lehn:moduli},
states that if $F$ is $\mu$-semistable, then we have:
\begin{equation}
\label{eq:bogomolov}
\Delta(F)\cdot H_X^{n-2}\geq 0.
\end{equation}
Recall that by Maruyama's theorem, see
\cite{maruyama:boundedness-small}, if $\dim(X)=n\geq 2$ and $F$ is a
$\mu$-semistable sheaf of rank $r<n$, then
its restriction to a general divisor in $|\OO_X(H_X)|$ is still $\mu$-semistable.

\vspace{0.2cm}
We introduce here some notation concerning moduli spaces.
Once fixed the polarization $H_X$, we denote by $\Mo_X(r,c_1,\ldots,c_n)$ the moduli space of
$S$-equivalence classes of rank-$r$ sheaves which are $H_X$-semistable
and have Chern classes $c_1,\ldots,c_n$.
We will drop the last values of the classes $c_k$ when they are
zero.
We denote by $\Mos_X(r,c_1,\ldots,c_n)$ the subset of stable sheaves of
$\Mo_X(r,c_1,\ldots,c_n)$.
The point of $\Mos_X(r,c_1,\ldots,c_n)$ represented by a sheaf $F$ will
be denoted again by $F$, by abuse of notation.
We denote by $\sH^g_d(X)$ the union of components of the {\it Hilbert scheme} of closed
subschemes $Z$ of $X$ with Hilbert polynomial $p(\OO_Z,t)=d t+1-g$,
containing integral curves of degree $d$ and arithmetic genus $g$.

We use the following terminology:
any claim referring to a {\em general} element in a given parameter
space $P$, will mean that the claim holds for all elements of $P$,
except possibly for those who lie in a countable
union of Zariski closed subsets of $P$.

\subsubsection{Homological algebra}
As a basic tool, we will use the bounded derived category of
coherent sheaves.
Namely, given $X$ as above, we will
consider the derived category $\D(X)$ of complexes of sheaves on $X$
with bounded coherent cohomology.
For definitions and notation we refer to
\cite{gelfand-manin:homological} and \cite{weibel:homological}.
In particular we write $[j]$ for the $j$-th shift to the right in the
derived category.

Let $Z$ be a local complete intersection subvariety of
$X$.
In view of the
Fundamental Local Isomorphism (see
\cite[Proposition III.7.2]{hartshorne:residues-duality}),
we have the natural isomorphisms:
\begin{align}
\label{ext-norm}& \EExt^k_X(\OO_Z,\OO_Z) \cong
\EExt^{k-1}_X(\cI_Z,\OO_Z) \cong \bigwedge^k N_Z, \\
\label{tor-norm}& \TTor_k^X(\OO_Z,\OO_Z) \cong
\TTor_{k-1}^X(\cI_Z,\OO_Z) \cong \bigwedge^k N^*_Z.
\end{align}


For the reader's convenience we recall the following spectral
sequences:
\begin{align}
  \label{ss1}
  & E^{p,q}_2 = \Ext_X^p(\HHH^{-q}(a),A) \Longrightarrow \Ext_X^{p+q}(a,A),   \\
  \label{ss2}
  & E^{p,q}_2 = \Ext_X^p(B,\HHH^{q}(b)) \Longrightarrow \Ext_X^{p+q}(B,b), \\
\label{spectr-ext} & E^{p,q}_2 = \HH^p(X,\EExt^q_X(A,B))  \Longrightarrow
\Ext^{p+q}_X(A,B).
\end{align}
where $a,b$ are complexes of sheaves on $X$, and $A,B$ are sheaves
on $X$. Recall that the maps in the $E_{2}$ term of these spectral sequences are
differentials:
\[
d^{p,q}_2:E^{p,q}_2 \to E^{p+2,q-1}_2.
\]

\subsubsection{Brill-Noether loci for vector bundles on a smooth
  projective curve}
\label{befana}

We recall here some basic results in Brill-Noether
theory, for definitions and notations we refer for instance to
\cite{teixidor-i-bigas:brill-noether-stable}.
Let $\Gamma$ be a smooth connected complex projective curve of genus $g$.
The Brill-Noether locus
$W^{s}_{r,c} \subset \Mos_{\Gamma}(r,c)$
is defined to be the subvariety consisting of rank $r$ stable bundles of
degree $c$ on $\Gamma$ having at least $s+1$
independent global sections.
The expected dimension of this variety is:
\[
\rho(r,c,s) = r^2(g-1)-(s+1)\,(s+1-c+r(g-1))+1.
\]

Consider a stable rank $r$ vector bundle $\cF$ on $\Gamma$, with $\cF$
in $\Mos_{\Gamma}(r,c)$.
We define the Gieseker-Petri map as the natural linear application:
\begin{equation} \label{petri}
\pi_{\cF}: \HH^0(\Gamma,\cF) \ts \HH^0(\Gamma,\cF^* \ts
\omega_{\Gamma}) \to \HH^0(\Gamma,\cF\ts \cF^*\ts
\omega_{\Gamma}).
\end{equation}
The map $\pi_{\cF}$ is injective if and only if $\cF$ is a
non-singular point of a component
of $W^{s}_{r,d}$ of dimension $\rho(r,d,s)$.
We will use more frequently in the sequel the transpose of
the Petri map, that reads:
\begin{equation} \label{dual-petri}
\pi^{\top}_\cF:\Ext^1_\Gamma(\cF,\cF) \to
\HH^0(\Gamma,\cF)^* \ts \HH^1(\Gamma,\cF)
\end{equation}
In fact the tangent space to $W^{s}_{r,d}$ at the point $\cF$ can be interpreted as the kernel of
$\pi^{\top}_\cF$, while the space of obstructions at $\cF$ is identified with the cokernel of $\pi^{\top}_\cF$.

\subsubsection{Smooth prime Fano threefolds}

Let now $X$ be a smooth connected complex projective variety of dimension $3$.
Recall that $X$ is called {\it Fano} if its anticanonical divisor
class $-K_X$ is ample.
A Fano threefold  $X$ is said to be {\it prime} if its Picard group is generated by the class of
$K_X$. These varieties are classified up to deformation, see for
instance \cite[Chapter IV]{fano-encyclo}. The number of
deformation classes is $10$, and they are characterized by the
{\it genus}, which is the integer $g$ such that
$\deg(X)=-K_X^3=2\, g-2$. Recall that the genera of prime Fano
threefolds take values in $\{2,3,\ldots,9,10,12\}$.
If $-K_X$ is very ample, we say that $X$ is {\it non-hyperelliptic}. In this
case we have $g\geq 3$.

If $X$ is a prime Fano threefold of genus $g$, the Hilbert scheme
$\sH^{0}_{1}(X)$ of lines contained in $X$ is a scheme of
dimension $1$. It is known by \cite{iskovskih:II} that the normal
bundle of a line $L\subset X$ splits either as $\OO_{L} \oplus
\OO_{L}(-1)$ or as $\OO_{L}(1) \oplus \OO_{L}(-2)$. The Hilbert
scheme $\sH^{0}_{1}(X)$ contains a component which is non-reduced
at any point if and only if the normal bundle of a general line
$L$ in that component splits as $\OO_{L}(1) \oplus \OO_{L}(-2)$.
In this case, the threefold $X$ is said to be {\em exotic} (see
\cite{prokhorov:exotic}). On the other hand, we say that $X$ is
{\em ordinary} if it contains a line $L$ with normal bundle
$\OO_{L} \oplus \OO_{L}(-1)$, equivalently if $\sH^{0}_{1}(X)$ has
a generically smooth component. Recall that, if $X$ is general
enough, $\sH^0_1(X)$ is in fact a smooth irreducible curve see
\cite[Theorem 4.2.7]{fano-encyclo}, and the references therein.

Let us recall that a non-hyperelliptic prime Fano threefold $X$ is
exotic if and only if it contains infinitely many non-reduced
conics (see \cite{brambilla-faenzi:ACM}). For $g\geq 9$, the
results of \cite{gruson-laytimi-nagaraj} and
\cite{prokhorov:exotic} imply that $X$ is non-exotic unless $g=12$
and $X$ is the Mukai-Umemura threefold, see \cite{mukai-umemura}.
In fact, the only other known examples of exotic prime Fano
threefolds besides Mukai-Umemura's case are those containing a
cone. For instance if $X$ is the Fermat quartic threefold in
$\p^4$ ($g=3$), then $\sH^0_1(X)$ is a curve with $40$ irreducible
components, each of multiplicity $2$ (see \cite{tennison}). We do
not know if there exist exotic prime Fano threefolds of genus $7$.
In view of a result of Iliev-Markushevich (restated in Proposition
\ref{OL} further on), this amounts to ask whether there are
non-tetragonal smooth curves $\Gamma$ of genus $7$ admitting
infinitely many divisors $\cL$ of type $g^1_5$ such that
$K_\Gamma-2\cL$ is effective (see Remark \ref{marc}).

Remark that the cohomology groups $\HH^{k,k}(X)$ of a prime Fano
threefold $X$ of genus $g$ are generated by the
divisor class $H_X$ (for $k=1$), the class $L_X$ of a line contained
in $X$ (for $k=2$), the class $P_X$ of a closed point of $X$ (for $k=3$).
Hence we will denote the Chern classes of a sheaf on $X$ by the integral
multiple of the corresponding generator. Recall that $H_X^2=(2\,g-2)L_X$.
Given a smooth curve $C \subset X$ of degree $d$ and genus $p_a$, we have:
\begin{equation}\label{chern-retta}
c_1(\OO_C)=0, \qquad c_2(\OO_C)=-d, \qquad c_3(\OO_C)=2-2 p_a - d.
\end{equation}

Applying the theorem of Riemann-Roch to a sheaf $F$ on $X$, of
(generic) rank $r$ and with Chern classes $c_1,c_2,c_3$,
we obtain the following formulas:
\begin{align}
\label{RiemRoch}\chi(F) & = r + \frac{11+g}{6}\,c_1 +\frac{g-1}{2}\, c_1^2
-\frac{1}{2} c_2 +  \frac{g-1}{3}\,c_1^3
-\frac{1}{2}\,c_1\,c_2+\frac{1}{2}\,c_3, \\
\chi(F,F) & = r^2 - \frac{1}{2} \Delta(F),\label{doubleRiemRoch}
\intertext{and, in case $r=2$ and $g=7$, formula \eqref{RiemRoch} becomes:}
\label{RRg7}
\chi(F) & = 2 + 3\,c_1 +3\, c_1^2
-\frac{1}{2} c_2 +  2\,c_1^3
-\frac{1}{2}\,c_1\,c_2+\frac{1}{2}\,c_3.
\end{align}

Recall that if $T\neq 0$ is a torsion sheaf supported in codimension $p>0$,
then $c_k(T)=0$ for $k<p$, while $c_p(T)$ is the class of the
scheme-theoretic support
of $T$ in $\HH^{p,p}(X)$ (see e.g. \cite{fulton:intersection}).
Moreover since $\chi(T(t))$ is positive for $t \gg 0$,
looking at the dominant term of $\chi(T(t))$,
we see that $(-1)^{p-1}c_p(T)>0$.

Recall also that a smooth projective surface $S$ is a {\it K3 surface} if
it has trivial canonical bundle and irregularity zero.
A general hyperplane section of a non-hyperelliptic prime
Fano threefold of genus $g$ is a K3 surface $S$ whose Picard group is
generated by the restriction $H_S$ of $H_X$ to $S$, and whose
(sectional) genus equals $g$.
We consider stability with respect to $H_S$.
Given a stable sheaf $F$ of rank $r$ on a K3 surface $S$ with Chern classes $c_1,c_2$,
the dimension at $F$ of the moduli space $\Mo_S(r,c_1,c_2)$ is:
\begin{equation}
  \label{eq:dimension}
  \Delta(F) - 2\, (r^2-1).
\end{equation}
For this equality we refer for instance to \cite[Part II, Chapter 6]{huybrechts-lehn:moduli}.

\begin{rmk} \label{ext-rational}
Assume that $X$ is a prime Fano threefold, and let $L$ be a
line contained in $X$, with $N_L\cong \OO_L\oplus\OO_L(-1)$. Then we have:
\[
 \ext^1_X(\OO_L,\OO_L) = 1, \quad \ext^2_X(\OO_L,\OO_L) = 0.
\]
One can easily check this statement, using
\eqref{ext-norm} and \eqref{spectr-ext}.
\end{rmk}

\begin{rmk}\label{fibrato-retta}
Let $X$ be a prime Fano threefold, and $L$  a
line contained in $X$.
Then by the well-known Hartshorne-Serre correspondence
(for instance, by an adaptation of \cite[Theorem 4.1]{hartshorne:stable-reflexive} to our setup) we can associate to
$L$ a rank $2$ vector bundle $F_L$, with $c_1(F_L)=-1$ and
$c_2(F_L)=1$ (see also \cite{madonna:fano-cy}).
Moreover we have the following exact sequence:
\begin{equation} \label{eq-fibrato-retta}
0\to \OO_X\to F_L \to \cI_L(-1) \to 0.
\end{equation}
\end{rmk}

\subsection{Geometry of Fano threefolds of genus 7}

We recall here the construction of a Fano threefold of genus
$7$ as a
section of the spinor $10$-fold, outlined by Mukai in
\cite{mukai:curves-K3}, \cite{mukai:biregular}.
See also \cite{mukai:curves-symmetric},
\cite{mukai:curves-symmetric-I},
\cite{iliev-markushevich:genus-7}.

Let $V$ be a $10$-dimensional $\C$-vector space, equipped
with a
non-degenerate quadratic form. The algebraic group $\Spin(V)$
corresponds to a Dynkin diagram of type $D_{5}$.
It admits
two $16$-dimensional irreducible representations $\s^{+}$
and $\s^{-}$,
called the half-spin representations, having maximal weight
respectively $\lambda_{+}=\lambda_{4}$ and
$\lambda_{-}=\lambda_{5}$, where the $4$th and $5$th nodes of $D_{5}$
are connected only to the unique trivalent node.
These representations are naturally dual to each other.

The corresponding roots $\alpha_{+}=\alpha_{4}$ and
$\alpha_{-}=\alpha_{5}$ give rise to the Hermitian symmetric
spaces
$\Sigma^+$ and $\Sigma^{-}$, defined
by $\Sigma^{\pm} = \Spin(10)/{\sf P}(\alpha_{\pm})$, where
${\sf P}(\alpha_{+})$ and ${\sf P}(\alpha_{-})$ are the parabolic subgroups
associated respectively to $\alpha_+$ and $\alpha_-$.
These can be seen as the connected
components of the orthogonal Grassmann variety
$\G_Q(\p^4,\p(V))$ of $4$-dimensional
isotropic linear subspaces $\p^4$ contained in the smooth
quadric hypersurface $Q$
in $\p^{9} = \p(V)$ corresponding to the quadratic form on
$V$.
Note that $\OO_{\G_Q(\p^4,\p(V))}(1)|_{\Sigma_\pm}=\OO_{\Sigma_\pm}(2H)$.
We denote by $\UU_{\pm}$ the restriction of the tautological
subbundle
on $\G_Q(\p^{4},\p(V))$ to $\Sigma^{\pm}$.
We have thus the universal exact sequence:
\begin{equation}
  \label{eq:universalU+}
  0 \to \UU_\pm \to V \ts \OO_{\Sigma^\pm} \to \UU_\pm^* \to 0.
\end{equation}
Under the duality on $V$ given by $Q$, the bundle $\UU_\pm$ is
isomorphic to $\UU_\pm^\perp$.
The hyperplane divisors $H_{\Sigma^{\pm}}$ provide natural
equivariant
embeddings of $\Sigma^{\pm}$ into $\p(\s^{\pm})$.
Given a subvariety $Y\subset \Sigma^{\pm}$, we denote by
$H_Y$ the
restriction of $H_{\Sigma^{\pm}}$ to $Y$.

Now choose a $9$-dimensional vector subspace $A$ of
$\s^{+}$,
and consider its ($7$-dimensional) orthogonal space
$B = A^{\perp} \subset \s^{-}$ under the duality $(\s^{+})^*
\cong \s^{-}$.
We define:
\begin{align}
\label{def-X} X = \Sigma^{+} \cap \p(A) \subset \p(\s^{+}),
\\
\label{def-Gamma} \Gamma = \Sigma^{-} \cap \p(B) \subset
\p(\s^{-}).
\end{align}

If the subspace $A$ is general enough,
then $X$ is smooth and it turns out that it is a prime Fano threefold
of genus $7$, in particular we have $K_X=-H_X$, $H_X^3=12$.
Further, any such prime Fano threefold of genus $7$
is obtained in this way.
In turn, the curve $\Gamma$ is a smooth canonical curve of
genus $7$, called the {\em homologically projective dual curve} of
$X$.
By \cite[Table 1]{mukai:curves-symmetric-I}, we know that
the curve $\Gamma$ is not trigonal nor tetragonal and $W^2_{1,6}$ is empty.
Moreover, a general curve of genus $7$ is of this kind.

\subsubsection{Semiorthogonal decomposition of the derived category of
$X$}

Here we briefly sketch the construction due to
Kuznetsov
\cite{kuznetsov:v12}, of the semiorthogonal decomposition of
$\D(X)$.
We consider the product variety $X \times \Gamma$, together
with the two
projections $p \colon X \times \Gamma \rr X$, $q \colon X
\times \Gamma \rr \Gamma$.
The symmetric form on $V$ provides the following natural
exact sequence on
$X \times \Gamma \subset \Sigma^{+} \times \Sigma^{-}$:

\begin{equation} \label{kuz-exact}
0 \rr \EE^* \rr \UU_{-} \rr \UU_{+}^* \xr{\alpha} \EE \rr 0
\end{equation}
(here $\UU_{\pm}$ denotes also the pull-back of $\UU_{\pm}$ to
$X\times \Gamma$).
It turns out that $\EE$ is a locally free
sheaf on $X \times \Gamma$ with the following invariants:
\begin{align}
& c_1(\EE) = H_X + H_{\Gamma}, \label{dalmazia}\\
& c_2(\EE) = \frac{7}{12} \, H_X \, H_{\Gamma} + 5\,L +
\eta,
\end{align}
where $\eta$ sits in $\HH^{3}(X,\C)\ts \HH^{1}(\Gamma,\C)$ and
satisfies $\eta^2 = 14$.
In view of the results of
\cite{mukai:brill-noether},
\cite{mukai:vector-bundle-brill-noether},
\cite{kuznetsov:v12} and \cite{iliev-markushevich:genus-7},
the vector bundle $\EE$ is a universal object for moduli
functors on $X$
and $\Gamma$ in the sense specified as follows.

\begin{thm}[Mukai, Iliev-Markushevich, Kuznetsov]
\label{summary}
Let $A \subset \s^{+}$ be chosen so that $X$ defined by
\eqref{def-X} is a smooth threefold,
and define $\Gamma$ and $\EE$ as in \eqref{def-Gamma},
\eqref{kuz-exact}. Then:
\begin{enumerate}[i)]
\item \label{moduli-X} the curve $\Gamma$ is isomorphic to
$\Mo_X(2,1,5)$,
\item \label{moduli-Gamma} the manifold $X$ is isomorphic to
the Brill-Noether locus of
      stable bundles $\cE$ on $\Gamma$ with $\rk(\cE)=2$,
      $\det(\cE)\cong H_{\Gamma}$, $\hh^0(\Gamma,\cE)=5$,
\item the bundle $\EE$ universally represents both moduli
problems
  \eqref{moduli-X} and \eqref{moduli-Gamma},
\item  \label{somaro} for all $y\in \Gamma$, the sheaf
$\EE_y$ is a globally generated ACM vector bundle.
\end{enumerate}
\end{thm}

Given points $x \in X$, $y \in \Gamma$, and given a vector
bundle
$\FF$ on $X \times \Gamma$,
we denote by $\FF_{y}$ (resp. $\FF_x$) the
bundle over $X$ (resp. over $\Gamma$) obtained restricting
$\FF$ to
$X \times \{y\}$ (resp. to $\{x\} \times \Gamma$).
We still denote by $\UU_{+}$ (resp. $\UU_{-}$) the
restriction
of $\UU_\pm$ to $X$ (resp. to $\Gamma$).
The vector bundles $\UU_{+}$ and $\UU_{-}$ have rank $5$.
We have $c_1(\UU_-)=-2H_{\Gamma}$ and:
\begin{align*}
& c_1(\UU_+)=-2, && c_2(\UU_+)=24, &&
c_3(\UU_+)=-14.
\end{align*}

We define the following exact functors:
\begin{align}
\label{cocca}
& \ph :  \D(\Gamma) \rr \D(X), && \ph(-) = \RR p_*(q^*(-)
\ts \EE), \\
& \phx : \D(X) \rr \D(\Gamma), && \phx(-) = \RR q_*(p^*(-)
\ts \EE^* (H_{\Gamma}))[1], \\
& \phs : \D(X) \rr \D(\Gamma), && \phs(-) = \RR q_*(p^*(-)
\ts \EE^* (-H_X))[3].
\end{align}
We recall that $\ph$ is fully faithful, $\phs$ is left
adjoint to $\ph$, and $\phx$ is right adjoint to $\ph$.
The main result of \cite{kuznetsov:v12} provides the
following
semiorthogonal decomposition:
\begin{equation} \label{decomposition}
 \D(X) \cong \langle \OO_X, \UU^*_{+}, \ph(\D(\Gamma))
\rangle.
\end{equation}

This decomposition will be used to write a canonical
resolution of a
given sheaf over $X$.
In view of \cite{gorodentsev:exceptional}, given a sheaf $F$
over $X$,
the decomposition \eqref{decomposition} provides a
functorial exact triangle:
\begin{equation}\label{triangolo}
\ph(\phx(F)) \to F \to \mathbf{\Psi} (\mathbf{\Psi^*}(F)),
\end{equation}
where $\mathbf{\Psi}$ is the inclusion of the subcategory
$\langle \OO_X, \UU^*_+ \rangle$ in $\D(X)$ and
$\mathbf{\Psi^*}$ is the left
adjoint functor to $\mathbf{\Psi}$.
The $k$-th term of the complex $\mathbf{\Psi}
(\mathbf{\Psi^*}(F))$ can be
written as follows:
\begin{equation}\label{triangolo-psi}
(\mathbf{\Psi} (\mathbf{\Psi^*}(F)))^k \cong
\Ext_X^{-k}(F,\OO_X)^*\ts \OO_X\oplus
\Ext_X^{1-k}(F,\UU_+)^*\ts \UU_+^*.
\end{equation}

\begin{rmk} \label{evai}
Given a sheaf $F$ on $X$, one can describe more explicitly the map $F \to \mathbf{\Psi}
(\mathbf{\Psi^*}(F))$. We do this here, in order to show that the complex $\mathbf{\Psi}
(\mathbf{\Psi^*}(F))$ is {\em minimal}, i.e. the only non-zero maps in
the complex are from copies of $\OO_X$ to
copies of $\UU_+^*$.
In other words, for any $k$,
the differential $d^k$ from the $(k-1)$-st term to the $k$-th term
is strictly triangular.

We consider the product $X \times X$ and the projections
 $q_{1}$ and $q_{2}$ onto the two factors.
Let $\Delta : X \to X \times X$ be the diagonal embedding.
We denote by $\mathbf{K}$ the following complex on $X\times X$:
$$\UU_+ \boxtimes\UU_+ \to \OO_X \boxtimes \OO_X \to \Delta_* (\OO_X)$$
obtained restricting the standard resolution of the diagonal
on the Grassmannian, see \cite{kapranov:derived}.
Now denote by $\mathbf{U}$ the complex $\OO_X\boxtimes\OO_X [3]
\to \UU_+^* \boxtimes \UU_+^*[3]$.
Since $(\Delta_* (\OO_X))^*\cong \Delta_* (\OO_X(1))[-3]$,  dualizing
$\mathbf{K}$ and shifting by $3$, we get that
$\mathbf{K}^*[3]$ is
quasi-isomorphic to $\Delta_*(\OO_X(1)) \to \mathbf{U}$.
This quasi-isomorphism can be
rewritten as a distinguished triangle
$$ \mathbf{K}^*(-1,0)[3] \to \Delta_*(\OO_X) \to \mathbf{U}(-1,0).$$
Then, given a sheaf $F$ on $X$, the complex $\mathbf{\Psi}
(\mathbf{\Psi^*}(F))$ is given by $\RR q_{2*} (q_{1}^{*}(F(-1))
\ts \mathbf{U})$, and the map $F\to \mathbf{\Psi}
(\mathbf{\Psi^*}(F))$ is induced by
$\Delta_*(\OO_X) \to \mathbf{U}(-1,0)$. Finally, we recall the
vanishing $\Ext^k_X(\UU_+^*,\OO_X)=\Ext^k_X(\OO_X,\UU_+^*)=0$ for
any $k>0$.

Having this in mind, one can easily
prove the minimality statement,
indeed \cite[Lemma 1.6]{kapranov:derived} applies,
and we can use \cite[Lemma 3.2]{ancona-ottaviani:peloritani}
to deduce that the differentials between the graded pieces of
$\RR q_{2*} (q_{1}^{*}(F(-1)) \ts \UU_+^* \boxtimes \UU_+^*)$ are zero,
as well as the differentials between the graded pieces of
$\RR q_{2*} (q_{1}^{*}(F(-1)))$.
\end{rmk}

\begin{rmk} \label{chestrano}
Given an object $F$ of $\D(X)$, we have an exact triangle:
\[
F \to \mathbf{\Psi} (\mathbf{\Psi^*}(F)) \to \ph(\phx(F))[1],
\]
so we may think of $\ph(\phx(-))[1]$ as the {\it right mutation functor}
$R_{\langle\OO_X,\, \UU_+^*\rangle}$ with respect to the subcategory $\langle \OO_X, \UU_+^* \rangle$ of
$\D(X)$, see \cite{gorodentsev:exceptional}.
\end{rmk}

\subsubsection{Some lemmas on universal bundles}

We close this section with some lemmas regarding the image
via the integral functors
defined above of some natural sheaves on $X$ and $\Gamma$.
These results will be needed further on.
From the sequence \eqref{kuz-exact} we obtain:
\begin{align}
\label{GGdual} 0 \rr \EE^* \rr  \UU_{-} \rr \GG \rr 0,\\
\label{GG} 0 \rr \GG \rr  \UU_{+}^* \rr \EE \rr 0,
\end{align}
where $\GG$ is a rank $3$ vector bundle with
$c_{1}(\GG) = H_{X} - H_{\Gamma}$.

\begin{lem} \label{GG-stability}
The vector bundles $\UU_{+}$  and $\GG_y$, for any $y\in
\Gamma$, are stable and ACM.
Moreover, we have $\HH^0(\Gamma,\GG_x)=0$  for any $x\in X$.
\end{lem}
\begin{proof}
Let us prove first that
$\HH^0(\Gamma,\GG_x)=0$  for any $x\in X$.
Notice that $\GG_x\cong\wedge^2\GG^*_x(-1)$, because $\GG_x$
has rank $3$ and
$c_1(\GG_x)\cong -H_\Gamma$.
Let us dualize \eqref{GGdual} and restrict it to
$\{x\}\times \Gamma$.
We obtain an inclusion:
$$\wedge^2\GG_x^*(-1) \mono \wedge^2\UU^*_-(-1).$$
Then we have
$\HH^0(\Gamma,\GG_x)
\subset\HH^0(\Gamma,\wedge^2\UU^*_-(-1))$.
So it suffices to show that the latter space is $0$.
To prove this, one can tensor by $\wedge^2 \UU^*_-(-1)$
the Koszul complex:
\[
0 \to \wedge^{9} A \ts \OO_{\Sigma_{-}}(-9) \to \cdots \to A
\ts
\OO_{\Sigma_{-}}(-1) \to \OO_{\Sigma_{-}} \to \OO_{\Gamma}
\to 0,
\]
and the conclusion follows applying Borel-Bott-Weil theorem
(see e.g.\ \cite{weyman:tract})
on $\Sigma_{-}$ to the
homogeneous vector bundles $\wedge^2 \UU^*_-(-t)$, for
$t=1,\ldots,10$.

Let us now turn to $\UU_+$. Consider the
Koszul complex:
\[
0 \to \wedge^{7} B \ts \OO_{\Sigma_{+}}(-7) \to \cdots \to B
\ts
\OO_{\Sigma_{+}}(-1) \to \OO_{\Sigma_{+}} \to \OO_{X} \to 0,
\]
and tensor it with $\UU_+$.
Applying Borel-Bott-Weil theorem on $\Sigma_{+}$
we obtain that, for any $t$, the
homogeneous vector bundles $\UU_+(t)$ on $\Sigma_+$
have natural cohomology and more precisely we get:
\[
 \HH^k(\Sigma_+,\UU_+(-t))=0, \qquad \mbox{for} \qquad
 \left\{
 \begin{array}{l}
   \mbox{all $k$ and $t=0,\ldots,7$,} \\
   \mbox{$k \neq 0$ and $t<0$,} \\
   \mbox{$k \neq 10$ and $t>7$.}
 \end{array}
 \right.
\]
Then it easily follows that $\UU_+$ is an ACM bundle on $X$
and
\[
 \HH^k(X,\UU_+)=0, \quad \mbox{for all $k$.}
\]
Applying the same argument to $\wedge^2 \UU_+$,
we obtain the following:
\[
 \HH^k(X,\wedge^2 \UU_+)=0, \quad \mbox{for $k\neq 1$, and}
\quad \hh^1(X,\wedge^2 \UU_+)=1.
\]
In particular, Serre duality implies:
\begin{equation}\label{formuletta}
\HH^0(X,\wedge^4 \UU_+(1)) =0, \qquad \HH^0(X,\wedge^3
\UU_+(1)) = 0.
\end{equation}
By Hoppe's criterion, see \cite[Theorem 1.2]{ancona-ottaviani:special} and
\cite[Lemma 2.6]{hoppe:rang-4},
this proves stability of $\UU_+$.

Recall that the dual of an ACM vector bundle is also ACM.
Therefore, the dual bundles of $\UU_+$ and $\EE_y$ are ACM
by \eqref{somaro} of Theorem \ref{summary}.
This easily implies, by \eqref{GGdual} and \eqref{GG}, that the bundle
$\GG_y$ is ACM.
To prove that $\GG_y$ is stable, by Hoppe's criterion it is
enough to show that
the groups $\HH^0(X,\GG_y^*)$, $\HH^0(X,\GG_y(-1))$ both
vanish.
We consider the restriction to
$X \times \{y\}$ of \eqref{GG}.
Since $\UU_+^*(-1)\cong \wedge^4 \UU_+(1)$,
we obtain the latter vanishing by \eqref{formuletta}.
Dualizing the same exact sequence
and using $\HH^{1}(X,\EE^{*}_{y})=0$ (recall that $\EE^{*}_{y}$ is ACM)
we get the former.
\end{proof}

\begin{lem}
\label{grottendic}
Given an object $\cF$ in $\D(\Gamma)$ and an object $F$ in $\D(X)$ we
have the
following functorial isomorphisms:
\begin{align}
\label{gro-X}
& \RR \HHom_X(\ph(\cF),\OO_X) \cong
\ph(\RR\HHom_\Gamma(\cF,\OO_\Gamma))\otimes \OO_X(-1) [1],
\\
\label{gro-gamma}&
\RR \HHom_\Gamma(\phx(F),\OO_\Gamma) \cong
\phx(\RR\HHom_X(F,\OO_X)) \otimes \omega_\Gamma^* [1].
\end{align}
\end{lem}

\begin{proof}
By Grothendieck duality, (see
\cite[Chapter III]{hartshorne:residues-duality}, or
\cite{conrad:grothendieck-duality}),
given a complex $\sK$ on $X \times  \Gamma$, we have:
   \begin{align}
     \label{eq:duality-X}
     &   \RR\HHom_{X}(\RR p_{*}(\sK),\OO_{X}) \cong
     \RR p_{*}  (\omega_\Gamma \ts \RR\HHom_{X\times
       \Gamma}(\sK, \OO_{X \times \Gamma}))[1],\\
     \label{eq:duality-gamma}
     &     \RR\HHom_{\Gamma}(\RR q_{*}(\sK),\OO_{\Gamma}) \cong
     \RR q_{*} (\omega_{X} \ts \RR\HHom_{X\times
       \Gamma}(\sK, \OO_{X \times \Gamma}))[3],
   \end{align}
and the isomorphisms are functorial.
Recall that $\omega_X\cong \OO_X(-1)$ and $\omega_\Gamma\cong\OO_\Gamma(H_\Gamma)$.
So by \eqref{dalmazia} we have $\EE^*\ts \omega_\Gamma\cong \EE\ts\OO_X(-1)$.
Then, setting $\sK = q^{*}(\cF) \ts \EE$ in \eqref{eq:duality-X},
we get \eqref{gro-X}.
Setting $\sK = p^{*}(F) \ts \EE^* \ts \omega_\Gamma[1]$ in
\eqref{eq:duality-gamma}, we obtain \eqref{gro-gamma}.
\end{proof}

\begin{rmk}\label{rango-grado}
Given an object $F$ in $\D(X)$, whose class in the Grothendieck group
has rank $r$ and Chern classes $c_1,c_2,c_3$, let us
assume that $\phx(F)$ is concentrated in degree $p$,
i.e. $\HHH^k(\phx(F))=0$ for $k\neq p$, so that $\phx(F)[p]$ is a sheaf.
Then the rank of $\phx(F)[p]$ is the dimension of the vector space $\HHH^p(\phx(\f))_y$,
for 
a general point $y\in\Gamma$, which
coincides with $(-1)^{p+1} \chi(\EE_y ,F)$, by definition of $\phx$.
Note that, if this rank does not depend
on the point $y \in \Gamma$, then the sheaf $\phx(\f)$ is locally
free over $\Gamma$.
Moreover applying Riemann-Roch formula we can conclude that:
\begin{equation}\label{rango}
\rk(\phx(F)[p]) = (-1)^{p}\left(-c_1+c_1c_2-4c_1^3-c_3\right).
\end{equation}
On the other hand, Grothendieck-Riemann-Roch formula gives:
\begin{equation}\label{grado}
\deg(\phx(F)[p]) =(-1)^p\left( -6c_1+6c_1^2-c_2-24c_1^2+6c_1c_2-6c_3\right).
\end{equation}
\end{rmk}

\begin{lem}\label{quandofiniremo}
The following relations hold on $\Gamma$, for each point $y
\in \Gamma$:
\begin{align}\label{fistar}
& \phs(\OO_X) \cong \UU_-,  && \phs(\UU^*_{+}) \cong
\OO_{\Gamma}, &&\phs(\EE_y) \cong \OO_y,  \\
& \phx(\OO_X) = 0, &&  \phx(\UU_{+}^{*}) =0, \quad &&
\phx(\EE_y)\cong \OO_y,\label{fiics}
\end{align}
and on $X$:
\begin{align}
\label{phO} & \HHH^0(\ph(\OO_{\Gamma}))\cong\UU^*_{+}, &&
\quad \HHH^1(\ph(\OO_\Gamma)) \cong \UU_{+}(1), \\
\nonumber & \HHH^k(\ph(\OO_\Gamma))=0, &&\quad\mbox{for $k\neq0,1$}, \\
\label{phOy}& \ph(\OO_y) \cong \EE_y.&
\end{align}
\end{lem}
\begin{proof}
The isomorphism \eqref{phOy} follows immediately from the
definition of $\ph$.
Since the functor $\ph$ is fully faithful we easily obtain also the
relations
$\phs(\EE_y) \cong \phx(\EE_y) \cong \OO_y$.
It is clear that $\phx(\OO_{X})=\phx(\UU_{+}^{*})=0$.

The isomorphism $\phs(\UU^*_{+}) \cong \OO_{\Gamma}$ is
proved in \cite[Lemma 5.6]{kuznetsov:v12}.
Twisting \eqref{GGdual} by $\OO_{X\times \Gamma}(-H_X)$ and
taking $\RR q_*$, we get $\phs(\OO_X) \cong \UU_-$.
Indeed, we have $\HH^k(X,\GG_y(-H_X))=0$ for any integer $k$,
since the vanishing for $k=1,2$ follows from the fact that $\GG_y$ is ACM
(by Lemma \ref{GG-stability}), and the vanishing for $k=0,3$ follows
from the fact that $\GG_y$ is stable (again by Lemma \ref{GG-stability}).

Given $x\in X$, we restrict \eqref{GG} to $\{x\}\times
\Gamma$ and
taking global sections we get $(\UU^*_{+})_x \subset
\HH^0(\Gamma,\EE_x)$, by Lemma \ref{GG-stability}. Notice
that
$\dim (\UU^*_{+})_x=5$ and by Brill-Noether theory we know
that the
bundle $\EE_x$ cannot have more than $5$ sections, see
\cite{bertram-feinberg}.
Hence $\UU^*_{+,x}\cong \HH^0(\Gamma,\EE_x)$ for all $x\in
X$.
We obtain the isomorphism
$\HHH^0(\ph(\OO_{\Gamma}))\cong\UU^*_{+}$.
By \eqref{gro-X} we get:
$$\HHH^1(\ph(\OO_\Gamma))^*\cong \HHH^0(\ph(\OO_\Gamma)) \ts
\OO_{X} (-1).$$

This gives the isomorphism $\HHH^1(\ph(\OO_\Gamma)) \cong
\UU_{+}(1)$.
\end{proof}

The following corollary of Lemma \eqref{grottendic} has been pointed
out by the referee. We set $\tau$ for the functor $\cF \mapsto
\RR\HHom_\Gamma(\cF,\omega_\Gamma)$ defined on $\D(\Gamma)$.

\begin{corol} \label{coin}
Set $T$ for the functor $F \mapsto \ph(\phx(\RR\HHom_X(F,\OO_X)))[1]$.
Then $T$ is an autoequivalence of $^\perp\langle\OO_X,\UU^*_+\rangle$.
Moreover, we have $\phx \circ T = \tau \circ \phx$.
\end{corol}

\begin{proof}
By Remark \ref{chestrano} one has $\ph(\phx(F^*))[1]=R_{\langle\OO_X,\, \UU_+^*\rangle}(F^*)$.
It is easy to see that the right mutation functor
$R_{\langle\OO_X,\, \UU_+^*\rangle}$ is an equivalence of
$\langle\OO_X, \UU_+^*\rangle^\perp$ onto $^\perp\langle\OO_X, \UU_+^*\rangle$.
Further, by sending $F$ to $\RR\HHom_{X}(F,\OO_X)$, we clearly get an
equivalence of $^\perp\langle\OO_X,\UU^*_+\rangle$
onto $\langle\UU_+,\OO_X\rangle^\perp=\langle\OO_X,\UU^*_+\rangle^\perp$.
It follows that $T$ is an autoequivalence of $^\perp\langle\OO_X,\UU^*_+\rangle$.
To prove $\phx \circ T = \tau \circ \phx$, it suffices to use \eqref{gro-gamma}.
\end{proof}

Denoting by $L_{\langle\UU_+,\OO_X\rangle}$ the left mutation
with respect to  $\langle \UU_+,\OO_X \rangle$ of
$\D(X)$, we have:
\begin{equation}\label{referee}
R_{\langle\OO_X,\, \UU_+^*\rangle}(\RR\HHom_{X}(F,\OO_X))\cong \RR\HHom_{X}(L_{\langle\UU_+,\OO_X\rangle}(F),\OO_X).
\end{equation}


\section{Rank 2 stable sheaves on prime Fano threefolds}

\label{sec:3}

Throughout this section, $Y$ will denote a smooth non-hyperelliptic
complex prime Fano threefold.
In this section we present some results concerning rank $2$
stable sheaves $F$ with $c_1(F)=1$ on $Y$.
We will first analyze the cases of minimal $c_2$ (see the next
subsection) and then look for bundles with higher $c_2$.

\subsection{Rank $2$ stable sheaves with $c_1=1$ and minimal $c_2$}

We provide a lower bound on $c_2(F)$ for the existence of $F$,
namely $\Mo_Y(2,1,c_2)$ is non-empty if and only if $c_2(F)\geq
m_g=\lceil \frac{g+2}{2}\rceil$. Then we give some properties
of $F$ in the cases $c_2=m_g$ and $c_2=m_g+1$,
see Proposition \ref{agglomerato}. This description is
deeply inspired on the analysis of the case $g=8$ pursued by Iliev
and Manivel in \cite{iliev-manivel:genus-8}.
Finally, we restate a result concerning non-emptyness and generic smoothness of
this space (Theorem \ref{riassunto-m}).

\begin{lem} \label{K3}
Let $Y$ be a smooth non-hyperelliptic Fano threefold of genus $g$,
and let $\f$ be a rank $2$ stable sheaf on $Y$ with
$c_1(\f)=H_Y$.
Then we have:
\begin{equation} \label{minimum-c2}
c_2(\f)\geq \frac{g+2}{2}.
\end{equation}
\end{lem}
\begin{proof}
Let $S\subset Y$ be a general hyperplane section surface.
Since $Y$ is non-hyperelliptic,
by Moishezon's theorem \cite{moishezon:algebraic-homology},
we have $\Pic(S)\cong \Z = \langle H_S \rangle$.
Consider the restriction $\f_S=\f\ts\OO_S$ and notice that
the sheaf $\f_S$ is still torsion-free, since $S$ is general.
Moreover it is semistable
by Maruyama's theorem (\cite{maruyama:boundedness}), hence
stable since $c_1(\f_S)=H_S$ and $\Pic(S)=\langle
H_S \rangle$.
Since $S$ is a K3 surface, the dimension of the moduli space
$\Mo_S(2,1,c_2(\f_S))$ can be computed by \eqref{eq:dimension}
and \eqref{Delta}
 and it is $4\,c_2(\f_S)-2\,g-4$.
So this number has to be non-negative, and we obtain
\eqref{minimum-c2}.
\end{proof}

In view of the previous lemma we define:
\begin{equation} \label{mg}
m_g=\left\lceil{\frac{g+2}{2}}\right\rceil.
\end{equation}

\begin{lem}\label{lem:JC}
 Let $C$ be a curve in $\sH^0_d(Y)$, with $d<m_g$. Then $C$
is Cohen-Macaulay and we
 have $\HH^k(Y,\cI_C)=0$ for all $k$.
\end{lem}
\begin{proof}
 First observe that the curve $C$ has no isolated or
embedded
  points. Indeed, the purely $1$-dimensional piece
$\tilde{C}$ of $C$
  is a curve of degree $d$ and arithmetic genus $\ell$,
where $\ell$
  is the length of the zero-dimensional piece of $C$.
  In order to see that, for $\ell > 0$, this leads to a
contradiction,
  one notes that since $\HH^0(Y,\cI_{\tilde{C}})=0$, we have
  $\hh^2(Y,\cI_{\tilde{C}}) \geq \chi(\cI_{\tilde{C}}) = \ell$.
  Thus we would have a non-zero element of
  $\Ext^1_Y(\cI_C(1),\OO_Y)$, corresponding to a rank $2$ sheaf
  $F$ with $c_1(F)=1$, $c_2(F)=d$.
  It is easy to see that the sheaf $F$ would be stable.
  Indeed, assuming that
  there exists a destabilizing torsion-free subsheaf $K$, then it is easy to check
  that $\rk(K)=1$ and $c_1(K)=1$ and we have the following commutative diagram:
  \begin{equation}\nonumber
    \xymatrix@-2ex{
      &  & 0 \ar[d] & 0 \ar[d] \\
      &  & K \ar@{=}[r] \ar[d] & K \ar[d] \\
      0 \ar[r] & \OO_Y \ar[r]\ar@{=}[d] & F \ar[r]\ar[d] & \cI_C(1)\ar[r]\ar[d]&0\\
      0 \ar[r] & \OO_Y \ar[r] & S \ar[r]\ar[d] & T\ar[r]\ar[d]&0\\
      &  & 0  & 0  \\
    }
  \end{equation}
where $T$ has rank $0$ and $c_1(T)=0$. This implies that $T$ is supported at a subvariety
$Z\subset Y$ of dimension less than or equal to $1$.
It follows that $\Ext^1_Y(T,\OO_Y)\cong \HH^2(Z,T(-1))^*=0$.
Note that, by the above diagram, the element in
$\Ext^1_Y(\cI_C(1),\OO_Y)$ corresponding to $F$ (i.e. to the middle
row) is the image of the element in $\Ext^1_Y(T,\OO_Y)$ corresponding
to $S$ (the bottom row). But we have seen $\Ext^1_Y(T,\OO_Y)=0$,
so the middle row of the above diagram splits, a contradiction.

Hence the sheaf $F$ is stable, thus contradicting Lemma \ref{K3}.
The above argument implies
that the group $\HH^2(Y,\cI_C)$ vanishes. Note that this
implies the statement by Riemann-Roch.
\end{proof}

\begin{lem}\label{zeta}
 Let $S$ be a K3 surface of Picard number $1$ and sectional genus $g$,
and let $m=m_g$ be defined by \eqref{mg}.
Then for any $k\geq 1$ and for any $\ell\leq m+k-2$,
$S$ contains no zero-dimensional subscheme $Z$ of length
$\ell$ with $\hh^1(S,\cI_Z(1))=k$.
Moreover if $g\ge4$, then for any $k\geq 2$ and any $\ell\leq m+k-1$,
$S$ contains no zero-dimensional subscheme $Z$
of length $\ell$ with $\hh^1(S,\cI_Z(1))=k$.
\end{lem}

\begin{proof}
We split the induction argument in two steps.
\begin{steppe}
{\it For any  $\ell\leq m+k-2$, there are no subschemes $Z$ of $S$ of length
$\ell$ with $\hh^1(S,\cI_Z(1))=k$, for any $k\ge1$.}

We prove our statement by induction on $k\ge 1$.
Assume first $k=1$, and consider a subscheme $Z\subset S$ of length $\ell\ge1$ and with $\hh^1(S,\cI_Z(1))=1$.
By Serre duality we have $\Ext^1_S(\cI_Z(1),\OO_S)\cong\HH^1(S,\cI_Z(1))^*$.
Let $F$ be the sheaf on $S$ defined by the non-trivial extension
\begin{equation}\label{section-zeta}
0\to \OO_S \to \f \to \cI_Z(1) \to 0.
\end{equation}
Notice that $c_1(\f)=1$, $c_2(\f)=\ell$.
We want to prove now that $F$ is stable, for then the dimension
\eqref{eq:dimension} must be non-negative, and this implies
that $\ell\geq m$.
To do it, assume that $Q$ is a destabilizing quotient of $F$.
Hence $Q$ is torsion-free of rank $1$, and
clearly we must have $c_1(Q)\le0$.
If the composition of $\OO_S \mono F$ and $F\to Q$ is non-zero then it
is necessarily surjective, hence $F$ would
contain $\OO_S$ as a direct summand. But this is impossible since the extension is non-trivial.
This implies that $Q$ is a quotient of $\cI_Z(1)$. But this is also impossible.

In particular we have proved that
there exists no scheme $Z$ of length $1$ and with
$\hh^1(S,\cI_Z(1))\geq1$, since $m\ge2$.

Now, assuming the claim for $k\geq 1$, we shall prove it for $k+1$,
namely we shall prove that a
subscheme $Z$ of length $\ell\leq m+k-1$ with $\hh^1(S,\cI_Z(1))=k+1$
cannot exist. Indeed, given such $Z$,  we choose a chain of subschemes $Z_1\subset
\cdots \subset Z_\ell = Z$ with $Z_j$ of length $j$. Note that
$\hh^1(S,\cI_{Z_j}(1))$ equals $\hh^1(S,\cI_{Z_{j+1}}(1))+\varepsilon_j$, with
$\varepsilon_j \in \{0,1\}$. There must be some $j < \ell $ such that
$\varepsilon_j=1$, for $\hh^1(S,\cI_{Z_1}(1))=0$. Let $j_0$ be the
greatest such $j$, and observe that $\hh^1(S,\cI_{Z_{j_0}}(1))=k$.
The existence of $Z_{j_0}$ contradicts the
induction hypothesis.
\end{steppe}

\begin{steppe}\label{2steppe}
{\it We assume now that $g\ge4$ and we prove that there are no
subschemes of $S$ of length $m+k-1$ and $\hh^1(S,\cI_Z(1))=k$, for any $k\ge2$.}

We prove the statement by induction on $k\ge2$. Assume first $k=2$. Suppose that
$Z$ is a subscheme of length $m+1$ and $\hh^1(S,\cI_Z(1))=2$.
Let $\f$ be the rank $3$ sheaf associated to $Z$ by the non-trivial extension:
\begin{equation}\label{rank3}
0\to \OO_S\ts\HH^1(S,\cI_Z(1)) \to \f \to \cI_Z(1) \to 0.
\end{equation}

Note that $\rk(F)=3$, and $c_1(F)=1$, $c_2(F)=m+1$.
We prove now that $F$ is stable.
Let $Q$ be a destabilizing quotient of $F$.
We may assume that $Q$ is semistable.
This implies that $1\le \rk(Q)\le 2$ and $c_1(Q)\le0$.

If $\rk(Q)=1$, then we conclude as in Step 1.
Then we may assume that $\rk(Q)=2$. Consider the kernel $K$ of the projection
$F\to Q$. We have $\rk(K)=1$ and $c_1(K)\ge1$. The map $K \to F$ then gives an injective map
 $K \to \cI_Z(1)$, so that $K \cong \cI_{Z'}(1)$, for some subscheme $Z'$ of $S$ containing $Z$.
In particular we have $c_1(K)=1$ and $c_1(Q)=0$.
We have thus the following diagram:
  \begin{equation}\nonumber
    \xymatrix@-2ex{
      &  & 0 \ar[d] & 0 \ar[d] \\
      &  & K \ar@{=}[r] \ar[d] & K \ar[d] \\
      0 \ar[r] & \OO_S^2 \ar[r]\ar@{=}[d] & F \ar[r]\ar[d] & \cI_Z(1)\ar[r]\ar[d]&0\\
      0\ar[r]& \OO_S^2\ar[r] & Q \ar[r]\ar[d] & T\ar[r]\ar[d]&0\\
      &  & 0  & 0  \\
    }
  \end{equation}
Here $T$ has rank $0$ and $c_1(T)=0$ hence $c_2(T)\leq 0$.
Note that, since $Q$ is semistable, Bogomolov inequality \eqref{eq:bogomolov} gives
$c_2(Q)\geq 0$. But $c_2(Q)=c_2(T)\leq 0$ hence $c_2(Q)=0$ so that
$Q$ is isomorphic to $\OO_S^2$. Therefore $T=0$ and we conclude that
$F$ should contain
$\OO_S^2$ as a direct summand, which is not the case by the choice of
the extension giving $F$.

We have thus proved that $F$ is stable.
But by \eqref{eq:dimension}, the dimension of the moduli space
$\Mo_S(3,1,c_2(\f))$
equals $6\,c_2(\f)-4\,g-12$ and
for $g\geq 4$ this dimension is negative, a contradiction.

Finally by induction on $k\geq2$ one easily proves, as in Step 1, that
there are no subschemes of $S$ of length $\ell$ and
$\hh^1(S,\cI_Z(1))=k$, with $\ell\leq m+k-1$.
\end{steppe}
\end{proof}

\begin{rmk}\label{agostino}
Note that if $F$ is an ACM bundle on $Y$ and $S$ is a hyperplane section
surface, then the restriction $F_S$ is ACM too. In particular if $Y$ has genus $7$,
the bundle $\EE_y$, introduced in the previous section,
is ACM for any $y\in \Mo_Y(2,1,5)$ and its restriction to $S$ is ACM too.
\end{rmk}

The following proposition is inspired on the approach of Iliev and Manivel, see
\cite{iliev-manivel:genus-8}.

\begin{prop}[Iliev-Manivel]  \label{agglomerato}
Let $Y$ be a smooth non-hyperelliptic Fano threefold of genus $g$ and set
$m=m_g$. Let $\f$ be a rank 2 stable sheaf on $Y$, with $c_1(\f)=1$, $c_2(\f)=c_2\in\{m,m+1\}$, $c_3(\f)=c_3\geq 0$.
When $c_2=m+1$, we assume also $g\ge4$.
Then:
\begin{enumerate}[i)]
\item \label{annullarsi}
$\HH^k(Y,F(-1))=0$,  for all  $k \in \Z$, and
$\HH^j(Y,F)=0$,  for all $j\neq 0$;
\item \label{minimum}
if $c_2=m$, then $F$ is locally free. If moreover $g\ge 4$, then $F$ is
globally generated and ACM;
\item \label{minimum+1}
if $c_2=m+1$, 
then $F$ is either locally free,
or there exists an exact sequence:
\begin{equation}\label{doubledual}
0\to \f \to \e \to \OO_L \to 0,
\end{equation}
where $\e$ is a rank $2$ vector bundle with $c_1(E)=1$, $c_2(E)=m$ and $L$ is a
line contained in $Y$.
\end{enumerate}
\end{prop}

\begin{proof}
Note that since $Y$ is non-hyperelliptic,  we have $g\ge3$.

\begin{step} \label{annullino}
{\it We prove \eqref{annullarsi} for $j=2,3$ and $k=0,3$.
In the meantime we show $\HH^0(Y,F) \neq 0$.}

The sheaf $F$ is stable, hence torsion-free.
Assume that $\HH^2(Y,F) \neq 0$.
Then any non-trivial element of
$\HH^2(Y,\f)^*\cong\Ext^1_Y(F,\OO_Y(-1))$ provides an
extension of the form:
\[
0 \to \OO_Y(-1) \to \tilde{\f} \to \f \to 0,
\]
where $\tilde{\f}$ is a rank $3$ sheaf which is easily seen to be semistable.
This sheaf satisfies $c_1(\tilde F)=0$ and
$c_2(\tilde F)=c_2-(2g-2)$. Since either $c_2=m$, or $c_2= m+1$ and $g\ge4$, it follows that
$c_2(\tilde F)<0$,
which contradicts Bogomolov's inequality \eqref{eq:bogomolov}.
We have proved $\HH^2(Y,F) = 0$.

By \eqref{RiemRoch} we compute $\chi(\f)= 3+g-c_2+\frac12 c_3$.
Then $\hh^0(Y,F)>0$, i.e.\ there exists a non-zero global section of $\f$.

By stability we have $\HH^0(Y,F(-1))=0$.
Moreover by Serre duality and stability we have
$\HH^3(Y,F(-1))=\HH^3(Y,F)=0$.
Indeed, $\HH^3(Y,F(-1))$ is dual to $\Hom_Y(F,\OO_Y)$.
This group is zero, for the image of a  nontrivial map $F \to
\OO_Y$ would be a destabilizing quotient of $F$. Similarly we prove that
$\HH^3(Y,F)=0$.
\end{step}

\begin{step} {\it Take double duals}.\label{FET}
Setting $E=F^{**}$, we consider
the double dual exact sequence:
\begin{equation}\label{ddd}
0\to \f \to \e \to T \to 0,
\end{equation}
where $T$ is a torsion sheaf supported in codimension at least $2$, so $c_1(T)=0$ and
$c_2(T) \leq 0$. The sheaf $E$ has rank $2$, $c_1(E)=1$ and $c_2(E)
\le c_2(F)$.

Let us show that $E$ is stable.
Assuming the contrary, we let $K$ be a destabilizing subsheaf of $E$,
and we note that $K$ must have rank $1$ and $c_1(K) \geq 1$.
Let $K'$ be the pull-back of $K$ to $F$.
The support of the image $K''$ of $K$ in $T$ is contained in the support of $T$, which
has codimension at least $2$, so that $c_1(K'')=0$. Thus $c_1(K')\geq
1$ and $K'$ destabilizes $F$, a contradiction.

By Step \ref{annullino}, we have:
$\HH^k(Y,E(-1))=0$, for $k=0,3$ and $\HH^j(Y,E)=0$ for $j=2,3$.
Moreover since $\hh^0(Y,F)>0$, clearly we have $\hh^0(Y,E)> 0$.
\end{step}

\begin{step}\label{Elocfree}
{\it Show that $E$ is locally free, satisfying
$\HH^k(Y,E(-1))=0, \, \forall k$ and $\HH^j(Y,E)=0$ for $j=2,3$.}

Recall that the sheaf $E$ is reflexive, so its singularity locus has codimension at least $3$,
hence $\e_S=\e\ts\OO_S$ is locally free for a general hyperplane section $S$ of $Y$.
Moreover since $E$ is stable, by Maruyama's theorem
the sheaf $E_S$ is $\mu$-semistable, hence stable, for a general $S$.
Fix a hyperplane section $S$, such that $E_S$ is locally free and stable,
and consider the exact sequence:
\begin{equation}
  \label{eq:restriction}
  0 \to E(-1) \to E \to E_S \to 0.
\end{equation}

Since $\hh^0(Y,E(-1))=0$, it follows that $\hh^0(S,E_S)\ge \hh^0(Y,E)>0$.
Let $Z$ be the zero locus of a general section of $\e_S$.
Note that $Z$ has dimension zero and length $c_2(E) \ge m$ (Lemma \ref{K3}), and recall
the exact sequence:
\begin{equation}\label{section-ES}
0\to \OO_S \to \e_S \to \cI_Z(1) \to 0.
\end{equation}

By Serre duality and stability we have
$\HH^2(S,E_S)^*\cong \HH^0(S,E_S^*)=0$, so the induced map
$\HH^1(S,\cI_Z(1))\to\HH^2(S,\OO_S)$ is surjective.
By Lemma \ref{zeta}, since the subscheme $Z$ is zero-dimensional
of length either $m$, or $m+1$ (and in this case $g\ge4$),
then we must have $\hh^1(S,\cI_Z(1))=1$.
Hence from \eqref{section-ES}, using $\HH^1(S,\OO_S)=0$, $\hh^2(S,\OO_S)=1$ and
$\HH^2(S,E_S)=0$, we get $\HH^1(S,\e_S)=0$.
Taking global sections of \eqref{eq:restriction}, since $\HH^2(Y,E)=0$,
we obtain $\HH^2(Y,\e(-1))=0$.
So:
\begin{equation} \label{maveramente}
 \chi(\e(-1)) = -\hh^1(Y,E(-1)) \leq 0.
\end{equation}

On the other hand $E$ is reflexive, so a direct generalization of
\cite[Proposition 2.6]{hartshorne:stable-reflexive} gives:
\[
c_3(E) = \hh^0(Y,\EExt^1_Y(E,\omega_Y)) \geq 0,
\]
and $c_3(E)=0$ if and only if $E$ is locally free.
So, formula \eqref{RiemRoch} yields:
\[\chi(\e(-1))=c_3(\e)/2\geq0.\]
We deduce $c_3(\e)=0$, so $E$ is locally free.
Now from \eqref{maveramente} we obtain $\HH^1(Y,E(-1))=0$.
Note that, using \eqref{eq:restriction}, this implies also $\HH^1(Y,E)=0$.
Statement \eqref{annullarsi} thus holds whenever $F \cong E$.
\end{step}

\begin{step}
{\it Assume that $c_2=m$ and prove that $\f$ is locally free}.

By the previous step it is enough to prove that $F\cong E$.
Note that since $E$ is stable, by Lemma \ref{K3} we have
$c_2(E)\ge m= c_2$ and
so we get $c_2(T)=c_2(E)-c_2\geq 0$. Hence
$c_2(T)$ vanishes.
Thus the sheaf $T$ is supported on a subscheme of
codimension $3$.

Now, since $c_1(T)=c_2(T)=0$, by \eqref{ddd} we have
$c_3(T)=c_3(E)-c_3\geq0$.
By Step \ref{Elocfree} we know that
$c_3(\e)=0$, hence the assumption $c_3\geq 0$ forces $c_3(T)=0$.
We have thus proved that $T=0$, so the sheaf $F$ is isomorphic to $E$, hence it is locally free.
Moreover, since $F \cong E$, by the previous step we have  the vanishings $\HH^2(Y,F(-1))=0$,
$\HH^1(Y,F(-1))=0$, and $\HH^1(Y,F)=0$.
Of course, this completes the proof of \eqref{annullarsi} in the case $c_2=m$.
\end{step}

\begin{step}
{\it Assume $c_2=m$ and $g\ge4$ and show that $\f$ is
globally generated and ACM.}

Following \cite[Proposition 5.4]{iliev-manivel:genus-8} one
reduces to show that
for any point $x\in Y$ and for a general surface $S'$
through $x$, the vector bundle
$F_{S'}=F\ts\OO_{S'}$ is globally generated.
Clearly it is enough to prove that $\cI_{Z}(1)$ is globally generated,
where we denote again by $Z$ the zero locus of a general global
section of $F_{S'}$.
This amounts to show that $Z$ is cut out scheme-theoretically by its
linear span,
in other word that $Z$ cannot be contained in a subscheme
$Z'\subset {S'}$ such that $\len(Z')=m+1$ and $\hh^1(S,\cI_Z(1))=2$.
But no such subscheme exists by Lemma \ref{zeta}, as soon as $g\ge4$.

Finally $F$ is ACM, by Griffith's
theorem, \cite[Theorem 5.52]{shiffman-sommese:vanishing}, since $\f$ is globally generated.
This completes the proof of \eqref{minimum}.
\end{step}

\begin{step}
{\it Assume $c_2=m+1$ and $g\ge4$, and prove \eqref{minimum+1}}.

We have to show \eqref{doubledual} in case $F$ is not locally free,
and still \eqref{annullarsi} has to be shown in this case.
We consider the exact sequence \eqref{ddd}, introduced in Step \ref{FET}.
Since $E$ is locally free, by Step \ref{Elocfree},
it is easy to see that we must have $c_2(E)=m$ and $c_2(T)=-1$.
By statement \eqref{minimum}, $E$ is also globally generated.
Therefore the $1$-dimensional piece of the support of $T$ is a line $L\subset Y$.
Since the hyperplane section $S$ chosen at Step \ref{Elocfree} is general,
we may also assume that $L\cap S=x$, for a point $x\in Y$.

A general global section of $F_S$ (respectively, of $E_S$)
vanishes on a subscheme $Z'\subset S$ (respectively, $Z\subset S$).
The scheme $Z$ has length $m$ and satisfies $\hh^1(S,\cI_Z(1))=1$, and we have:
\begin{equation} \label{ottobrata}
    0 \to \cI_{Z'}(1) \to \cI_Z(1) \to \OO_x \to 0.
\end{equation}

Since the sheaf $E$ is globally generated, $\cI_Z(1)$ is too, hence $Z$ is cut sheaf-theoretically by
hyperplanes. Then the map $\HH^0(S,\cI_{Z'}(1)) \to \HH^0(S,\cI_Z(1))$ induced by \eqref{ottobrata} is not an isomorphism.
We obtain $\hh^1(S,\cI_{Z'}(1))=1$, which easily implies $\HH^1(S,F_S)=0$.
Taking global section of the exact sequence
\begin{equation}
  \label{baasta}
  0 \to F(-1) \to F \to F_S \to 0,
\end{equation}
and recalling that $\HH^2(Y,F) = 0$, we get $\HH^2(Y,F(-1)) = 0$.
Then, following the argument of Step \ref{Elocfree}, we can easily prove that
$c_3(F)=0$, and $\HH^1(Y,F(-1)) = 0$.
By \eqref{baasta} we conclude that $\HH^1(Y,F) = 0$.
This completes the proof of \eqref{annullarsi} for $c_2=m+1$.

Now we mimic a remark of Druel, see \cite{druel:cubic-3-fold}.
Namely, since $\HH^1(Y,\f(-1))=0$, we have
$\HH^0(Y,T(-1))=0$.
It follows that $T$ is a Cohen-Macaulay curve and by a
Hilbert
polynomial computation we obtain
$T\cong \OO_L$ for a given line $L\subset Y$.
This concludes the proof of \eqref{minimum+1}.
\end{step}
\end{proof}

Finally, we reproduce here, for the reader's convenience, a result summarizing the information we have
on the moduli space $\Mo_Y(2,1,m_g)$, taken from \cite[Theorem 3.2]{brambilla-faenzi:ACM}.
For a more accurate description of such space, we refer
to \cite{brambilla-faenzi:ACM} and the references therein.

\begin{thm}
Let $Y$ be a smooth non-hyperelliptic prime Fano
threefold of genus $g$.
Then the space $\Mo_Y(2,1,m_g)$ is not empty and any sheaf $F$
 satisfies:
\begin{align}
 \label{pippo-3}
   & F\ts \OO_L \cong \OO_L \oplus \OO_L(1), && \mbox{for some line $L
     \subset Y$.}
\end{align}

Assume further that $Y$ is ordinary if $g=3$ and that $Y$ is contained in a
smooth quadric if $g=4$. Then the space $\Mo_Y(2,1,m_g)$ contains a sheaf
$F$ satisfying:
\begin{eqnarray}
    \label{pippo-1}
    &\Ext^2_Y(F,F)=0.
  \end{eqnarray}

The space $\Mo_Y(2,1,m_g)$ is equidimensional of dimension $0$ if $g$ is even and $1$ if $g$
is odd.
\label{riassunto-m}
\end{thm}

Recall that the non-emptyness of the space $\Mo_Y(2,1,m_g)$ is derived
from a case by case analysis, going back to
\cite{madonna:quartic} for
$g=3$, \cite{madonna:fano-cy} for $g=4,5$,
\cite{gushel:fano-6} for $g=6$,
\cite{iliev-markushevich:genus-7},
\cite{iliev-markushevich:sing-theta:asian},
\cite{kuznetsov:v12}, for
$g=7$, \cite{gushel:fano-8-I}, \cite{gushel:fano-8-II},
\cite{mukai:biregular} for $g=8$,
\cite{iliev-ranestad} for $g=9$, \cite{mukai:biregular} for
$g=10$,
\cite{kuznetsov:v22} (see also \cite{schreyer:V22},
\cite{faenzi:v22}) for $g=12$.

\subsection{A good component of the moduli space $\Mo_Y(2,1,d)$}
\label{sec:moduli}

Recall that $Y$ denotes a non-hyperelliptic
smooth prime Fano threefold.
We will construct a good (in the sense specified by Theorem \ref{add-line}) component of the space $\Mo_Y(2,1,d)$, and
for this we will need to assume that $Y$ is ordinary.
In particular we will assume that the Hilbert scheme $\sH^0_1(Y)$
has a generically smooth component.
In case $g=4$ we will have to assume that $Y$ is contained in
a smooth quadric in $\p^5$, due to the restriction appearing in the
previous theorem.

We start with the next proposition, which attributes a particular emphasis
to the condition $\HH^1(Y,\f(-1)) = 0$. We think of it as an
evidence that this condition is analogous to the vanishing
$\HH^1(\p^3,E(-2)) = 0$ required for $E$ in $\Mo_{\p^3}(2,0,d)$ to
be an instanton bundle, see \cite{adhm}. Note that the condition
$\HH^1(Y,\f(-1)) =0 $ holds for any $F$ in $\Mo_Y(2,1,d)$, with
$d=m_g$ and $d=m_g+1$, see Proposition \ref{agglomerato}).

\begin{prop}\label{commento}
Let $Y$ be a smooth prime Fano threefold and $d$ an integer.
If a sheaf $F\in \Mo_Y(2,1,d)$ satisfies $\HH^1(Y,\f(-1)) = 0$, then
we have the following vanishings:
\begin{eqnarray}
\label{comm-uno}&\HH^k(Y,\f(-1)) = 0, &\mbox{ for any $k$,}\\
\label{comm-due}&\HH^1(Y,\f(-t)) = 0, &\mbox{for any $t\geq 1$.}
\end{eqnarray}
\end{prop}
\begin{proof}
First let us prove \eqref{comm-uno}.
By stability and Serre duality, we have $\HH^0(Y,\f(-1)) =\HH^3(Y,\f(-1))= 0$.
By \eqref{RiemRoch} it is easy to compute that $\chi(F(-1))=0$, and
this implies the vanishing for $k=2$.

In order to prove \eqref{comm-due}, let us take a general hyperplane section $S$ of $Y$.
Then we have the restriction exact sequence, for any integer $t$,
\begin{equation}
  \label{eq:restrizione}
    0 \to F(-1-t) \to F(-t) \to F_S(-t) \to 0.
\end{equation}
Note that the sheaf $F_S$ is semistable, by Maruyama's theorem.
This implies $\HH^0(Y,F_S(-t))=0$ for any $t\ge1$.
Then, taking cohomology of \eqref{eq:restrizione}, we deduce that
$\HH^1(Y,\f(-t)) = 0$, for any $t\geq 1$.
\end{proof}

Now, we construct inductively a component of
$\Mo_Y(2,1,d)$, for all $d \geq m_g$. This component is
generically smooth of the
expected dimension and its general element $F$ is locally
free and satisfies
$\HH^1(Y,F(-1))=0$.
To do so we note that, given a locally free sheaf $E$ in $\Mo_Y(2,1,d)$ and a
line $L\subset Y$, assuming that $E \ts \OO_L \cong \OO_L \oplus
\OO_L(1)$, by projecting $E$ onto $E\ts \OO_L$ and then to $\OO_L$ we get a sheaf $F$ lying into an
exact sequence of the form:
\begin{equation} \label{EFL}
0 \to F \to E \to \OO_L \to 0.
\end{equation}

\begin{lem} \label{M}
Let $Y$ be a smooth ordinary prime Fano
threefold of genus $g$, let $L\subset Y$ be a line
belonging to a generically smooth irreducible component ${\sf{H}}$
of $\sH^0_1(Y)$ and choose an irreducible component $\Mo$ of $\Mo_Y(2,1,d)$,
with $\dim(\Mo)=n$.
Assume that there is a locally free sheaf $E$ in $\Mo$
such that $E\ts \OO_L \cong \OO_L \oplus \OO_L(1)$.

Then the set of sheaves $F$ fitting in \eqref{EFL} with $E
\in \Mo$ and $L\in {\sf{H}}$
is an irreducible $(n+1)$-dimensional locally closed subset of
$\Mo_Y(2,1,d+1)$.
\end{lem}

\begin{proof}
Let $E$ be a locally free sheaf lying in $\Mo$.
It is easy to see, using \eqref{chern-retta} to compute the Chern classes of $\OO_L$,
that a sheaf $F$ lying into an exact sequence of the
form \eqref{EFL} is a rank 2 stable (non-reflexive)
sheaf with
$c_1(F)=1$, $c_2(F)=d+1$, $c_3(F)=0$.

Note that the condition $E
\ts \OO_L \cong \OO_L \oplus \OO_L(1)$ is equivalent to
$\hom_Y(E,\OO_L)=1$, and also to $\Ext^1_Y(E,\OO_L)=0$.
By assumption and using semicontinuity (see \cite[Satz 3 (i)]{banica-putinar-schumacher}), we see that
this vanishing takes place for general $L \in {\sf{H}}$
and for general $E \in \Mo$.

Observe that a surjective map $E \to \OO_L$ exists if and only
if the previous vanishing holds.
Since $\hom_Y(E,\OO_L)=1$, the map
$\sigma : E \to \OO_L$ is
unique up to a non-zero scalar, so the kernel of $\sigma$ is determined (up to
isomorphism) by $E$ and $L$.

Therefore we have a rational map $\Mo \times {\sf{H}} \to
\Mo_Y(2,1,d+1)$
which associates to the general member of $\Mo \times {\sf{H}}$
the sheaf $F = \ker(\sigma)$, where $\sigma$ generates $\Hom_Y(E,\OO_L)$.
This map is generically injective, since $E$ is recovered as $F^{**}$
and $\OO_L$ as the quotient $F/F^{**}$.
Note that both $\Mo$ and ${\sf{H}}$ are irreducible,
respectively of dimension $n$, and $1$.
Thus the image of the rational map above is irreducible locally closed
of dimension $n+1$.
\end{proof}

\begin{thm} \label{add-line}
Let $Y$ be a smooth ordinary prime Fano threefold of genus $g$,
and if $g=4$ we further assume that $Y$ is contained in a smooth
quadric in $\p^5$, and we let $m=m_g$.
Then we can choose a line $L\subset Y$ with
$N_L\cong \OO_L \oplus \OO_L(-1)$, and a point $x$ of $L$ such that,
for any integer $d \geq m$, there exists a rank $2$ stable
locally free sheaf
 $\f_d$ with $c_1(\f_d)=1$, $c_2(\f_d)=d$, and satisfying:
\begin{align}
\label{smooth} & \Ext_Y^2(\f_d,\f_d) = 0, \\
\label{noH1} & \HH^1(Y,\f_d(-1)) = 0, \\
\label{nosection} & \HH^0(L,\f_d(-2 \, x)) = 0.
\end{align}

The sheaf $F_{d}$ thus belongs to a generically smooth component of
$\Mo_{Y}(2,1,d)$ of dimension:
$$2d-g-2.$$
\end{thm}

\begin{proof}
We work by induction on $d\geq m$.
For $d=m$, it is enough to choose $L$ and $F_d=F$ according to Theorem \ref{riassunto-m}.
Indeed, the only property that we need to check in this case is
\eqref{nosection}, but this is clear
since we may assume that \eqref{pippo-3} holds for a line $L$ such
that $N_L\cong\OO_L\oplus\OO_L(-1)$.

Now we work out the induction process, and we divide it into several steps.

\setcounter{step}{0}
\begin{step}
{\it Construct a sheaf $F_{d}$ in $\Mo_{Y}(2,1,d)$ for all $d$.}
By induction we can choose a rank $2$ locally free sheaf $\f_{d-1}$ with
$c_1(\f_{d-1})=1$ (so that $\f_{d-1}^*\cong \f_{d-1}(-1)$),
$c_2(\f_{d-1})=d-1$, satisfying
$\Ext_Y^2(\f_{d-1},\f_{d-1}) = 0,$
$\HH^1(Y,\f_{d-1}(-1)) = 0$, and
$\HH^0(L,\f_{d-1}(-2 \, x)) = 0$.
From the last vanishing
it easily follows that $\f_{d-1} \ts \OO_L \cong \OO_L
\oplus \OO_L(x)$.
Therefore there exists a (unique up to a non-zero scalar) surjective
morphism $\f_{d-1} \ts \OO_L\to \OO_L$.
Then we get a
projection $\sigma$ as the composition
of surjective morphisms: $\f_{d-1}\to \f_{d-1} \ts \OO_L\to
\OO_L$.
We denote by $\f_d$ the kernel of $\sigma$ and we have the
exact sequence:
\begin{equation} \label{defEd}
0 \rr \f_d \rr \f_{d-1} \overset{\sigma}{\rr} \OO_L \rr 0.
\end{equation}
\end{step}

\begin{step} {\it Prove that $F_d$ lies in $\Mo_{Y}(2,1,d)$ and
satisfies \eqref{smooth}, \eqref{noH1} and \eqref{nosection}.}
\label{step-2}
\end{step}

In view of Lemma \ref{M}, $\f_d$ is a stable torsion-free sheaf with
$c_1(\f_d)=1$ and $c_2(\f_d)=d$. We have $\HH^1(Y,\f_d(-1))=0$
since $\HH^1(Y,\f_{d-1}(-1))=0$ by induction and
$\HH^0(Y,\OO_L(-1))=0$. So \eqref{noH1} holds.

In order to prove \eqref{smooth}, let us apply the functor
$\Hom_Y(\f_d,-)$ to
\eqref{defEd}.
This gives the exact sequence:
\[
\Ext^1_Y(\f_{d},\OO_L) \rr \Ext^2_Y(\f_{d},\f_{d}) \rr
\Ext^2_Y(\f_d,\f_{d-1}).
\]
We will prove that both the first and the last terms of the
above sequence vanish.
To prove the vanishing of the latter, apply
$\Hom_Y(-,\f_{d-1})$ to
the exact sequence \eqref{defEd}. We get the exact sequence:
$$ \Ext^2_Y(\f_{d-1},\f_{d-1}) \rr \Ext^2_Y(\f_{d},\f_{d-1})
\rr \Ext^3_Y(\OO_L,\f_{d-1}).$$
By induction, we have $\Ext^2_Y(\f_{d-1},\f_{d-1})=0$. Serre
duality yields:
\[
\begin{footnotesize}
\Ext^3_Y(\OO_L,\f_{d-1})^* \cong \HH^0(Y,\OO_L\ts
\f_{d-1}^*(-1)) \cong
\HH^0(L,\f_{d-1}(-2x))=0.
\end{footnotesize}
\]
Therefore we obtain $\Ext^2_Y(\f_{d},\f_{d-1})=0$.
To show the vanishing of the group $\Ext^1_Y(\f_{d},\OO_L)$,
we apply the functor $\Hom_Y(-,\OO_L)$ to \eqref{defEd}.
We are left with the exact sequence:
\[
\Ext^1_Y(\f_{d-1},\OO_L) \rr \Ext^1_Y(\f_{d},\OO_L) \rr
\Ext^2_Y(\OO_L,\OO_L).
\]
The rightmost term vanishes by Remark \ref{ext-rational}.
By Serre duality on $L$ we get
$\Ext^1_Y(\f_{d-1},\OO_L)\cong
\HH^1(L,\f_{d-1}^*)\cong\HH^0(L,\f_{d-1}(-2x))^*$.
But this group vanishes by induction. We have thus established \eqref{smooth}.
Note that, since clearly
$\hom_Y(\f_{d},\f_{d}) =1$ and
$\Ext^3_Y(\f_{d},\f_{d})=\Hom_Y(F_d,F_d(-1))^*=0$ by stability,
then by \eqref{doubleRiemRoch} and \eqref{Delta} we compute
$\chi(F_d,F_d)=3+g-2d$, which implies
\begin{equation}
  \label{ladimensione}
  \ext^1_Y(\f_{d},\f_{d})= 2\,d-g-2,
\end{equation}

Now let us prove property \eqref{nosection}.
Tensoring \eqref{defEd} by $\OO_L$ we get the following
exact sequence of sheaves on $L$
\begin{equation} \label{fourterms}
0 \rr \TTor^Y_1(\OO_L,\OO_L) \rr \f_{d}\ts\OO_L \rr
\f_{d-1}\ts\OO_L  \rr \OO_L \rr 0.
\end{equation}
By \eqref{tor-norm} we know that
$\TTor_1^Y(\OO_L,\OO_L)\cong N^*_L \cong
\OO_L\oplus\OO_L(x)$,
by the choice of $L$.
Now we twist \eqref{fourterms} by $\OO_L(-2x)$ and take
global sections.
By induction $\HH^0(L,\f_{d-1}(-2 \, x)) = 0$, so our claim
\eqref{nosection} follows easily. We have proved Step \ref{step-2}.

\begin{step} {\it Flatly deform $F_d$ to a
stable vector bundle $\f$.} \label{step-3}
\end{step}

First note that this will finish the proof.
Indeed, properties \eqref{smooth},
\eqref{noH1}, \eqref{nosection} will hold for $F$ too.
To see this, note that
\eqref{noH1} and \eqref{smooth} will still hold on $\f$ by
semicontinuity (see \cite[Theorem 12.8]{hartshorne:ag} and
\cite[Satz 3 (i)]{banica-putinar-schumacher}). In order to prove
\eqref{nosection}, we notice that
$\HH^0(Y,F_d \ts \OO_L(-2x)) \cong \Ext^2(\OO_L,F_d(-1))$.
Since the sheaves $\OO_L$ form a flat family over the deformation space,
we conclude that \eqref{nosection} holds for $F$ again by
semicontinuity, see \cite{banica-putinar-schumacher}.
Further, note that the equality $\ext^1_Y(\f,\f)= 2\,d-g-2$ implies
here that $F$ lies in a generically smooth component of
$\Mo_{Y}(2,1,d)$ of dimension $2d-g-2$.

It only remains thus to complete Step \ref{step-3}, i.e.
to check that a general deformation $F$ of $F_{d}$
in $\Mo_Y(2,1,d)$ is a locally free sheaf.
In order to show this, we consider the double dual exact sequence:
\begin{equation}
\label{poibasta}
0 \to F \to F^{**} \to T \to 0,
\end{equation}
where $T$ is a torsion sheaf whose support $W$ has dimension
at most
$1$. We would like to prove $T=0$.
It is easy to see that $F^{**}$ is
stable.
This implies $\HH^0(Y,F^{**}(-1))=0$ which in turn gives
$\HH^0(Y,T(-1))=0$, since $\HH^1(X,F(-1))=0$. So $W=\supp(T)$ contains no isolated or embedded
points, i.e. it is a Cohen-Macaulay curve.

We will now argue that the exact sequence \eqref{poibasta} is not
of the form \eqref{defEd}.
Indeed, recall that the variety $\Mo_Y(2,1,d-1)$
is smooth at the point corresponding to the locally free
sheaf $\f_{d-1}$, since $\Ext^2_Y(F_{d-1},F_{d-1})=0$ by induction
hypothesis.
Let $\Mo$ be the component of $\Mo_Y(2,1,d-1)$ containing
$\f_{d-1}$ and ${\sf{H}}$ the component of $\sH^0_1(Y)$ containing
$L$.
The dimension of $\Mo$ equals $\ext^1_Y(F_{d-1},F_{d-1})=
2\,d-g-4$.
So by Lemma \ref{M}, the set of sheaves fitting as kernel of
\eqref{defEd}, with the middle term lying in $\Mo$,
has dimension $2\,d - g-3$.
Now, if for a general element $F$ in an irreducible open neighborhood of
$F_d$, the sequence \eqref{poibasta} was of the form \eqref{defEd},
then the sheaf $F^{**}$ would lie in $\Mo$ (for $F^{**}$
specializes to $F_{d-1}$) and the quotient $T$ would lie in
${\sf{H}}$ (for $T$ specializes to $\OO_L$).
But we have proved that the set of such sheaves $F$ has dimension
$2\,d - g-3$, while $F$ lies in a component of dimension one greater
(see \eqref{ladimensione}).

Finally, let us show that, assuming $T\neq 0$, we are lead to a
contradiction: this will finish the proof.
To do this, we show that the support $W$ of $T$ must have
degree $1$. In fact we prove that $c_2(T) = -1$, and we
only need to show $c_2(T) \ge -1$.
Note first the equality $\chi(T(t))=-\hh^1(Y,T(t))$ for any
negative integer $t$.
Recall that, by \cite[Remark
2.5.1]{hartshorne:stable-reflexive},
we have $\HH^1(Y,F^{**}(t))=0$ for
all $t\ll 0$.
Thus, tensoring \eqref{defEd} by $\OO_Y(t)$ and taking
cohomology, we obtain $\hh^1(Y,T(t))\le \hh^2(Y,F(t))$ for
all $t\ll 0$.
Further, for any integer $t$, we can compute:
\[c_1(T(t))=0, \quad  c_2(T(t))=c_2(T)=d-c_2(F^{**}), \quad
c_3(T(t))=c_3(F^{**})-(2t+1)c_2(T), \]
hence by Riemann-Roch:
\[
\chi(T(t))= - c_2(T)(t+1)+ \frac{c_3(F^{**})}{2}.
\]
Since $F$ is a deformation of $F_d$, semicontinuity gives
$\hh^2(Y,F(t))\le \hh^2(Y,F_d(t))$.
Now, \eqref{defEd} provides
$\hh^2(Y,F_d(t))=\hh^1(Y,\OO_L(t))=-\chi(\OO_L(t))=-t-1$.
Summing up we have, for all $t\ll0$:
\[
- (t+1) c_2(T) + \frac{c_3(F^{**})}{2} \ge t+1.
\]
This implies that $c_2(T)\ge -1 + \frac{c_3(F^{**})}{2(t+1)}$ for all $t \ll 0$,
which implies $c_2(T)\ge -1$.

We have thus proved that $T$ is of the form $\OO_L(a)$, for some
$L \in \sH^0_1(Y)$, and for some integer $a$. Then
$c_3(T)=c_3(F^{**})=1+2\,a$, so $a\geq 0$, see \cite[Proposition
2.6]{hartshorne:stable-reflexive}. On the other hand we have seen
$\HH^0(Y,T(-1))=0$, so $a \leq 0$. But then $T$ should be of the
form $\OO_L$, a contradiction.
\end{proof}

By Theorem \ref{add-line} and Lemma \ref{M} we can pose the
following.

\begin{dfn} \label{dfn:Md}
Choose a component $\sH$ of $\sH^0_1(Y)$ containing a line $L$ such that $N_L\cong \OO_L\oplus\OO_L(-1)$.
Choose a component $\Mo(m_g)$ of the moduli space $\Mo_Y(2,1,m_g)$
containing a sheaf
$F$ satisfying the properties listed in Theorem \ref{riassunto-m}, with respect to some line
contained in $\sH$.
Then, for each $d\geq m_g+1$, we recursively define $\No(d)$ as
the set of non-reflexive sheaves
fitting as kernel in an exact sequence of the form
\eqref{defEd}, with $\f_{d-1} \in \Mo(d-1)$,
and $\Mo(d)$ as the component of the moduli scheme
$\Mo_Y(2,1,d)$
containing $\No(d)$.
In view of Theorem \ref{add-line} the component $\Mo(d)$ is
generically smooth of dimension $2d-g-2$ and contains $\No(d)$ as
an irreducible divisor.
\end{dfn}

\begin{rmk}
Note that the previous definition may depend on several choices.
Indeed, there might exist several different components $\No(d)$ and
the results we prove hold for each of them.
\end{rmk}

\begin{rmk}
Making use of Theorem \ref{add-line}, it is possible to
classify
all ACM bundles of rank $2$ and $c_1=1$ on smooth non-hyperelliptic
ordinary prime Fano threefolds.
We refer to \cite{brambilla-faenzi:ACM} for a complete
investigation.
\end{rmk}

\section{Rational cubics on  Fano threefolds of genus $7$}\label{sezione:cubics}

Let $X$ be a smooth prime Fano threefold of genus $7$, and let
$\Gamma$ be its homologically projectively dual curve.
For $1 \leq d \leq 4$, the subset of $\sH^0_d(X)$ containing
rational normal curves is described by
the results of \cite{iliev-markushevich:sing-theta:asian}.
It is known to have dimension $d$, and to be
isomorphic to $W^1_{1,5}$ for $d=1$,
isomorphic to $\Gamma^{(2)}$ for $d=2$,
and  birational to $\Gamma^{(3)}$ for $d=3$.
The isomorphism $\sH^0_2(X) \cong \Gamma^{(2)}$ was also
proved
by Kuznetsov, making use of the semiorthogonal
decomposition of $\D(X)$.

Here we first rephrase  in our framework the results concerning lines
contained in $X$.
Then, we make more precise the result on cubics,
showing that the Hilbert scheme $\sH^0_3(X)$ is in fact
isomorphic to the
symmetric cube $\Gamma^{(3)}$.

\subsection{Lines on a Fano threefold of genus $7$}

The fact that the Hilbert scheme of lines contained in $X$
is isomorphic to the Brill-Noether locus $W^1_{1,5}$
is due to Iliev-Markushevich,
\cite{iliev-markushevich:sing-theta:asian}.
This reformulation will be used further on.

\begin{prop}[Iliev-Markushevich] \label{OL}
Let $X$ be a smooth prime Fano threefold of genus $7$.
Then we have the following isomorphisms:
\begin{align}
\label{orecchiette}&\sH^0_1(X) \to W^{1}_{1,5}, &L \mapsto
\phx(\OO_L)[-1],\\
\label{burrata}& \sH^0_1(X) \to W^{2}_{1,7}, & L \mapsto
\phx(\OO_L(-1)).
\end{align}
\end{prop}

\begin{proof}
  Let $L \subset X$ be a line contained in $X$.
  For any $y\in \Gamma$, we have $\EE^*_y \ts \OO_L \cong
  \OO_L \oplus \OO_L(-1)$,
  so $\hh^0(X,\EE^*_y \ts \OO_L)=1$.
  Indeed $\EE_y$ is globally generated for each $y$ and has
  degree $1$ on $L$.
  So $\HHH^k(\phx(\OO_L))_{y} = 0$ for all $y\in \Gamma$
  and all $k\neq -1$.
  Hence $\phx(\OO_L)[-1]$ is a line bundle on $\Gamma$, by \eqref{rango} and \eqref{chern-retta},
  and has degree $5$ by Grothendieck-Riemann-Roch \eqref{grado}.
  Now observe that, using \eqref{phO} and the spectral sequence \eqref{ss1}, we get:
  \begin{align*}
    \HH^0(\Gamma,\phx(\OO_L)[-1]) & \cong
    \Hom_X(\ph(\OO_{\Gamma}),\OO_L[-1]) \cong \\
    & \cong \Hom_X(\HHH^1(\ph(\OO_{\Gamma})),\OO_L) \cong \\
    & \cong \HH^0(L,\UU_+^*(-1)).
  \end{align*}
  Since $\UU^*_+$ is globally generated and $c_1(\UU^*_+)=2$, one sees
  easily that this space must have dimension $2$.
  So $\phx(\OO_L)[-1]$ lies in $W^{1}_{1,5}$, and \eqref{orecchiette}
  is well-defined.

  \vspace{0.2cm}
  Let us now define the inverse of \eqref{orecchiette}.
  For any $\cL$ in $W^1_{1,5}$, we set $\cL \mapsto
  \HHH^1(\ph(\cL))$,
  and we see that
  this map is clearly an inverse of \eqref{orecchiette}, as soon as it is
  well-defined.
  Hence we have to prove that $\HHH^1(\ph(\cL))$ is a sheaf of the
  form $\OO_L$, for a line $L \subset X$. Let us accomplish this task.

  Since the curve $\Gamma$ is not tetragonal, it is easy to see that
  $\cL$ is globally generated, i.e. we have:
  \begin{equation}
    \label{2011}
    0 \to \cL^* \to \OO_\Gamma^2 \to \cL \to 0.
  \end{equation}
  Moreover, a general global section of $\cL$ vanishes of $5$ distinct
  points $y_1,\ldots,y_5 \in \Gamma$, hence it gives:
  \[
  0 \to \OO_\Gamma \to \cL \to \bigoplus_{i=1,\ldots,5} \OO_{y_i} \to 0.
  \]
  Set $\sK_\cL=\HHH^0(\ph(\cL))$ and
  $\FF_\cL=\HHH^1(\ph(\cL))$. By \eqref{phO} and \eqref{phOy}, applying $\ph$ to this
  sequence gives:
  \begin{equation}
    \label{capodanno}
  0 \to \UU_+^* \to \sK_\cL \to \bigoplus_{i=1,\ldots,5}
  \EE_{y_i} \xr{\beta} \UU_+(1) \to \FF_\cL \to 0.
  \end{equation}

  Note that $\beta \neq0$. Indeed, if $\beta=0$, then the sheaf $\FF_Z\cong \UU_+(1)$ would have rank $5$.
  On the other hand, by definition of $\ph$ and by Serre duality we compute
  $\rk(\FF_\cL)=\hh^1(\Gamma,\cL\ts\EE_y)=\hh^0(\Gamma,\cL^*\ts\EE_y)$
  for general $y \in \Gamma$.
  Since $\cL^* \ts \EE_y$ is a semistable sheaf of rank $2$ and degree $2$,
  we may apply Clifford's theorem (see e.g. \cite{mercat:petite})
  to conclude that
  $\hh^0(\Gamma,\cL^*\ts\EE_y)\leq 3$ for any $y\in \Gamma$, contradicting $\rk(\FF_\cL)=5$.

  Then we consider the non-zero sheaf $\im(\beta)$.
  By stability of the $\EE_{y_i}$'s and of $\UU_+(1)$, the sheaf $\im(\beta)$ can only have
  slope either $3/5$, or $1/2$.
  If the latter case takes place, we also know that $\im(\beta)$ is isomorphic
  either to $\EE_{y_i}$, or $\EE_{y_i}\oplus \EE_{y_j}$, for some $i\neq j\in\{1,\ldots,5\}$.
  In particular $\beta$ restricted to $\EE_{y_i}$ is a non-zero multiple
  of the identity map, so the same happens to $\phs(\beta)$, restricted
  to $\EE_{y_i}$.
  But, applying the functor $\phs$ to the exact sequence above,
  since $\phs(\UU_+(1))=0$, we get $\phs(\beta)=0$, which is a contradiction.

  Hence we must have $\mu(\im(\beta))=3/5$.
  So $\beta$ is generically surjective, $\sK_\cL$ is a
  torsion-free sheaf of rank $10$ and $\FF_\cL$ is a torsion sheaf.
  Therefore, since $\HHom_X(\FF_\cL,\OO_X)=0$ and
  $\EExt_X^{-1}(\sK_\cL,\OO_X)=0$, using \eqref{gro-X} we also
  get $\HHH^0(\ph(\cL^*))=0$ in view of the spectral sequence \eqref{ss1}.
  Then, applying $\ph$ to \eqref{2011} we get an injection
  $\iota : (\UU_+^*)^2 \subset \sK_\cL$. These two sheaves have the
  same rank and first Chern class, $(\UU_+^*)^2$ is locally free and $\sK_\cL$ is torsion-free.
  It is easy to conclude that $\iota$ is then an isomorphism.
  We compute now by \eqref{capodanno} that
  $\FF_\cL$ has Chern classes $(0,-1,1)$, and again by
  \eqref{capodanno} we see $\HH^0(X,\FF_\cL(-1))=0$.
  This suffices to deduce $\FF_\cL \cong \OO_L$.

  Let us look at \eqref{burrata}. Consider $\cP = \cL^*\ts \omega_\Gamma$,
  which is a line bundle of degree $7$. By
  Serre duality we get
  $\hh^0(\Gamma,\cP)=\hh^1(\Gamma,\cL)=3$,
  hence
  $\cP$ lies in $W^2_{1,7}$.
  Applying \eqref{gro-gamma}, since $\RR
  \HHom_X(\OO_L,\OO_X)[2] \cong
  \OO_L(-1)$, we obtain the functorial isomorphism:
  $$
  \phx(\OO_L(-1)) \cong \cL ^* \ts \omega_\Gamma.
  $$
  Therefore, \eqref{burrata} is also well-defined.
  We have then a commutative diagram:
  \[
  \xymatrix{
   \sH^0_1(X) \ar[r] \ar@{=}[d] & W^{1}_{1,5} \ar^{\cL \mapsto \cL^*
     \ts \omega_\Gamma}[d] && \OO_L \ar@{=}[d] \ar@{|->}[r] & \ph(\OO_L[-1]) \ar@{|->}[d] \\
   \sH^0_1(X) \ar[r] & W^{2}_{1,7} && \OO_L \ar@{|->}[r] & \ph(\OO_L(-1)).
  }
  \]
  Then \eqref{burrata} is an isomorphism, and we are done.
\end{proof}

A few comments are in order. First, observe that $\Ext^k_X(\OO_L,\OO_X)=0$ for all $k$,
while
$\Ext^k_X(\OO_L,\UU_+)=0$ for $k\neq 3$, and
$\Ext^3_X(\OO_L,\UU_+)=A_L^*$, where
$A_L$ denotes $\HH^0(\Gamma,\phx(\OO_L)[-1])$.
Making use of the exact triangle \eqref{triangolo},
this gives the isomorphisms:
\begin{align}\label{rigatoni}
  & \HHH^k(\ph(\phx(\OO_L))) \cong \left\{
    \begin{array}{ll}
      \OO_L & \mbox{for $k=0$,} \\
      A_L \ts \UU_+^* & \mbox{for $k=-1$,} \\
      0 & \mbox{otherwise.}
    \end{array}\right.
\end{align}
  Note that \eqref{capodanno} gives the resolution of $\OO_L$:
  \[
    0 \to \UU_+^* \to \bigoplus_{i=1,\ldots,5}
  \EE_{y_i} \xr{\beta} \UU_+(1) \to \OO_L \to 0.
  \]
  Notice also that, given a line $L \subset X$, taking $\cP =
  \phx(\OO_L(-1))^* \ts \omega_\Gamma$, we have $\cP$ in $W^2_{1,7}$.
  Using \eqref{triangolo}, we can write the canonical resolution of
  $\OO_L(-1)$, that reads:
  \begin{equation}
    \label{OL-1}
    0 \to \OO_X \to (\UU_+^*)^3 \xr{\varsigma} \ph(\cP) \to \OO_L(-1) \to 0.
  \end{equation}
  Note that, considering the evaluation map
  $e_\cP=e_{\OO_\Gamma,\cP} : \OO_\Gamma^3 \to \cP$,
  the map $\varsigma$ above is just $\HHH^0(\ph(e_\cP))$.
  This gives an explicit description of the inverse of \eqref{burrata}
  as $\cP \mapsto \cok(\HHH^0(\ph(e_\cP)))$.

\begin{rmk} \label{marc}
  In view of the isomorphism $\sH^0_1(X) \cong W^{1}_{1,5}$, we
  note that the threefold $X$ is exotic if and only if $W^{1}_{1,5}$
  has a component which is non-reduced at any point.
  It is well-known (see e.g. \cite[Proposition 4.2]{acgh}) that $\cL$ is a singular point of $W^{1}_{1,5}$
  if and only if the Petri map:
  \[
  \pi_\cL : \HH^0(\Gamma,\cL) \ts \HH^0(\Gamma,\cL^* \ts \omega_\Gamma) \to \HH^0(\Gamma,\omega_\Gamma).
  \]
  is not injective.
  We have seen that any line bundle $\cL$
  in $W^{1}_{1,5}$ is globally generated. Therefore $\ker(\pi_\cL)$ is isomorphic to
  $\HH^0(\Gamma,\cL^* \ts \cL^* \ts \omega_\Gamma)$.

  This proves that the threefold $X$ is exotic if and only if $\Gamma$ admits infinitely many
  line bundles $\cL$ in $W^{1}_{1,5}$ such that $\cL^*\ts \cL^* \ts \omega_\Gamma$ is effective.
\end{rmk}

\subsection{Conics on a Fano threefold of genus $7$}

Kuznetsov's result on conics contained in $X$, see \cite{kuznetsov:v12} asserts that, if $C
\subset X$ is a connected curve Cohen-Macaulay of arithmetic genus $0$
(a conic), then $\phx(\OO_C)$ is the structure sheaf of a length-$2$
subscheme of $\Gamma$, and we have the following canonical resolution:
\begin{equation} \label{res-conic}
  0 \to \OO_X \to \UU_+^* \to \ph(\phx(\OO_C)) \to \OO_C \to 0.
\end{equation}

Let us look at what happens for reducible conics.

\begin{lem} \label{gennaio-2011!}
  Let $M,N \subset X$ be distinct lines and set $\cM =
  \phx(\OO_M)[-1]$, $\cP = \phx(\OO_N(-1))$.
  Then $N \cap M \neq \emptyset$ if and only if
  $\HH^0(\Gamma,\cM^* \ts \cP)\neq 0$.
  In this case, setting $C = M \cup N$, we have $\phx(\OO_C) \cong
  \cP/\cM$.
\end{lem}

\begin{proof}
  Recall that $\cM \in W^1_{1,5}$ and $\cP \in W^2_{1,7}$.
  If $M\cap N=\{x\}$, we have the exact sequence:
  \begin{equation}
    \label{incontrano}
    0 \to \OO_{N}(-1) \to \OO_{C} \to \OO_M \to 0.
  \end{equation}
  Applying $\phx$ to the sequence above, by Proposition \ref{OL} we have:
  \[
  0 \to \cM \to \cP \to \phx(\OO_{C})\to 0
  \]
  We deduce that $\HH^0(\Gamma,\cM^* \ts \cP)\neq0$, and $\phx(\OO_C) \cong \cP/\cM$.

  Conversely, assume
  $\HH^0(\Gamma,\cM^* \ts \cP) \neq 0$.
  We get a commutative exact diagram:
  \[
  \xymatrix@-2ex{
  0 \ar[r] & \OO_\Gamma^2 \ar[r] \ar[d] & \OO_\Gamma^3  \ar[r] \ar[d] & \OO_\Gamma\ar[d] \ar[r] & 0\\
  0 \ar[r] & \cM \ar[r] & \cP \ar[r] & \OO_{Z} \ar[r] & 0,
  }
  \]
  where $Z \subset \Gamma$ has length $2$ and the vertical maps are
  natural (surjective) evaluations.
  Applying $\ph$ to this diagram and taking cohomology, using \eqref{res-conic}
  and \eqref{OL-1}, we get:
  \[
  \xymatrix@-2ex{
  0 \ar[r] & (\UU_+^*)^2 \ar[r] \ar@{=}[d] & (\UU_+^*)^3  \ar[r] \ar[d] & \UU_+^* \ar[d] \ar[r] & 0 \ar[d] \\
  0 \ar[r] & (\UU_+^*)^2 \ar[r] \ar[d] & \ph(\cP)\ar[d] \ar[r] & \ph(\OO_{Z})\ar[d] \ar[r] & \OO_L\ar[r] \ar@{=}[d] &0\\
  &0 \ar[r] &\OO_N(-1) \ar[r] & \OO_C \ar[r] & \OO_L \ar[r] & 0
  }
  \]
  This gives back \eqref{incontrano} with $C$ a conic in
  $X$, so $M \cap N = \{x\}$.
\end{proof}

Our next goal is to investigate the Hilbert scheme $\sH^0_3(X)$.
We will need the  following lemma.

\begin{lem} \label{bau}
  Let $C$ be any Cohen-Macaulay curve of degree $d\geq 3$
  and
  arithmetic genus $p_a$ contained
  in $X$. Then $\phx(\OO_C)$ is a vector bundle on $\Gamma$
  of rank $d-2+2p_a$ and degree $7\,d-12+12\,p_a$.
\end{lem}

\begin{proof}
  The following argument is inspired on the proof of
  \cite[Lemma
  5.1]{kuznetsov:v12}.
  We have to prove that, for each $y \in \Gamma$, the group
  $\HH^0(X,\EE^*_y \ts \OO_C)$ vanishes.
  By \eqref{GG}, it is enough to prove:
  \begin{equation}
    \label{gatto}
    \HH^0(C,\UU_+)=0.
  \end{equation}
  Assume the contrary, and consider a non-zero global section
  $u$ in
  $\HH^0(C,\UU_+)$. Let $U$ be the $1$-dimensional subspace
  spanned by
  $u$. By \eqref{eq:universalU+}, we have $U\subset
  \HH^0(C,\UU_+)\subset \HH^0(C,\OO_X\otimes
  V)\cong V$. Set $V'=U^\perp/U$. Then the orthogonal
  Grassmann variety $\G_Q(\p^3,\p(V'))\subset\Sigma_+$ is a
  quadric and clearly the
  curve $C$ is contained in $X':=\G_Q(\p^3,\p(V'))\cap X$.
  Recall that $X$ is a linear section of $\Sigma^+$, then
  $X$ must
  contain either a $2$-dimensional quadric or a plane. But
  this is
  impossible by Lefschetz theorem.

  This proves that $\phx(\OO_C)$ is a vector bundle on
  $\Gamma$.
  By \eqref{rango} and \eqref{chern-retta},
  we conclude that $\rk(\phx(\OO_C))=d-2+2p_a$.
  Finally, by \eqref{grado},
  we compute the degree of $\phx(\OO_C)$.
\end{proof}

\subsection{Hilbert scheme of rational cubics on $X$ and the symmetric cube of $\Gamma$}

We are now in position to prove the following result.

\begin{thm} \label{thm:cubic-iso}
  Let $X$ be a smooth prime Fano threefold of genus $7$, and $\Gamma$
  be its homologically projectively dual curve.
  The map $\psi: \OO_C \mapsto (\phx(\OO_C))^*\ts \omega_\Gamma$
  gives an isomorphism between
  $\sH^0_3(X)$ and $\Gamma^{(3)}$. In particular
  $\sH^0_3(X)$ is a smooth irreducible threefold.
\end{thm}

\begin{proof}
  We have seen in Lemma \ref{bau} that, for any element $C$
  of $\sH^0_3(X)$,
  the sheaf $\cL=\phx(\OO_C)$ is a line bundle of degree $9$ on
  $\Gamma$.

  Let us show that $\cL$ lies in $W^3_{1,9}$. Note first that, for any $x \in X$,
  we have $\hh^1(\Gamma,\cL\otimes \EE_x)=\hh^0(\Gamma,\cL^*\otimes
  \EE_x)$ by Serre duality.
  Hence, by stability of $\EE_x$
  and since $\cL$ has degree $9$,
  we get $\hh^1(\Gamma,\cL\otimes \EE_x)=0$, for all $x \in X$, hence
  $\HHH^{1}(\ph(\cL))=0$.
  So, by decomposition \eqref{triangolo} applied to the sheaf $\OO_C$,
  the cohomology of the complex $(\mathbf{\Psi}
  (\mathbf{\Psi^*}(\OO_C)))$ can appear only in degree $-1$ or $0$.
  By formula \eqref{triangolo-psi}, we have
  $(\mathbf{\Psi} (\mathbf{\Psi^*}(\OO_C)))^{k}=0$
  for $k\neq -1,-2,-3$ and
  \begin{align*}
    &(\mathbf{\Psi} (\mathbf{\Psi^*}(\OO_C)))^{-3}=\OO_X^a, \\
    &(\mathbf{\Psi} (\mathbf{\Psi^*}(\OO_C)))^{-2}=\OO_X^{a+2}\oplus{\UU^*_+}^{b},\\
    &(\mathbf{\Psi} (\mathbf{\Psi^*}(\OO_C)))^{-1}={\UU^*_+}^{b+4},
  \end{align*}
  where $a=\hh^0(C,\OO_X(-1))$ and $b=\hh^0(C,\UU_+^*(-1))$, and clearly,
  by Riemann-Roch formula,
  $a+2=\hh^1(C,\OO_X(-1))$ and $b+4=\hh^1(C,\UU_+^*(-1))$.
  It follows that the cohomology of the complex $(\mathbf{\Psi}(\mathbf{\Psi^*}(\OO_C)))$ is
  concentrated in degree $-1$.
  Moreover since $\HHH^k (\mathbf{\Psi} (\mathbf{\Psi^*}(\OO_C)))=0$ for $k=-3$ and $k=-2$, then
  we have that the differential
  $d^{-2}: \OO_X^a \to \OO_X^{a+2}\oplus{\UU^*_+}^{b}$
  is injective and that $\ker(d^{-1})=\im(d^{-2})$.

  Let us show that $a=b=0$.
  By the minimality of the complex (see Remark \ref{evai}), the
  differentials in the complex $(\mathbf{\Psi} (\mathbf{\Psi^*}(\OO_C)))$
  take the form:
  \[
  d^{-2} =
  \begin{pmatrix}
    0 \\
    d^{-2}_2
  \end{pmatrix},
  \qquad
  d^{-1} =
  \begin{pmatrix}
    d^{-1}_1 & 0
  \end{pmatrix}.
  \]
  Thus $\ker(d^{-1}) = \ker(d^{-1}_1)\oplus{\UU^*_+}^{b}$, and
  $\im(d^{-2}) = \im(d^{-2}_2) \subset {\UU^*_+}^{b}$.
  Since $\ker(d^{-1}) = \im(d^{-2})$, we have $\ker(d^{-1}_1)=0$ and
  $\im(d^{-2}_2) = {\UU^*_+}^{b}$. Then $d^{-2}_2$ is an isomorphism
  of $\OO_X^a$ to ${\UU^*_+}^{b}$, which implies $a=b=0$.

  We have thus:
  \begin{align}
    &\label{pippo2} \hh^1(C,\OO_X(-1))=2, & \HH^0(C,\OO_X(-1))=0.\\
    &\label{pippo} \hh^1(C,\UU_+^*(-1))=4, &
    \HH^0(C,\UU_+^*(-1))=0,
  \end{align}
  and we get the following exact sequence:
  \begin{equation} \label{res-cL}
    0 \to  \OO_X^{2} \xr{d^{-1}} (\UU_+^*)^{4} \xr{\zeta_C} \ph(\cL)
    \to \OO_C \to 0.
  \end{equation}
  Then we have, by \eqref{fistar}:
  \begin{align*}
    \hh^0(\Gamma,\cL) & = \hom_\Gamma(\OO_\Gamma,\cL) = \hom_\Gamma(\phs(\UU_+^*),\cL) = \hom_X(\UU^*_+,\ph(\cL)) = 4,
  \end{align*}
  where the last equality is obtained
  applying the functor
  $\Hom_X(\UU^*_+,-)$ to the exact sequence \eqref{res-cL}
  and
  using the fact that the group $\HH^0(C,\UU_+)$ vanishes
  by \eqref{gatto}.
  We have thus proved that $\cL$ is an element of $W^3_{1,9}$.
  This defines a morphism $\varphi : \sH^0_3(X) \to
  W^3_{1,9}$ as $\varphi(\OO_C)=\phx(\OO_C)$.

  Observe that the correspondence $\tau: \cL \mapsto \cL^*\ts
  \omega_\Gamma$
  provides an isomorphism between $W^3_{1,9}$ and
  $W^0_{1,3}$.
  On the other hand note that $W^0_{1,3}\cong
  \Gamma^{(3)}$, indeed the curve $\Gamma$
  is not trigonal (see \cite[Table 1]{mukai:curves-symmetric-I}).
  We have thus defined a morphism $\psi: \OO_C \mapsto
  (\phx(\OO_C))^*\ts \omega_\Gamma$
  from $\sH^0_3(X)$ to $\Gamma^{(3)}$ as $\psi=\tau\circ\varphi$.

  \vspace{0.2cm}

  Now, we would like to construct an inverse map $\vartheta : W^3_{1,9} \to
  \sH^0_3(X)$ of $\varphi$. We consider a line bundle $\cL$ in
  $W^3_{1,9}$, and the natural
  evaluation map of sections of $\cL$:
  \[
  e_\cL := e_{\OO_\Gamma,\cL}: \HH^0(\Gamma,\cL) \ts \OO_\Gamma \to \cL.
  \]
  Applying the functor $\ph$ to this map and taking cohomology, we
  get a map:
  \[
  \HHH^0(\ph(e_\cL)) : \HH^0(\Gamma,\cL) \ts \UU_+^* \to \ph(\cL).
  \]

  We would like to show that the cokernel of this map is the structure
  sheaf of a rational curve of degree $3$ on $X$. This will define an
  inverse map $\vartheta$ to $\varphi$.
  So let $Z=\{y_1,y_2,y_3\}$ be three, non necessarily distinct
  points of $\Gamma$. We have an exact sequence:
  \[
  0 \to \OO_\Gamma \to \OO_\Gamma(Z) \to \OO_Z \to 0.
  \]

  We set $\FF_Z = \HHH^1(\ph(\OO_\Gamma(Z)))$.
  Applying $\ph$ to the sequence above, by \eqref{phO} and \eqref{phOy},  we get:
  \[
  0 \to \UU_+^* \to \HHH^0(\ph(\OO_\Gamma(Z))) \to \bigoplus_{i=1,2,3}
  \EE_{y_i} \xr{\xi} \UU_+(1) \to \FF_Z \to 0.
  \]

  Reasoning  as in the proof of Proposition \ref{OL}, one can prove that
  $\im(\xi)$ has rank $5$ and first Chern class $3$,
  i.e. $\mu(\im(\xi))=3/5$ (this time we use Mercat's theorem
  \cite[Theorem A-1]{mercat:petite} rather than Clifford's theorem to
  show $\xi \neq 0$).


  So $\xi$ is generically surjective, and $\ker(\xi)$ is reflexive
  (by \cite[Proposition 1.1]{hartshorne:stable-reflexive})
  with rank $1$ and $c_1(\ker(\xi))=0$, i.e. $\ker(\xi) \cong \OO_X$.
  We have thus:
  \[
  0 \to \OO_X \to \bigoplus_{i=1,2,3}
  \EE_{y_i} \xr{\xi} \UU_+(1) \to \FF_Z \to 0.
  \]
  It follows that $\FF_Z$ is a sheaf of rank $0$
  with Chern classes $(0,-3,1)$. Now, twisting by $\OO_X(-1)$ the sequence above and
  taking cohomology, we get $\HH^0(X,\FF_Z(-1))=0$. It follows that $\FF_Z$ is torsion-free over a
  Cohen-Macaulay curve $C_Z$ of arithmetic genus $0$ and degree $3$ in $X$.

  Dualizing the sequence above we have:
  \[
  0 \to \UU_+^*(-1) \to \bigoplus_{i=1,2,3}
  \EE_{y_i}^*  \to \OO_X  \to \EExt^2_X(\FF_Z,\OO_X) \to 0.
  \]
  So we can define our map $\vartheta$ as
  $\vartheta(Z)=\RR\HHom_X(\FF_Z,\OO_X[-2])$.
  Applying Lemma \ref{grottendic}, we see that $\vartheta$ and
   $\varphi$ are inverse to each other.
\end{proof}

\section{Vector bundles on Fano threefolds of
  genus $7$}
\label{sec:moduli-bn} In this section, we assume that $X$ is a
smooth prime Fano threefold of genus $7$, and we let $\Gamma$ be
its homologically projectively dual curve. We set up a birational
correspondence between the component $\Mo(d)$ of Definition
\ref{dfn:Md} and a component of the Brill-Noether variety
$W^{2\,d-11}_{d-5,5\,d-24}$. This correspondence will turn out to
be an isomorphism for $d=6$.
In this section, we will need to
assume that $X$ is ordinary for $d\ge7$, in order to ensure the existence of the
well-behaved component $\Mo(d)$, constructed in Theorem \ref{add-line}.

\subsection{Vanishing results}

In order to setup the correspondence mentioned above,
we will have to prove that various cohomology groups are zero.
This is the purpose of the next series of lemmas.

\begin{lem} \label{primo}
Let $d\geq 6$ and let $\f$ be a sheaf in $\Mo_X(2,1,d)$ such that:
\begin{equation}\label{ipotesona}
\HH^1(X,F(-1))=0,
\end{equation}
then we have, for all $y\in \Gamma$:
\begin{align}
\label{gigetto} &\Ext^{k}_X(\EE_y,\f)=0, && \mbox{for all $k\neq 1$,} \\
\label{gigino} &\Ext^{j}_X(\f,\EE_y)=0, && \mbox{for all $j\geq 2$.}
\end{align}
Moreover, if $\f$ is locally free, then \eqref{gigino} holds for $j=0$
as well.
\end{lem}

\begin{proof}
Let us prove the first vanishing.
Notice that for $k<0$ and for $k>3$ the claim is obvious since $X$ has dimension $3$.
For $k=3$, the claim amounts to $\Hom_X(\f,\EE_y(-1))=\Hom_X(\f,\EE_y^*)=0$,
which follows from stability of $\f$ and $\EE_y$.

For $k=0$, we have to show that $\Hom_X(\EE_y,\f)=0$.
Assume the contrary, and let $f$ be a non-trivial map
$f:\EE_y \to \f$.
Note that the image $I$ of $f$ must have rank $2$, for if it had rank $1$
it would destabilize either $F$ (if $c_1(I) \geq 1$), or $\EE_y$
(if $c_1(I) \leq 0$).
So $f$ is injective, and its cokernel is a torsion sheaf $T$ with
$c_1(T)=0$, hence $c_2(T) \leq 0$. This gives the inequality $5=c_2(\EE_y) \geq
c_2(\f)$, which is in fact an equality by the minimality of
$c_2(\EE_y)$. But we have $c_2(\f)=d\geq 6$, a
contradiction.

For $k=2$, let us show that $\Ext^{1}_X(\f,\EE_y^*)=0$.
Applying the functor $\Hom_X(\f,-)$ to the restriction of
\eqref{GGdual} to $X \times \{y\}$,
we get:
\begin{equation} \label{Ext1F}
\Hom_X(\f,\GG_y) \to \Ext^1_X(\f,\EE^*_y) \to
\Ext^1_X(\f,\OO_X) \ts (\UU_{-})_{y}.
\end{equation}
It is easy to see that the term on the left hand side
vanishes by virtue of stability of $\GG_y$ and $\f$
(see Lemma \ref{GG-stability}).
On the other hand, the rightmost term vanishes by the assumption $\HH^1(X,\f(-1))=0$
and Proposition \ref{commento}.
We have thus proved \eqref{gigetto}.

Let us now turn to \eqref{gigino}.
For $j=3$, one easily sees that this group vanishes by stability of $\f$ and
$\EE_y$.
In order to prove \eqref{gigino} for $j=2$, we apply the
functor $\Hom_X(\f,-)$ to \eqref{GG}, and by stability of
$\GG_y$, we
are reduced to show that $\Ext^2_X(\f,\UU_+^*)=0$. But this
follows
easily applying $\Hom_X(\f,-)$ to \eqref{eq:universalU+}.

Finally, if $F$ is locally free, then
a nonzero map $f\in\Hom_X(F,\EE_y)$ of stable locally free sheaves
$f: F \to \EE_y$ would have to be an isomorphism, for these two bundles have the
same slope. But their second Chern classes are not equal, hence there
cannot be such $f$.
\end{proof}

\begin{lem} \label{secondo}
Let $d\geq 6$ and let $\f$ be a sheaf in $\Mo_X(2,1,d)$ satisfying
\eqref{ipotesona}.
Then we have:
\begin{align}\label{alila}
&\Ext^{k}_X(\f,\OO_X)=0, && \mbox{for any $k\in\Z$,}\\
\label{olala}
&\Ext^{k}_X(\f,\UU_+)=0, && \mbox{for any $k\neq 2$,}\\
\label{elele} &
\ext_X^2(\f,\UU_+)=2d-10.&&
\end{align}
\end{lem}

\begin{proof}
By Serre duality and \eqref{ipotesona}, the vanishing \eqref{alila}
follows from Proposition \ref{commento}.

Moreover, for $k=0,3$ we have
$\Ext_X^k(\f,\UU_+)= 0$ by
stability of the sheaves $\UU_+$ and $\f$.
Since $\Ext_X^1(\f,\OO_X)=0$, by applying the functor
$\Hom_X(\f,-)$ to the sequence
\eqref{eq:universalU+} we get $\Ext_X^1(\f,\UU_+) \cong
\Hom_X(\f,\UU^*_+)$,
which vanishes by stability of $\f$ and $\UU_+$. This proves \eqref{olala}.
Finally, by the Riemann-Roch formula \eqref{RRg7} we obtain the last equality.
\end{proof}

\begin{lem} \label{terzo}
Let $d\geq 7$ and let $\f$ be a sheaf in $\Mo_X(2,1,d)$ satisfying
\eqref{ipotesona}.
Then we have:
\[\Hom_X(\UU^*_+,F)=0.\]
\end{lem}

\begin{proof}
Fix a point $y$ in $\Gamma$.
Applying the functor $\Hom_X(-,F)$ to the exact sequence
\eqref{GG}
we obtain an exact sequence:
\[
0 \to \Hom_X(\EE_y,F) \to \Hom_X(\UU^*_+,F) \to
\Hom_X(\GG_y,F).
\]
By Lemma \ref{primo} we know
that the leftmost term vanishes.
Assume that the rightmost does not, and consider a non-zero
map $f :
\GG_y \to F$. Set $F' = \im(f)$ and note that, by stability
of the
sheaves $F$ and $\GG_y$, we must have $\rk(F')=2$ and
$c_1(F')=1$.
We have thus an exact sequence:
\begin{equation}
  \label{eq:Fprime}
  0 \to F' \to F \to T \to 0,
\end{equation}
where $T$ is a torsion sheaf, with $\dim(\supp(T))\leq 1$.
Note that $F'$ is stable, since any destabilizing subsheaf would destabilize also $F$.
By \cite[Propositions 1.1 and
1.9]{hartshorne:stable-reflexive}, the sheaf $\ker(f)$
must be a line bundle of degree zero.
This means that $\ker(f) \cong \OO_X$, and we have an exact
sequence:
\[
0 \to \OO_X \to \GG_y \to F' \to 0.
\]
So $F'$ satisfies $c_2(F')=7$, $c_3(F')=2$.
Hence by \eqref{eq:Fprime} it follows that
$d=c_2(F)=7+c_2(T)\le 7$, since $c_2(T)$ is non-positive.
Hence we have $d = 7$ and in this case
we have $c_2(T)=0$, and $c_3(T)=-2$.
But this is a contradiction since $c_3(T)$
must be non-negative.
We have thus proved our claim.
\end{proof}

\subsection{Canonical resolution of a bundle in $\Mo_X(2,1,d)$}

We will show here that a sheaf $F$
in $\Mo_X(2,1,d)$ which satisfies
$\HH^1(X,F(-1))=0$ admits a canonical resolution having two
terms, see the
formula \eqref{resolution} below.
Recall that, if the threefold $X$ is ordinary, such a sheaf exists for all $d\geq 6$ by Theorem \ref{add-line}.

\begin{prop} \label{cosonuovo}
Let $d\geq 6$ and let $\f$ be a sheaf in $\Mo_X(2,1,d)$ such that
$\HH^1(X,F(-1))=0.$
Then $\cF=\phx(\f)$ is a simple vector bundle on $\Gamma$, with:
\begin{equation} \label{alberto}
\rk(\cF)=d-5, \qquad \deg(\cF)=5\,d-24.
\end{equation}

Moreover, $\f$ admits the following canonical resolution:
\begin{equation}\label{resolution}
0 \to \Ext_X^2(\f,\UU_+)^* \ts \UU_+^* \xr{\zeta_\f}
\ph(\cF) \to \f \to 0,
\end{equation}
and $\ph(\cF)$ is a simple vector bundle. 
\end{prop}

\begin{proof}
Consider the stalk over a point $y \in \Gamma$ of the sheaf
$\HHH^k(\phx(\f))$.
We have:
\begin{equation}
\HHH^k(\phx(\f))_y \cong \Ext^{k+1}_X(\EE_y,\f) \ts
\omega_{\Gamma,y}.
\end{equation}

Hence by Lemma \ref{primo} it follows that this group vanishes for all
$y\in \Gamma$ and for all $k \neq 0$, so $\phx(\f)$ can be identified with a coherent
sheaf $\cF$ on $\Gamma$.
By Remark \ref{rango-grado} the sheaf $\cF$ is locally
free over $\Gamma$, with rank $d-5$, and $\deg(\cF)=5\,d-24$.

Let us exhibit the resolution \eqref{resolution}.
Note that, by formula \eqref{triangolo-psi} and Lemma \ref{secondo} we get that the complex
$(\mathbf{\Psi} (\mathbf{\Psi^*}(F)))$ is concentrated in degree $-1$ and isomorphic to
$\Ext_X^{2}(F,\UU_+)^*\ts \UU_+^*$.
Hence the exact triangle \eqref{triangolo} provides
the resolution \eqref{resolution} for $\f$.

This resolution also proves that 
$\HHH^i(\ph(\cF))=0$, for all $i\neq0$, which means (by the definition
\eqref{cocca} of $\ph$) that
$\RR^i p_*(q^*(\cF\ts \EE))=0$ for all $i\neq0$.
Then by \cite[Corollaries 7.9.9 and 7.9.10]{EGA3}, we conclude that
$\ph(\cF)=\RR^0 p_*(q^*(\cF)\ts \EE)$ is a locally free sheaf on $X$.


Let us now prove that $\cF$ is a simple bundle.
If $d=6$, then $\cF$ is a line bundle, hence it is
obviously simple.
For $d\geq7$ we want to prove that the group
$\Hom_\Gamma(\cF,\cF)\cong\Hom_X(\ph(\cF),F)$
is $1$-dimensional.
Applying the functor $\Hom_X(-,F)$ to the sequence
\eqref{resolution}
we obtain:
\[\Hom_X(\ph(\phx(F)),F)\cong \Hom_X(F,F),\]
since the term $\Hom_X(\UU^*_+,F)$ vanishes by
Lemma \ref{terzo}.
Hence $\cF=\phx(F)$ is simple, for $F$ is.
Since the functor $\ph$ is fully faithful, it follows that also the vector bundle 
$\ph(\cF)$ is simple.
\end{proof}

\begin{lem} \label{sette}
Let $d\geq 7$ and let $F$ be as in the previous proposition,
and set $\cF=\phx(F)$, $A_F = \Ext^2_X(F,\UU_+)^*$.
Then we have the natural isomorphism:
\begin{equation}
  \label{eq:H2H0}
  A_F \cong \HH^0(\Gamma,\cF) \cong
\Hom_X(\UU^*_+,\ph(\cF)).
\end{equation}

In particular $\hh^0(\Gamma,\cF)=2d-10$.
\end{lem}

\begin{proof}
By Lemma \ref{terzo}
we know that $\Hom_X(\UU^*_+,F)=0$.
Therefore, applying the functor $\Hom_X(\UU^*_+,-)$ to the
resolution \eqref{resolution} we
obtain:
\[
A_F =\Ext^2_X(F,\UU_+)^* \cong \Hom_X(\UU_+^*,\ph(\cF)).
\]
By adjunction and \eqref{fistar} we have the isomorphisms:
\[
\Hom_X(\UU_+^*,\ph(\cF))\cong
\Hom_\Gamma(\phs(\UU_+^*),\cF)
\cong
\HH^0(\Gamma,\cF).
\]
The last statement follows by \eqref{elele}.
\end{proof}

\begin{lem} \label{lem-gg}
Let $d\geq 6$ and let $\f$ be a locally free sheaf in
$\Mo_X(2,1,d)$ with $\HH^1(X,F(-1))=0$.
Set $\cF = \phx(F)$. Then $\cF$ is globally generated and we have the exact sequence
\begin{equation}\label{ferra}
0 \to \cF^* \to \Ext^2_X(\f,\UU_+)^* \ts \OO_{\Gamma}
\to \cF \to  0.
\end{equation}
\end{lem}
\begin{proof}
Consider the complex $\phs(\f)$.
Let us compute
the stalk over the point $y \in \Gamma$ of
the sheaf $\cH^k(\phs(\f))$.
We obtain:
\[
\cH^{-k}(\phs(\f))_y \cong \Ext^{k}_X(F,\EE_y)^*.
\]

But the above vector space vanishes for any $k\neq 1$ in view of Lemma \ref{primo}, for $F$ is locally free.
So $\phs(\f)[-1]$ is a locally free sheaf.
Moreover,
applying \eqref{gro-gamma} we have
$(\phx(\f))^*\cong \RR \HHom_\Gamma(\cF,\OO_\Gamma) \cong
\phx(F^*) \otimes \omega_\Gamma^* [1]$.
Using the definition of $\phx$ and the isomorphism $F^*\cong F(-H_X)$, we have
$\phx(F^*) \otimes \omega_\Gamma^* [1]\cong
\RR q_*(p^*(F)\ts \EE^* (H_{\Gamma}-H_X))[1]\otimes \OO_\Gamma(-H_\Gamma) [1]$.
By definition of $\phs$ we have
$\phs(F)[-1] = \RR q_*(p^*(F)\ts \EE^* (-H_X))[2]$,
and so we get:
\[
\phs(\f)[-1]\cong (\phx(\f))^*.
\]

Remark that, for any sheaf $\cP$ on the curve $\Gamma$,
since the functor $\ph$ is fully faithful, we have:
$$\phs(\ph(\cP)) \cong \cP.$$

Thus, applying the functor $\phs$ to \eqref{resolution}, we
obtain, in view of \eqref{fistar},
the exact sequence \eqref{ferra}.
Hence,  the sheaf $\cF$ is globally generated.
\end{proof}

In order to set up our correspondence between $\Mo(d)$ and
$W^{2\,d-11}_{d-5,5\,d-24}$, we have to prove that, for a general $F$
in $\Mo(d)$, the vector bundle $\phx(F)$ over $\Gamma$ is stable. This
is done in the next lemma.

\begin{lem}\label{stabile}
For each integer $d\geq 6$, there exists a Zariski dense
open subset
$\Omega(d) \subset \Mo(d)$, such that each point $F_d$ of
$\Omega(d)$ satisfies $\HH^1(X,F_d(-1))=0$, and
 $\cF_d=\phx(\f_d)$ is a stable sheaf.
\end{lem}

\begin{proof}
Let us prove the statement by induction on $d\geq 6$.
If $d=6$, $\cF_6=\phx(\f_6)$ is stable since it is a line bundle.
Suppose now $d>6$, assume that $F_d$ and $F_{d-1}$ fit into
\eqref{defEd} for some line $L\subset X$, and that $\cF_{d-1}=\phx(\f_{d-1})$ is
a stable bundle.
Recall that $\cL = \phx(\OO_L)[-1]$ is a line bundle of
degree $5$ by Proposition \ref{OL}.
Applying the functor $\phx$ to the sequence \eqref{defEd},
we get
\begin{equation}\label{phxFd}
0\to \cL \to \cF_d \to \cF_{d-1} \to 0.
\end{equation}
Notice that the extension \eqref{phxFd} is non-trivial
because $\cF_d$ is indecomposable since it is simple (see Proposition \ref{cosonuovo}).

Since, by formulas \eqref{alberto}, we know that
$\mu(\cF_d)=\frac{5d-24}{d-5}=5+\frac{1}{d-5}$, it is
enough to
prove that $\cF_d$ is semistable.
Assume by contradiction that there exists a subsheaf $\cK$
destabilizing $\cF_d$ of rank $r<d-5$ and degree $c$.
Since $\cF_{d-1}$ is stable, we must have
\[
5+\frac{1}{d-5}< \frac{c}{r} \le 5+\frac{1}{d-6},
\]
from which we get
\[
0< \frac{c}{r}-\frac{5d-24}{d-5} \le \frac{1}{(d-5)(d-6)}.
\]
It is easy to check that the only possibility is $r=d-6$
and $c=5d-29$, and so we would have $\cK\cong
\cF_{d-1}$ and
\eqref{phxFd} would split, a contradiction.
Hence $\cF_d$ is stable. Therefore the same holds for a
general
point of $\Mo(d)$ by Maruyama's result
\cite{maruyama:openness}.
\end{proof}

\begin{rmk}
We do not know whether the bundle $\cF=\phx(F)$ is stable for {\it all} sheaves $F$ in $\Mo_X(2,1,d)$ with
$\HH^1(X,F(-1))=0$, not even
assuming that $F$ lies in the component $\Mo(d)$.
\end{rmk}

In the next section we will study the space $\Mo_X(2,1,d)$,
focusing first on the case $d \geq 7$,
where we give a birational description. The case $d=6$ will
be treated in greater detail further on.

\subsection{The moduli spaces $\Mo_X(2,1,d)$, with $d\geq7$}
Here we show that the component $\Mo(d)$ of the variety
$\Mo_X(2,1,d)$ containing
the sheaves arising from the construction of Theorem
\ref{add-line} is birational to a component
$\Wo(d)$ of $W^{2\,d-11}_{d-5,5\,d-24}$.
Recall that in Lemma \ref{stabile} we have introduced the open set
$\Omega(d)\subset\Mo(d)$. Every sheaf $F\in\Omega(d)$ satisfies the following two conditions:
\begin{enumerate}[i)]
\item \label{H1} the group $\HH^1(X,F(-1))$ vanishes,
\item \label{assume-stable} the vector bundle $\cF=\phx(\f)$
is stable.
\end{enumerate}
Then we have a morphism:
\begin{align*}
  \varphi: \Omega(d) & \to W^{2\,d-11}_{d-5,5\,d-24}, \\
  F & \mapsto \phx(\f)
\end{align*}
which is well-defined by Proposition \ref{cosonuovo} and Lemmas \ref{sette}, \ref{stabile}.

We denote by $\Wo(d)$ the irreducible component of $W^{2\,d-11}_{d-5,5\,d-24}$
containing the image of $\varphi$. We can thus state the following result.

\begin{thm} \label{thm:moduli-brill}
Let $X$ be a smooth ordinary prime Fano threefold of genus $7$,
and let $\f$ be a sheaf in $\Omega(d)$ for $d\geq7$.
Then:
\begin{enumerate}[i)]
\item \label{uno-d} the tangent space, respectively the space of obstructions, to $\Wo(d)$ at the point
  $[\phx(\f)]$ is naturally identified with $\Ext^1_X(F,F)$,
  respectively with $\Ext^2_X(F,F)$;
\item \label{due-d} the varieties $\Mo(d)$ and $\Wo(d)$ are birational, both
  generically smooth of dimension $2d-9$.
\end{enumerate}
\end{thm}

\begin{proof}
The goal of our proof will be to construct an inverse map $\vartheta$
to the morphism $\varphi$, defined on a suitable open subset of our
component $\Wo(d)$.

We let $B(d)$ be the subset of $\Wo(d)$ consisting of those sheaves
$\cF$ such that:
\begin{enumerate}[a)]
\item \label{serve-a} the following natural evaluation map is surjective:
  \[
  e_\cF=e_{\OO,\cF} : \HH^0(\Gamma,\cF) \ts \OO_\Gamma \to \cF
  \]
\item \label{serve-b} the kernel of $e_\cF$ is isomorphic to $\cF^*$;
\item \label{serve-c} the complex $\ph(\cF)$ is concentrated in degree zero.
\end{enumerate}
Then we would like to define $\vartheta$ over $B(d)$ in the following way:
\begin{align*}
  \vartheta: B(d) & \to \Mo_X(2,1,d) \\
  \cF & \mapsto \cok(\HHH^0(\ph(e_\cF))).
\end{align*}

Let us prove that the sheaf $F=\cok(\HHH^0(\ph(e_\cF)))$ lies in $\Mo_X(2,1,d)$.
First note that the duality \eqref{gro-gamma} gives the isomorphism:
\[
\ph(\cF^*)[1] \cong \ph(\cF)^*(1),
\]
where these are both locally free sheaves by \eqref{serve-c} and
by the same argument as in Proposition \ref{cosonuovo}.
By \eqref{serve-a} and \eqref{serve-b} we can apply the functor $\ph$
to the exact sequence:
\[
0 \to \cF^* \xr{e_\cF^\top} \HH^0(\Gamma,\cF) \ts \OO_\Gamma \xr{e_\cF} \cF \to 0,
\]
and get (by \eqref{serve-c}) a long exact sequence of the form:
\[
0 \to \HH^0(\Gamma,\cF) \ts \UU_+^* \xr{\HHH^0(\ph(e_\cF))} \ph(\cF) \to \ph(\cF)^*\ts \OO_X(1) \to
\HH^0(\Gamma,\cF)^* \ts \UU_+(1) \to 0.
\]
The sheaf $F$ is then the image of the middle map in the above
sequence.
By \cite[Proposition 1.1]{hartshorne:stable-reflexive} the sheaf $F$
is reflexive, and sits into the following
exact sequence:
\begin{equation}\label{tappeto}
0 \to \HH^0(\Gamma,\cF) \ts \UU_+^* \xr{\HHH^0(\ph(e_\cF))} \ph(\cF) \to F \to 0.
\end{equation}

By Grothendieck-Riemann-Roch one can calculate the rank and the Chern
classes of $\ph(\cF)$, and deduce from the above exact sequence that
the sheaf $F$ has rank $2$ and $c_1(F)=1$, $c_2(F)=d$, $c_3(F)=0$.
Therefore $F$ is locally free (because it is reflexive with vanishing $c_3$), so it is also stable, once we prove
$\Hom_X(\OO_X(1),F)=0$.
Recall that $\Ext^k_X(\OO_X(1),\UU^*_+)=0$ for any integer $k$.
Further, it is easy to see that $\phs(\OO_X(1))=0$.
Then, applying the functor $\Hom_X(\OO_X(1),-)$ to \eqref{tappeto}, we have:
\[
\Hom_X(\OO_X(1),F)\cong \Hom_X(\OO_X(1),\ph(\cF))\cong
\Hom_\Gamma(\phs(\OO_X(1)),\cF)=0.
\]
We have thus proved that $F$ lies in $\Mo_X(2,1,d)$, so the map
$\vartheta$ is defined on $B(d)$.

\vspace{0.2cm}

Applying the functor $\phx$ to \eqref{tappeto} and using \eqref{fiics},
we have $\phx(F)\cong \cF$, so $\varphi(\vartheta(\cF))=\cF$.
Let us now show that $\vartheta \circ \varphi$ is the identity over the
open subset of $\Omega(d)$ consisting of locally free sheaves.
For any sheaf $E$ in $\Omega(d)$,
the sheaf $\cF = \phx(E)$ lies in $\Wo(d)$ and satisfies \eqref{serve-c}.
Assume now in addition that $E$ is locally free.
Set again $A_E=\Ext^2_X(E,\UU_+)^*$, and keep in mind the natural
isomorphism of Lemma \ref{sette}:
\[
A_E \cong \HH^0(\Gamma,E) \cong \Hom_X(\UU^*_+,\ph(\cF)).
\]
Further, by Lemma \ref{lem-gg}, since $E$ is a locally free
sheaf in $\Mo(d)$, we have that $\cF$ satisfies also \eqref{serve-a} and \eqref{serve-b}.
We look at the resolution
\[
0 \to A_E^* \ts \UU_+^* \xr{\zeta_E} \ph(\cF) \to E \to 0
\]
given by Proposition \ref{cosonuovo}.
Notice that the map
$\zeta_E$ agrees, up to a non-zero scalar, with the map $\cH^0(\ph(e_\cF))$.
Indeed, both such maps are non-zero elements of $A_E \ts A_E^*$, invariant under the natural
$\GL(A_E)$-action, and such invariant is either zero, either unique up to non-zero scalar.
Therefore, we have:
\[
E \cong \cok(\HHH^0(\ph(e_\cF))),
\]
which proves that $\vartheta \circ \varphi$ gives back $E$ when applied to $E$.

\vspace{0.2cm}

We have thus proved that $B(d)$ is isomorphic to an open subset of
$\Omega(d)$, which in turn is open in $\Mo(d)$, and
by Theorem \ref{add-line} $\Mo(d)$ is generically smooth of dimension  $2d-9$.
This will prove \eqref{due-d} once we show that $B(d)$ is dense in $\Wo(d)$.
In turn, it will follow from \eqref{uno-d} that $\Wo(d)$ is
smooth of dimension $2d-9$ at any point $\cF = \phx(F)$ such that $\Mo(d)$
is smooth at $F$, which implies that $B(d)$ is dense in $\Wo(d)$.

Therefore, it only remains to setup the natural identifications of tangent and
obstruction spaces required for \eqref{uno-d}. Let us accomplish this task.
By using adjunction, we have the natural isomorphisms:
\begin{align}
& \nonumber \Ext^1_X(\ph(\cF),F) \cong
\Ext^1_\Gamma(\cF,\cF), \\
& \label{H1cE} \Ext^1_X(\UU^*_+,F) \cong
\Ext^1_X(\ph(\OO_{\Gamma}),F) \cong
\Ext^1_{\Gamma}(\OO_{\Gamma},\cF).
\end{align}
Here, to prove \eqref{H1cE}, by \eqref{ss1} and \eqref{phO} it suffices to show
$\Ext^2_X(\UU_+(1),F)=0$ and $\Ext^3_X(\UU_+(1),F)=0$.
By Serre duality we have $\Ext^k_X(\UU_+(1),F)^*\cong\Ext^{3-k}_X(F,\UU_+)$, which
vanishes, for $k=2,3$, by Lemma \ref{secondo}.

Now, applying the functor $\Hom_X(-,F)$ to \eqref{resolution}
we obtain a long exact sequence:
\begin{align*}
&&\Hom_X(\UU^*_+,F) \ts A_F^*\to \Ext^1_X(F,F) \to \Ext^1_X(\ph(\cF),F) \xr{\eta_F}&&\\
&&\xr{\eta_F} \Ext^1_X(\UU^*_+,F) \ts A_F^* \to \Ext^2_X(F,F)\to \Ext^2_X(\ph(\cF),F),&&
\end{align*}
where the map $\eta_F$ is defined as $\Ext^1_X(\zeta_F,F)$.
Note that the first and the last terms of the above sequence vanish, hence we can identify
$\Ext^1_X(F,F)$ with the kernel of the map $\eta_F$ and $\Ext^2_X(F,F)$ with the cokernel of $\eta_F$.
But since $\zeta_F = \HHH^0(\ph(e_\cF))$, we have:
\[
\eta_F = \Ext^1_X(\ph(e_\cF),F) = \Ext^1_\Gamma(e_\cF,\phx(F)) =
\Ext^1_\Gamma(e_\cF,\cF) = \pi_\cF^\top.
\]

In view of the interpretation of the kernel and cokernel of
$\pi^\top_\cF$ (see Section \ref{befana}), we have thus constructed the required identification of the tangent space
to $W^{2\,d-11}_{d-5,5\,d-24}$ at the point $\cF$ (i.e. of
$\ker(\pi^\top_\cF)$) with $\Ext^1_X(F,F)$.
The same argument identifies the obstruction space with $\Ext^2_X(F,F)$.
\end{proof}

Working out a relative version of the construction of the previous
theorem, we get the following result.

\begin{corol} \label{thm:fine}
   The moduli space $\Omega(d)\subset \Mo(d)$ is fine.
\end{corol}

\begin{proof}
For any $d\geq 6$, we let $\Po(d)$ be the moduli space of stable
vector bundles on
$\Gamma$ of rank $d-5$ and degree $5\,d-24$.
Thus $\Wo(d)$ is a subvariety of $\Po(d)$.
Since the rank and the degree are coprime, it is well known
that this moduli space is fine.
So we denote by $\PP$ the universal bundle over $\Gamma
\times
\Po(d)$, and by abuse of notation its restriction to the product $\Gamma
\times \Wo(d)$.

We would like to exhibit a universal sheaf $\FF$ over $X
\times \Omega(d)$
such that, for a given closed point $z$ of $\Omega(d)$
representing a
stable sheaf $F$, the restriction of $\FF$ to $X\times
\{z\}$ is
isomorphic to $F$.

Recall by Theorem \ref{thm:moduli-brill} that $\varphi$ maps
$\Omega(d) \subset \Mo(d)$
to an open subset of $\Wo(d)$.
Consider the projections:
\[
\xymatrix@-2ex
{
& X \times \Gamma \times \Omega(d) \ar_{p \times 1}[dl]
\ar^{q \times \varphi}[dr]  \\
X \times \Omega(d) && \Gamma \times \varphi(\Omega(d))
}
\]

We consider the pull-back to $X \times \Gamma \times
\Omega(d)$ of the map $\alpha : \UU^*_+ \to \EE$ of
\eqref{kuz-exact}.
We tensor this map with $(q\times \varphi)^*(\PP)$.
We have thus a morphism:
\[
\UU^*_+ \boxtimes (q\times \varphi)^*(\PP) \xr{\alpha
\boxtimes 1} \EE
\boxtimes (q\times \varphi)^*(\PP).
\]

We define the universal sheaf $\FF$ as the cokernel of the
map
$(p\times 1)_*(\alpha \boxtimes 1)$.
Let us verify that $\FF$ has the desired properties.
So choose a closed point $z \in \Omega(d)\subset \Mo(d)$,
and consider the
corresponding sheaf $F_z$ on $X$ and the vector bundle
$\PP_{\varphi(z)} \cong \phx(F_z)$ on
$\Gamma$.
Notice that the sheaf $(q\times
\varphi)^*(\PP_{\varphi(z)})$ is just
$q^*(\phx(F_z))$. Then, evaluating at the point $z$ the map
$(p\times
1)_*(\alpha \boxtimes 1)$ we obtain the map:
\[
\HH^0(\Gamma,\phx(F_z)) \ts \UU_+^* \to \ph(\phx(F_z)).
\]

Recall the natural isomorphism $\HH^0(\Gamma,\phx(F_z))\cong
\Ext^2_X(F_z,\UU_+)^*$, and note that, by functoriality,
this map must agree
with the map $\zeta_{F_z}$ given by the resolution
\eqref{resolution}.
Thus its cokernel is isomorphic to $F_z$.
\end{proof}

\subsection{The moduli space $\Mo_X(2,1,6)$}

Here we focus on the moduli space $\Mo_X(2,1,6)$, and we prove that it
is isomorphic to the Brill-Noether locus $W^1_{1,6}$ on the
homologically projectively dual curve $\Gamma$.
This makes more precise a result of Iliev-Markushevich,
\cite{iliev-markushevich:sing-theta:asian}.
Then we investigate the subvariety of
$\Mo_X(2,1,6)$
consisting of vector bundles which are not globally
generated.
We will see that these bundles are in one-to-one
correspondence with
non-reflexive sheaves in $\Mo_X(2,1,6)$. Finally we will see
that these
two subsets are interchanged by a natural involution of
$\Mo_X(2,1,6)$. Here is our main result.

\begin{thm} \label{thm:lune}
  Let $X$ be a smooth prime Fano threefold of genus $7$.
\begin{enumerate}[A)]
 \item \label{lune-A}
  The map $\varphi:F\mapsto \phx(F)$ gives an isomorphism
  of the moduli space $\Mo_X(2,1,6)$ onto the Brill-Noether
  variety $W^1_{1,6}$. In particular, $\Mo_X(2,1,6)$ is a
  connected threefold.
  Moreover it is a fine moduli space.
\item\label{lune-B}
  If $X$ is not exotic, then $\Mo_X(2,1,6)$ has at most finitely many
  singular points.
  If $X$ is general, then  $\Mo_X(2,1,6)$ is smooth and irreducible.
\end{enumerate}
\end{thm}

We prove now the first part of the theorem, while we postpone the
second part to the end of the subsection.

\begin{proof}[Proof of Theorem \ref{thm:lune}, part \ref{lune-A}]
First of all the map $\varphi:F\mapsto \phx(F)$ is well-defined.
Indeed, let $\f$ be any sheaf in $\Mo_X(2,1,6)$.
By Proposition \ref{agglomerato}, part \eqref{annullarsi}, we know that $F$ satisfies the hypothesis \eqref{ipotesona}.
Then by Proposition \ref{cosonuovo}, $\phx(F)$ is a line bundle of
degree $6$ on $\Gamma$. Set $\cL = \phx(F)$.
We have to prove that $\cL$ admits at least two independent global sections.
If $F$ is locally free, by Lemma \ref{lem-gg} we have
that $\cL$ is globally generated. Hence $\hh^0(\Gamma,\cL)=1$ would imply
$\cL \cong \OO_\Gamma$, which is impossible, i.e. $\hh^0(\Gamma,\cL)
\ge 2$.
This means that $\phx(\f)$ lies in $W^1_{1,6}$.
If $F$ is not locally free, then it fits in the exact sequence
\eqref{doubledual}.
Recall that by Proposition \ref{OL} we know that $\phx(\OO_L)[-1]$ is a line bundle
$\cM$ contained in $W^{1}_{1,5}$.
Hence, applying $\phx$ to the exact sequence \eqref{doubledual}, we obtain:
\begin{equation} \label{non-gg}
0 \to \cM \to \cL \to \OO_y \to 0,
\end{equation}
where $y$ is a point of $\Gamma$.
Therefore we have again $\hh^0(\Gamma,\cL)\geq \hh^0(\Gamma,\cM) \geq 2$,
and $\phx(\f)$ lies in $W^1_{1,6}$.
Note that the equality $\hh^0(\Gamma,\cL) = 2$
must be attained for all $\cL$, since $W^2_{1,6}$ is empty in view
of Mukai's result (see \cite[Table 1]{mukai:curves-symmetric-I}).
Note that in this case the open subset $\Omega(6)$ coincides in fact with all of $\Mo_X(2,1,6)$.

Now we want to provide an inverse map $\vartheta:W^1_{1,6}\to\Mo_X(2,1,6)$
of $\varphi$.
Take a line bundle $\cL$ in $W^1_{1,6}$, and denote again by
$e_\cL=e_{\OO_\Gamma,\cL}:\HH^0(\Gamma,\cL)\ts \OO_\Gamma \to \cL$
the natural evaluation map.
We distinguish two cases according to whether $\cL$ is globally generated or not.

In the former case, we have an exact
sequence:
\begin{equation}
  \label{basta}
  0 \to \cL^* \to \HH^0(\Gamma,\cL) \ts \OO_\Gamma \xr{e_{\cL}} \cL \to 0.
\end{equation}
Since, for any $x \in X$ the vector bundle $\EE_x$ is stable, we have
$\HH^0(\Gamma,\EE_x\ts \cL^*)=\HH^1(\Gamma,\EE_x\ts
\cL)=0$.
Hence the line bundle $\cL$ satisfies all the conditions
\eqref{serve-a}, \eqref{serve-b},
\eqref{serve-c} of the proof of Theorem \ref{thm:moduli-brill}.
Then the same proof of such theorem allows us to define
$\vartheta(\cL)=\cok(\HHH^0(\ph(e_\cL)))$ and to prove that
$\varphi(\vartheta(\cL))=\cL$.

It remains to find an inverse image via $\varphi$ of a non-globally generated sheaf $\cL$.
In this case, the image $\cM \subset \cL$ of $e_\cL$ must be a line bundle, with
$\hh^0(\Gamma,\cM) = \hh^0(\Gamma,\cL)=2$.
Then  $\cM$ must lie in $W^1_{1,5}$, since $\Gamma$ has no $g^1_4$ by \cite{mukai:curves-symmetric-I}.
We have an exact sequence of the form \eqref{non-gg}, for
some $y\in \Gamma$.
Applying the functor $\ph$ to this sequence, by Proposition \ref{OL}
and formula \eqref{rigatoni}, we obtain:
$$0\to A_L\ts\UU^*_+ \xr{\HHH^0(\ph(e_\cL))} \ph(\cL)\to\EE_y\to\OO_L\to 0,$$
where $L$ is the line contained in $X$ such that $\cM\cong\phx(\OO_L)[-1]$
and $A_L=\HH^0(\Gamma,\cM)\cong\HH^0(\Gamma,\cL)$.
It is easy to see that the image of the middle map in the exact sequence above
is a sheaf $F\in \Mo_X(2,1,6)$.
We define again $\vartheta(\cL)=\cok(\HHH^0(\ph(e_\cL)))$ and
since $\phx(F)\cong\cL$, it follows $\varphi(\vartheta(\cL))=\cL$.
Note that the map $\vartheta$ is defined in the same way as for globally generated bundles.

To show that $\vartheta\circ\varphi$ is the identity on $\Mo_X(2,1,6)$,
we need to prove that the map $\zeta_F$ appearing in the resolution
\eqref{resolution} provided by Proposition \ref{cosonuovo} agrees up
to a scalar with $\HHH^0(\ph(e_\cL))$.
Applying the functor $\Hom_X(\UU_+^*,-)$ to \eqref{resolution}, and
using adjunction and \eqref{fistar}, we obtain
$$\Ext_X^2(\f,\UU_+)^* \subseteq \Hom(\UU_+^*,\ph(\phx(\f)))\cong \HH^0(\Gamma,\phx(F)).$$
Now recall that $\dim\Ext_X^2(\f,\UU_+)=2$, by \eqref{elele}, and
$\hh^0(\Gamma,\phx(F))\leq2$, since $W^2_{1,6}$ is empty, hence we
conclude that
$\Ext_X^2(\f,\UU_+)^* \cong \Hom(\UU_+^*,\ph(\phx(\f))),$ and
the map $\zeta_F$ in the resolution \eqref{resolution} is
uniquely determined. This proves that $\vartheta(\varphi(F))=F$.

Now, with the same proof of Corollary \ref{thm:fine} one can show that the moduli space
$\Mo_X(2,1,6)$ is fine.
The fact that $W^1_{1,6}$ is a connected threefold is well-known, see
for instance \cite[IV, Theorem 5.1 and V, Theorem 1.4]{acgh}.
This completes the proof of part \eqref{lune-A}.
\end{proof}

Now we will analyze the space $\Mo_X(2,1,6)$ in greater detail.

\begin{lem} \label{lem:F-not-gg}
Let $F$ be a sheaf in $\Mo_X(2,1,6)$. Then either $F$ is
globally
generated, or there exists an exact sequence:
\begin{equation}
  \label{eq:Ky}
  0 \to I \to F \to \OO_{L}(-1) \to 0,
\end{equation}
where $L$ is a line contained in $X$ and $I$ is a sheaf
fitting into:
\begin{equation}
  \label{eq:KE}
  0 \to \EE_y^* \to \HH^0(X,F) \ts \OO_X \to  I \to 0,
\end{equation}
for some $y \in \Gamma$.
\end{lem}

\begin{proof}
If the sheaf $F$ fits into the exact sequence \eqref{eq:Ky},
it cannot
be globally generated, since $\OO_L(-1)$ has no global
sections.
So let us prove the converse implication.

Assume thus that $F$ is not globally generated, let $I$
(respectively,
$T$ and $K$) be the image (respectively, the cokernel and
the kernel)
of the natural evaluation map $e_{\OO,F} : \HH^0(X,F) \ts
\OO_X \to
F$.

By Proposition \ref{agglomerato},
we have $\HH^k(X,F)=0$, for
each $k\neq 0$, and $\hh^0(X,F)=4$.
We have the exact sequences:
\begin{equation}
  \label{pancia}
  0 \to K \to \OO_X^4 \to I \to 0, \qquad 0 \to I \to F \to T \to 0,
\end{equation}
where the induced maps $\HH^0(X,\OO_X^4) \to \HH^0(X,I)$ and
$\HH^0(X,I) \to \HH^0(X,F)$ compose to $e_{\OO,F}$, in particular $\hh^0(X,I)=4$.

Let us check that the torsion-free sheaf $I$ must have rank $2$
and $c_1(I)=1$.
By stability of $\OO_X$ and $F$, we must have $c_1(I)=1$ (and in this
case $\rk(I)=2$) or $c_1(I)=0$. But in the latter case, by the
uniqueness of the graded object associated to the Jordan-H\"older filtration
of $\OO_X^4$ (\cite[Proposition 1.5.2]{huybrechts-lehn:moduli}), we must have that $I$ is semistable with Jordan-H\"older
filtration:
\[
0 \to \OO_X \to I \to \OO_X \to 0.
\]
So $I \cong \OO_X^2$ which contradicts $\hh^0(X,I)=4$.
We have proved that $I$ has rank $2$ and $c_1(I)=1$.
Since $F$ is stable, by \eqref{pancia} we deduce that $I$ is stable with $c_2(I) \geq 6$.

Now, one easily sees that $K$ is a stable reflexive
(by \cite[Proposition 1.1]{hartshorne:stable-reflexive})
sheaf of rank $2$ with $c_1(K)=-1$, $c_2(K) =
12 - c_2(I)$ (by \eqref{pancia}). Then we have $c_3(K)\ge0$ and by Lemma \ref{K3} it follows $c_2(K) \geq 5$.
This leaves two cases, namely $c_2(I)=6$ or $7$.

Assume first that $c_2(I)=7$.
Then we can apply Proposition \ref{agglomerato} to
the sheaf $K(1)$ to prove that $K$ is locally free.
It follows that $K$ is of the form
$\EE_y^*$ for some $y$ by virtue of
Theorem \ref{summary}.
It follows that $\HH^k(X,K)=0$ for all $k$ by Proposition
\ref{agglomerato}, which by \eqref{pancia} implies $\HH^k(X,I)=0$ for $k\geq 1$ and
in turn $\HH^k(X,T)=0$ for all $k$.
We obtain that $T$ is
isomorphic to $\OO_L(-1)$
by a Hilbert polynomial computation. This concludes the proof in case $c_2(I)=7$.

Let us assume now that $c_2(I)=6$, which implies $c_2(K)=6$ and
$c_3(K)\ge0$ (recall that $K$ is reflexive).
In this case Proposition \ref{agglomerato} implies that $K$ is locally
free so $c_3(K)=0$.
So by \eqref{pancia}, using that $c_1(I)=1$, $c_1(K)=-1$, $c_2(I)=c_2(K)=6$,
we compute $c_3(I)=-c_3(K)$. Hence $c_3(I)$ also vanishes. So $c_k(F)=c_k(I)$ for all
$k$ hence the sheaf $T$ has $c_k(T)=0$ for all $k$.
Therefore $T=0$ and
$F$ is globally generated.
\end{proof}

We can consider the closed subvarieties of the Brill-Noether
variety
$W^1_{1,6}$ described by the following conditions:
\begin{align}
  \label{eq:subvars}
  {\sf G} & = \{\cL \in W^1_{1,6} \, \vert \, \mbox{$\cL$ is
not globally
    generated} \} \\
  {\sf C} & = \{\cL \in W^1_{1,6} \, \vert \, \mbox{$\cL$ is
    contained in a line bundle lying in $W^2_{1,7}$} \}.
\end{align}

Moreover, we have the following involution:
\[
\tau : W^1_{1,6} \to W^1_{1,6}, \qquad \cL \mapsto
  \cL^*\ts \omega_\Gamma.
\]

A suggestion due to the referee helped us in making more precise
the statement of the following proposition. The condition that
$\Gamma$ can be represented as a plane septic with 8 nodes holds
if $\Gamma$ is general, see e.g.\ \cite{logan}.

\begin{prop} \label{prop:otto}
  The sets ${\sf C}$ and ${\sf G}$ are interchanged by the
  involution
  $\tau$, and are both isomorphic to the product
  $\Gamma \times W^1_{1,5}$.

  The intersection ${\sf C} \cap {\sf G} \subset W^1_{1,6}$
  is a finite cover of the curve $W^1_{1,5}$.
  If $\Gamma$ can be represented as a plane septic with 8 nodes, the
  degree of this cover is $16$.
\end{prop}

\begin{proof}
Given a line bundle $\cL$ in ${\sf G}$, we consider,
as in the proof of Theorem \ref{thm:lune}, part \eqref{lune-A},
the image $\cM \subset \cL$ of the natural
evaluation map $e_{\OO_\Gamma,\cL}$.
We have $\cM \in W^1_{1,5}$ and an exact sequence of the form \eqref{non-gg}, for
some $y\in \Gamma$.
This defines a map ${\sf G} \to \Gamma \times W^1_{1,5}$.

Let us define an inverse map.
We first note that,
given $\cM$ in $ W^1_{1,5}$ and $y \in \Gamma$,
we have $\Ext^1_{\Gamma}(\OO_y,\cM) \cong \C$.
The unique extension $\cF$ from
$\OO_y$ to $\cM$ 
must lie in $W^1_{1,6}$, since $W^2_{1,6}$ is empty
by \cite[Table 1]{mukai:curves-symmetric-I}.

To put this in family,
we denote by $\PP$ the Poincaré line bundle on $\Gamma \times W^1_{1,5}$.
We consider the projection $\pi : \Gamma \times \Gamma \times W^1_{1,5} \to
\Gamma \times W^1_{1,5}$ onto the last two components, and the diagonal embedding
$\Delta: \Gamma \times W^1_{1,5} \to \Gamma \times \Gamma \times W^1_{1,5}$.
Then, there is a line bundle $\FF$ on $\Gamma \times \Gamma \times W^1_{1,5}$
fitting into:
\[
0 \to \pi^*(\PP) \to \FF \to \Delta_*(\LL) \to 0,
\]
for some line bundle $\LL$ on $\Gamma \times \Gamma$, so that
the one-dimensional space $\cF_{z}$ is the fiber
over $(z,y,\cM)$ of $\FF$.
Therefore $\FF$ is a family of bundles in $W^1_{1,6}$ parametrized by
$\Gamma \times W^1_{1,5}$,  giving thus a classifying map $\Gamma \times W^1_{1,5}
\to W^1_{1,6}$, which is clearly the desired inverse map.

To see the relation with $W^2_{1,7}$, we set $\cN = \cM^*\ts \omega_\Gamma$, we have:
\[
\hh^0(\Gamma,\cM) = \hh^1(\Gamma,\cN) = 2, \qquad
\hh^1(\Gamma,\cM) = \hh^0(\Gamma,\cN) = 3.
\]
It follows that $\cN$ lies in $W^2_{1,7}$.
Dualizing the sequence \eqref{non-gg} and tensoring by
$\omega_\Gamma$, we obtain the exact sequence:
\begin{equation} \label{eq:cN}
0 \to \tau(\cL) \to \cN \to \OO_y \to 0.
\end{equation}
So the line bundle $\tau(\cL)$ lies in ${\sf C}$. Since this
procedure is reversible,
we have proved that the involution $\tau$ interchanges the
subsets {\sf G}
and {\sf C}. Note that the map $\tau: \cM \mapsto \cM^*\ts
\omega_\Gamma$
gives an isomorphism from $W^{1}_{1,5}$ to $W^{2}_{1,7}$.

Let us now describe the intersection ${\sf C} \cap {\sf G}
\subset
W^1_{1,6}$. Recall that the map $\varphi_{|\cN|}$ associated
to a
given $\cN \in W^{2}_{1,7}$ maps $\Gamma$ to $\p^2$.
This map is generically injective and
the image is a curve of degree $7$,
smooth away from a subscheme of length $8$.
If $\cN$ is general enough,
then we get $8$ distinct double points $y_1,\ldots,y_8$,
\cite[Lemma 2.6]{iliev-markushevich:sing-theta:asian}.
For each $y_i$ we have a unique $\cM_i \in W^1_{1,5}$ given
by
the projection from the double point $y_i$. On the other hand any
subbundle
$\cM \in W^1_{1,5}$ of $\cN$ must correspond to the
projection from a
double point $y_i$. Namely we will have:
\[
0 \to \cM \to \cN \to \OO_{Z_i} \to 0,
\]
where $Z_i$ is the subscheme of $\Gamma$ over the double point $y_i$.

Now fix a line bundle $\cN$ in $W^{2}_{1,7}$.
A subbundle $\cL \in {\sf C}$ of $\cN$ corresponds to the projection
from a smooth point $y$ as soon as $\cL$ is globally
generated.
Therefore, $\cL$ lies in ${\sf C} \cap {\sf G}$ if and only
if we
have:
\begin{equation}
  \label{badaben}
  \cM_i \subset \cL \subset \cN,
\end{equation}
for some $i=1,\ldots,8$.
Then $\cL$ must be of the form $\cM_i(z)$ for some point $z$ in
$\Gamma$ lying over the double point $y_i$.
The number of such $z$ is finite for each $y_i$.
Thus we
have realized ${\sf C} \cap {\sf G}$ as a finite cover of
$W^1_{1,5}$.

If there is one line bundle $\cN$ such that the image of $\Gamma$
under $\varphi_{|\cN|}$ has $8$ nodes, then for any $i$
there are two points $z_{i,1}$, $z_{i,2}$ of $\Gamma$ which are mapped to
$y_i$ by $\varphi_{|\cN|}$. Then we set $\cL_{i,j} =
\cM_i(z_{i,j})$, and the line bundles $\cL_{i,j}$ fit into the
inclusions \eqref{badaben}. Hence the degree of the cover is $16$.
\end{proof}

We consider now the pull-back $\theta$ of $\tau$ to $\Mo_X(2,1,6)$, i.e. we set:
\[
\theta : \Mo_X(2,1,6) \to \Mo_X(2,1,6), \qquad \theta = \varphi^{-1}
\circ \tau \circ \varphi.
\]
We will next show that $\theta$ can be seen on
$\Mo_X(2,1,6)$ in terms of the functor $T$ of Corollary \ref{coin}.

\begin{prop} \label{prop:lune}
  Let $F$ be an element of $\Mo_X(2,1,6)$. Then we have:
  \begin{enumerate}[i)]
  \item \label{lune:parte1}
    the sheaf $\f$ is not locally free if and only if $\phx(\f)$ lies in
    ${\sf G}$.
  \item \label{lune:parte2}
    the sheaf $\f$ is not globally generated if and only if
    $\phx(\f)$ lies in ${\sf C}$.
  \end{enumerate}

Moreover the function $\theta$ is an involution which interchanges
the two subsets of sheaves which are not locally free, and not
globally generated.

For each $F$ in $\Mo_X(2,1,6)$ we have:
\begin{equation}
  \label{atomium}
  \theta(F) = T(F) = \varphi^{-1} \phx(\RR \HHom_X(F,\OO_X))[1].
\end{equation}

Finally, $\theta(F)$ is isomorphic to both the following objects:
\begin{align}
  \label{cono}
   & \RR \HHom_X\left((\HH^0(X,F)\ts \OO_X \to F),\OO_X\right)[-1] && \mbox{and:} \\
  \label{cono2}  & \RR \HHom_X(F,\OO_X) \to \HH^0(X,F)^*\ts \OO_X.
\end{align}
\end{prop}

\begin{proof}
We have already proved the implication ``$\Leftarrow$'' of
\eqref{lune:parte1} in Lemma \ref{lem-gg}.
To prove the converse, we consider a sheaf $F$ which is not locally free.
Then $F$ fits into an exact sequence of the form \eqref{doubledual}.
Applying the functor $\phx$ to this sequence and setting $\cL=\phx(F)$
we obtain an exact sequence of the form \eqref{non-gg} for some $\cM$
in $W^1_{1,5}$ (see Proposition \ref{OL}).
Since $\HH^0(\Gamma,\cM) \cong \HH^0(\Gamma, \cL)$,
the evaluation map $\HH^0(\Gamma,\cL) \ts \OO_\Gamma \to
\cL$ cannot be surjective, so $\cL$ lies in ${\sf G}$.

To prove \eqref{lune:parte2}, in view of Lemma \ref{lem:F-not-gg},
we have to show that the sheaf $F$ fits into \eqref{eq:Ky}, for
some $I$ fitting in \eqref{eq:KE} if and only if the line bundle
$\phx(F)$ lies in ${\sf C}$ . To show ``$\Rightarrow$'', we
consider a sheaf $F$ fitting into an exact sequence of the form
\eqref{eq:Ky}. Recall by Proposition \ref{OL} that $\cN =
\phx(\OO_L(-1))$ lies in $W^2_{1,7}$. Since $\phx(\OO_X)=0$, by
the exact sequence \eqref{eq:KE} we have
$\phx(I)[1] \cong \phx(\EE^*_y)[2]$, and using \eqref{gro-gamma}
we conclude that:
 \[
 \phx(I)[1] \cong \phx(\EE_y^*)[2] \cong \OO_y.
 \]
Thus applying the functor $\phx$ to \eqref{eq:Ky} we obtain
an exact sequence:
\[
0 \to \phx(F) \to \cN \to \OO_y \to 0,
\]
and $\phx(F)$ lies in ${\sf C}$.

To prove the converse implication, we consider a globally
generated
sheaf $F$ and the exact sequence:
\begin{equation} \label{eq:kernel}
0 \to K \to \HH^0(X,F) \ts \OO_X \to F \to 0.
\end{equation}
Remark that $K$ is a locally free sheaf and $K^*$ lies in
$\Mo_X(2,1,6)$ as well.
We note that, applying \eqref{gro-gamma}, we get the
natural isomorphism:
\[
\phx(K) \cong \phx(K^*)^* \ts \omega_\Gamma[-1].
\]
On the other hand, by \eqref{eq:kernel} we get $\phx(K)
\cong
\phx(F)[-1]$. Then we have:
\[
\phx(F) \cong \tau(\phx(K^*)).
\]
But $\phx(K^*)$ is globally generated by Lemma \ref{lem-gg},
hence we
are done since $\tau$ interchanges ${\sf C}$ and ${\sf G}$.
We have thus established \eqref{lune:parte1} and \eqref{lune:parte2}.

It follows that $\theta$ interchanges the sheaves which are not
locally free, and the sheaves which are not
globally generated, and clearly $\theta$ is an involution.

To show the expression \eqref{atomium} of $\theta$, recall that $\phx \circ T = \tau \circ
\phx$ by Corollary \ref{coin}. Therefore, for any $F$ in
$\Mo_{X}(2,1,6)$ we have $\theta(F) = \varphi^{-1}(\phx(T(F)))$.
Since for any object $a$ in $\D(X)$ we have $\phx(\ph(\phx(a))) \cong
\phx(a)$, it follows that $\theta(F) = \varphi^{-1} \phx(\RR \HHom_X(F,\OO_X))[1]$.

By \eqref{atomium}, Corollary \ref{coin} and Remark \ref{chestrano},
it follows that $\theta(F)$ is isomorphic to \eqref{cono}.
Finally, by \eqref{referee}, we have that \eqref{cono} is isomorphic to \eqref{cono2}.
\end{proof}

\begin{prop} \label{pulcinotto}
  Let $F$ be a sheaf in $\Mo_X(2,1,6)$, and let $M, N\subset X$ be two
  distinct lines. Set $\cL = \phx(F)$, $\cM = \phx(\OO_M)[-1]$ and $\cP=\phx(\OO_{N}(-1))$.
  Then $F$ is not globally generated over $N$ and
  not locally free over $M$ if and only if we have $\cM \subset \cL
    \subset \cP$.
\end{prop}

\begin{proof}
  If $F$ is not locally free over $M$ and not globally generated over
  $N$, then by Proposition \ref{agglomerato} and Lemma
  \ref{lem:F-not-gg} we have two surjective maps $F \to \OO_M$ and
  $F\to \OO_N(-1)$.
  By applying $\phx$, the first one gives $\cM \subset \cL$ and the second one $\cL
  \subset \cP$, so one implication is clear.

  Conversely, if $\cM \subset \cL$, we have $\cL/\cM = \OO_y$, with $y
  \in \Gamma$ and
  applying $\ph$, by Proposition \ref{OL} we get:
  \[
  0 \to (\UU_+^*)^2 \to \ph(\cL) \to \EE_y \to \OO_M \to 0,
  \]
  so using \eqref{resolution} we see that $F$ is the image of the
  middle map in the sequence above, hence it is not locally free over $M$.
  Further if $\cL \subset \cP$, then we have
  a commutative exact diagram:
  \[
  \xymatrix@-2ex{
  0 \ar[r] & \OO_\Gamma^2 \ar[r] \ar[d] & \OO_\Gamma^3  \ar[r] \ar[d] & \OO_\Gamma\ar[d] \ar[r] & 0\\
  0 \ar[r] & \cL \ar[r] & \cP \ar[r] & \OO_{z} \ar[r] & 0,
  }
  \]
  with $z \in \Gamma$, and the vertical maps are
  natural evaluations.
  Applying $\ph$ to this diagram and taking cohomology, we get:
  \[
  \xymatrix@-2ex{
  0 \ar[r] & (\UU_+^*)^2 \ar[r] \ar[d] & (\UU_+^*)^3  \ar[r] \ar[d] & \UU_+^* \ar[d] \ar[r] & 0 \\
  0 \ar[r] & \ph(\cL) \ar[r]  & \ph(\cP) \ar[r] & \EE_z \ar[r] &0\\
  }
  \]
  Taking cokernels, by using snake lemma,
  \eqref{resolution}, \eqref{OL-1}, and \eqref{GG}
  we obtain an exact sequence:
  \[
  0 \to \OO_X \to \GG_z \to F \to \OO_N(-1) \to 0.
  \]
  The rightmost part of this sequence gives an exact sequence of the
  form \eqref{eq:Ky}, so $F$ is not globally generated over $N$.
  Note that the leftmost part gives rise to an exact sequence of the
  form \eqref{eq:KE}. Indeed, if $I$ is the image of middle map above,
  then $I$ is also the image of the evaluation map $e_{\OO_X,F}$ and,
  recalling \eqref{GGdual}, and the fact that for all $z \in \Gamma$,
  $(\UU_-)_z \cong \OO_X^5$,
  we get a commutative diagram:
  \[
  \xymatrix@-2ex{
    &&\OO_X \ar@{=}[r] \ar[d] & \OO_X \ar[d] \\
    0 \ar[r] & \EE_z^* \ar[r] \ar@{=}[d] & \OO_X^5  \ar[r] \ar[d] & \GG_z \ar[r] \ar[d] & 0 \\
    0 \ar[r] & \EE_z^* \ar[r] & \OO_X^4 \ar[r] & I \ar[r] & 0
  }
  \]
  The bottom row thus gives \eqref{eq:KE}.
  In particular $N \cap M$ is a point by Lemma \ref{gennaio-2011!}.
  In this case the conic $C = M \cup N$ satisfies $\phx(\OO_C) \cong \OO_{\{y,z\}}$.
\end{proof}

The following lemma follows a suggestion of Dimitri Markushevich.

\begin{lem} \label{theta}
  The set of singular points $\cL$ of $W^1_{1,6}$, such that $\cL$ is globally generated,
   is in bijection with the set of even effective theta-characteristics on $\Gamma$.
  In particular, this set is finite. Moreover, it is empty if $\Gamma$ is outside
  a divisor in the moduli space of curves of genus $7$, and of
  cardinality $1$ if $\Gamma$ is general in that divisor.
\end{lem}

\begin{proof}
  According to Mukai's classification in
  \cite{mukai:curves-symmetric-I}), the smooth curve section $\Gamma$ of the
  spinor $10$-fold satisfies $W^2_{1,6} = \emptyset$, and a general
  curve of genus $7$ is of this form.
  Now recall that a line bundle $\cL$ lies in the
  singular locus of $W^1_{1,6}$ if and only if the Petri map:
  \[
  \pi_\cL : \HH^0(\Gamma,\cL) \ts \HH^0(\Gamma,\cL^* \ts \omega_\Gamma) \to \HH^0(\Gamma,\omega_\Gamma)
  \]
  is not injective.
  Since $\cL$ is globally generated, it is clear that
  $\ker(e_{\OO_\Gamma,\cL})$ is isomorphic to $\cL^*$.
  Hence, the kernel of $\pi_\cL$ is isomorphic to
  $\HH^0(\Gamma,\cL^* \ts \cL^* \ts \omega_\Gamma)$.
  Therefore, the above map is injective unless $\cL \ts \cL \cong
  \omega _\Gamma$, which means that $\cL$ is an even
  effective theta-characteristic (even here means that
  $\hh^0(\Gamma,\cL)$ is an even number, $2$ in this case).

  By \cite[Theorem 2.16]{teixidor-i-bigas:half-canonical} the set of curves of genus $7$ admitting an even
  effective theta-characteristic form a divisor in the moduli space
  of curves of genus $7$, and the general curve in this divisor has precisely one
  even effective theta-characteristic.
  This concludes the proof.
\end{proof}

\begin{proof}
[Proof of Theorem \ref{thm:lune}, part \ref{lune-B}]
Let us assume $X$ to be non-exotic, and prove that $\Mo_X(2,1,6)$ has
at most finitely many singular points.
In view of Theorem \ref{thm:lune}, part \ref{lune-A}, the space $\Mo_X(2,1,6)$ is
isomorphic to $W^1_{1,6}$.
The number of singular points of $W^1_{1,6}$ which correspond to
globally generated line bundles is finite by Lemma \ref{theta}.

We consider thus a line bundle $\cL \in {\sf G}$
which is a
singular point of $W^1_{1,6}$.
Since $\tau$ is an isomorphism, we can also assume
that $\tau(\cL)$ is not globally generated, i.e.
$\cL\in{\sf C}\cap{\sf G}$.
In particular
we must have an exact sequence of the form \eqref{non-gg}, and
there exists $\cN\in W^2_{1,7}$ such that $\cM\subset
\cL\subset \cN$.
Applying $\tau$, we also get $\tau(\cN)\subset \tau(\cL) \subset
\tau(\cM)$.
Now, as in the proof of Lemma \ref{theta}, we note that $\tau(\cL)$ lies in the
singular locus of $W^1_{1,6}$ if and only if the Petri map $\pi_\cL$
is not injective. In this case the kernel is isomorphic to
  $\HH^0(\Gamma,\tau(\cN)^* \ts \cL)$.
Since we assume that this space is non-zero, we have an inclusion $\tau(\cN)\subset \cL$.
It follows that $\cM\cong\tau(\cN)$, since $\cL$ contains a unique
line bundle lying in $W^1_{1,5}$.
We have thus an inclusion $\cM\subset\tau(\cM)$, which means $\HH^0(\Gamma,\cM^* \ts \cM^* \ts \omega_\Gamma)\neq 0$,
so that $\cM$ is a singular point of $W^1_{1,5}$ (see Remark \ref{marc}).
Since $W^1_{1,5} \cong \sH^0_1(X)$ by Proposition \ref{OL}, and since $X$ is
not exotic, the number of singular points of this form in $W^1_{1,6}$ is finite, and
we are done.

Finally, note that if $X$ is general, then the curve $\Gamma$ is general. Then
it is well-known that $W_{1,6}^1$ is smooth and irreducible, see for instance
\cite[V, Theorem 1.6]{acgh}.
It follows that $\Mo_X(2,1,6)$ is a smooth irreducible threefold.
\end{proof}

\subsection{The space $\Mo_X(2,1,6)$ as a subspace of
$\Mo_S(2,1,6)$}
\label{sec-tyurin}
In this section we let $X$ be a {\em general} prime Fano threefold of genus $7$.
Let $S$ be a general hyperplane section of $X$. Assume in
particular that $S$ is a K3 surface of Picard number $1$ and sectional genus $7$.
In this paragraph we will show that $\Mo_X(2,1,6)$ is
isomorphic to a
Lagrangian submanifold of $\Mo_S(2,1,6)$.
This provides an instance of a general remark of Tyurin,
\cite{tyurin:fano-cy}.

Note that given a sheaf $F\in\Mo_X(2,1,6)$, in view of Proposition
\ref{agglomerato}, part \eqref{minimum+1}, its restriction $F_S = F \ts \OO_S$ is torsion-free as
soon as $S$ does not contain lines.
So in this section we will assume that the hyperplane section
$S$ does not contain any line.

We want to prove now that given {\em any} sheaf $F\in\Mo_X(2,1,6)$, its
restriction $F_S$ is stable. Assume first that $F_S$ is locally free.
By Proposition \ref{commento}, we know that $\HH^1(X,F(-2))=0$.
Hence, by the exact sequence \eqref{eq:restrizione} with $t=0$,
and by Proposition \ref{agglomerato}, we get
 $\HH^0(S,F_S(-1))=0$.
This implies that $F_S$ is stable,
by Hoppe's criterion.
Assume now that $F_S$ is not locally free.
Then by Proposition \ref{agglomerato}, part \eqref{minimum+1} it fits in the exact sequence
\eqref{doubledual} for some stable vector bundle $E\in\Mo_X(2,1,5)$ and a line $L\subset X$.
Since $E$ is stable and ACM by Proposition \ref{agglomerato}, part \eqref{minimum}, it easily follows
by the restriction sequence that $\HH^0(S,E_S(-1))=0$, hence $E_S$ is stable.
Then if there exists a destabilizing subsheaf of $F_S$,
it would destabilize also $E_S$, a contradiction.

Hence we define a restriction map:
\[
\rho_S : \Mo_X(2,1,6) \to \Mo_S(2,1,6), \qquad F \mapsto
F_S.
\]

\begin{lem}[Tyurin] \label{tyurin}
Let $S$ be a general hyperplane section of $X$.
Then the map $\rho_S$ is a closed immersion.
\end{lem}

\begin{proof}
We prove that the differential $d (\rho_S)_{F}$ is injective
at any point $F$ of $\Mo_X(2,1,6)$.
Applying the functor $\Hom_X(F,-)$ to \eqref{eq:restrizione} (with
$t=0$) we get:
\[
\Ext^1_X(F,F(-1)) \to \Ext^1_X(F,F) \xr{\delta}
\Ext^1_X(F,F_S).
\]

Note that the leftmost term in the above sequence vanishes
since by Serre duality $\Ext^1_X(F,F(-1)) \cong \Ext^2_X(F,F)^*$,
and this group vanishes for $\Mo_X(2,1,6)$ is a non-singular threefold by Theorem
\ref{thm:lune}, part \ref{lune-B}. So
$\delta$ is injective.
Recall that $\Ext^1_X(F,F)$ is naturally isomorphic to
$T_{F} \Mo_X(2,1,6)$,
hence we are done if we prove that $\Ext^1_X(F,F_S)$ is
naturally isomorphic to $T_{F_S} \Mo_S(2,1,6)$.

Let us prove that $\Ext^1_X(F,F_S) \cong \Ext^1_S(F_S,F_S)$.
Indeed, denoting by $\iota$ the inclusion of the surface $S$ in $X$, we have
\[\Ext^1_X(F,F_S) \cong \Ext^1_X(F,\iota_* \iota^* F) \cong
\Ext^1_S(\iota^* F,\iota^* F)\cong \Ext^1_S(F_S,F_S),\]
where the second isomorphism above holds if
$\LLb_k \iota^*(F)=0$ for any $k>0$. But this is true since $F_S$ is torsion-free.
This concludes the proof.
\end{proof}

\begin{lem}\label{dariooo}
Let $F,F'$ be two sheaves in $\Mo_{X}(2,1,6)$.
\begin{enumerate}[i)]
\item If $F$ is globally generated, then $\Ext^1_X(F,F'(-1))\neq 0$ if
  and only if $F'=\theta(F)$.
\item
If $F$ is not globally generated over the line $L\subset X$, then
$\Ext^{1}_{X}(F,F'(-1))\neq 0$ if and only if $F$ is not locally
free over $L$.
\end{enumerate}
\end{lem}

\begin{proof}
  We assume first that $F$ is globally generated.
  We have an exact sequence of the form:
  \begin{equation}
    \label{eq:Fgg}
    0 \to K \to \HH^{0}(X,F) \ts \OO_{X} \to F \to 0,
  \end{equation}
  and $K$ is a reflexive sheaf by \cite[Proposition 1.1]{hartshorne:stable-reflexive},
  hence $K(1)$ is a locally free
  sheaf in $\Mo_{X}(2,1,6)$, which is precisely $\theta(F)$ by
  Proposition \ref{prop:lune}.
  Applying $\Hom_{X}(-,F'(-1))$ to \eqref{eq:Fgg}, one obtains
  $\Hom_{X}(K,F'(-1)) \neq 0$, so $F' \cong K(1)\cong \theta(F)$.

  Assume now that $F$ is not globally generated.
  Applying the functor $\Hom_X(-,F'(-1))$ to \eqref{eq:Ky} we get
  $\Ext^1_X(F,F'(-1)) = \Ext_X^1(\OO_L,F')$. Indeed we compute
  $\Hom_X(I,F'(-1))=0$ by stability and we can easily check that $\Ext^1_X(I,F'(-1))=0$ by
  applying $\Hom_X(-,F'(-1))$ to \eqref{eq:KE}.
  Hence, since by Proposition \ref{agglomerato}, part \eqref{minimum+1} we have that
  $\Ext^{1}_{X}(\OO_L,F')\neq 0$ if and only if $F'$ is not locally free
  over the line $L$, we conclude the proof.
\end{proof}

\begin{prop}\label{iniett}
Let $F \not\cong F'$ be two sheaves in $\Mo_{X}(2,1,6)$, fix a general hyperplane
section $S$, and assume $F_{S} \cong F_{S}'$. Then,
\begin{enumerate}[i)]
\item \label{2010-uno} assuming $F$ globally generated, we have $F' = \theta(F)$,
\item \label{2010-due} otherwise there are lines $L,L'\subset X$ with $L\cap L' \in S$ such
  that $F$ (resp. $F'$) is not globally generated over $L'$
  (resp. $L$) and not locally free over  $L$ (resp. $L'$).
\end{enumerate}

Moreover, the set of sheaves $F$ in $\Mo_{X}(2,1,6)$ admitting a sheaf
$F'$ such that $F_S \cong F'_S$ is finite.
\end{prop}

\begin{proof}
  Assume that the sheaf $F_{S}$ is isomorphic to $F'_{S}$.
  We compose the projection $F \to F_S$ with this isomorphism to
  obtain a nonzero map $F \to F'_S$.
  In view of the exact sequence:
  \[
  0 \to F'(-1) \to F' \to F'_S \to 0,
  \]
  it is easy to see that the map $F \to F'_S$ lifts to an
  isomorphisms $F \cong F'$ if $\Ext^{1}(F,F'(-1)) = 0$.
  We can assume thus that this group is non-trivial. This proves
  \eqref{2010-uno} if $F$ is globally generated thanks to Lemma \ref{dariooo}.

  Let us now prove the final statement, still under the assumption that
  $F$ is globally generated.
    Set $F'=\theta(F)=K(1)$, and recall that $F'$ is locally free.
  We consider the symmetric square of \eqref{eq:Fgg}:
  $$ 0 \to \Sym^2 (F')^* \to \HH^{0}(X,F) \ts (F')^* \to
  \wedge^2 \HH^{0}(X,F) \ts \OO_X \to \OO_X(1) \to 0,$$
  and we take global  sections. Since $\HH^0(X,(F')^*)=0$, $\wedge^2(F')^*\cong\OO_X(-1)$ and
  $\HH^1(X,\Sym^{2} (F')^{*}) \cong \HH^1(X, (F')^{*}\ts (F')^{*})\cong \Ext_X^2(F',F')^*=0$,
  we obtain an injection $\iota_{F} : \wedge^{2} \HH^{0}(X,F) \mono
  \HH^{0}(X,\OO_{X}(1))$.
  Note that $\dim(\cok(\iota_{F}))=3$, hence setting $\Lambda_{F} = \p(\cok(\iota_{F}))\subset
  \p^{8}=\p(\HH^{0}(X,\OO_{X}(1)))$, we define a correspondence:
  \[\Lambda: \Mo_X(2,1,6)\to \G(2,8), \quad \Lambda:F\mapsto \Lambda_F.\]
  Clearly we have $\dim(\im(\Lambda))\le3$.

  Now we fix a general hyperplane section $S$.
  Taking global sections of the restriction of the symmetric square of \eqref{eq:Fgg},
  we obtain an exact commutative diagram:
  \begin{equation}
    \label{eq:wedge}
    \xymatrix@-2ex{
      & 0 \ar[r] \ar[d] & \wedge^{2} \HH^{0}(X,F) \ar@{=}[d]
\ar[r]^{\iota_{F}} &
      \HH^{0}(X,\OO_{X}(1)) \ar[d] \\
      0 \ar[r] & \HH^{1}(S,\Sym^{2} F_{S}^{*}) \ar[r] &
\wedge^{2} \HH^{0}(S,F_{S}) \ar[r] & \HH^{0}(S,\OO_{S}(1)).
    }
  \end{equation}
  Note that $\HH^{1}(S,\Sym^{2} F_{S}^{*})\neq 0$.
  Indeed since $K_S\cong F_S^*$,
  then the exact sequence \eqref{eq:Fgg} (restricted to $S$)
  provides a non-trivial element
  in $\Ext^1_S(F_S,F_S^*)\cong \HH^1(S,F_S^*\ts F_S^*) \cong \HH^1(S,\Sym^2  F_S^*)$,
  where we use $\wedge^2  F_S^*\cong\OO_S(-1)$.

  Then the diagram \eqref{eq:wedge} induces a projection
  $\HH^{0}(S,\OO_{S}(1))\to \cok(\iota_{F})$ and so
  the hyperplane defining the surface $S$ must
  contain $\Lambda_{F}$. We denote by $\G_S\cong\G(2,7)$
  the set of planes of $\G(2,8)$
  contained in $\p(\HH^{0}(S,\OO_{S}(1)))=\p^7$, where we use the
  projective notation for the Grassmannians and the projectivization
  of a vector space is the variety of all its hyperplanes.
  We have proved that if $\rho_S$ is not injective at $F$, then $\Lambda_F\in \G_S$.
  Clearly $\G_S\subset\G(2,8)$ is a subvariety of codimension $3$ and corresponds to the
  choice of a general global section of the rank $3$ universal bundle on $\G(2,8)$,
  which is globally generated.
  Hence, since the section corresponding to $S$ is general,
    using Bertini theorem for globally generated vector bundles (see
  e.g. \cite{kleiman:grassmannians}, or \cite{ottaviani-codim}),
    we conclude that the set of planes
  contained in $\im(\Lambda)\cap \G_S$ must be finite.
  This proves the second statement.

  It remains to prove the proposition when $F$ is not globally
  generated (say over a line $L'
  \subset X$, see Lemma \ref{lem:F-not-gg}).
  The sheaf $F_{S}$ thus fails to be globally generated over
  the point $x = L' \cap S$.
  In turn, the sheaf $F'$ is not globally generated, say
  over a line $L \subset X$, and we must have either $L=L'$, or $L\cap
  L' = x$.

  In the first case, we will show $F\cong F'$.
  Indeed, by Lemma \ref{lem:F-not-gg}
  we know that $F$ and $F'$ fit in \eqref{eq:Ky}, \eqref{eq:KE}, and
    \begin{align}
    \label{eq:Ky'}&  0 \to I' \to F' \to \OO_{L}(-1) \to 0,\\
    \label{eq:KE'}& 0 \to \EE_{y'}^* \to \HH^0(X,F') \ts \OO_X \to  I' \to 0.
    \end{align}
  Restricting the sequences \eqref{eq:Ky} and \eqref{eq:Ky'} to $S$, we get $I_S\cong I'_S$.
  Note that $\HH^1(S, \EE_{y'}^*)=0$ by Proposition \ref{agglomerato}, part \ref{annullarsi}.
  Then from the sequences \eqref{eq:KE} and \eqref{eq:KE'},
  we deduce that $(\EE_{y}^*)_S\cong (\EE_{y'}^*)_S$.
  But this implies the isomorphism $\EE_{y}\cong \EE_{y'}$, and thus
  the isomorphism $F\cong F'$.
  Indeed the restriction map from $\Mo_X(2,1,5)\to\Mo_S(2,1,5)$ is
  known to be injective, because it corresponds to the embedding of
  $\Gamma$ as linear section of $\Mo_S(2,1,5)$.

  Then we may assume $L\cap L' = x$.
  By Lemma \ref{dariooo}, since $\Ext^{1}_{X}(F,F'(-1))\neq 0$, then we know that $F$ is not locally free
  over the line $L$.
  From the sequences \eqref{eq:KE} and \eqref{eq:KE'},
  we deduce that $(\EE_{y}^*)_S\cong (\EE_{y'}^*)_S$ so $y = y'$.
  Set $E = F^{**}$, and recall that $E$
  belongs to $\Mo(2,1,5)$, so $E = \EE_z$ for some $z \in \Gamma$.
  Then we have $(F'_S)^{**} \cong E_S$, hence $(F')^{**} \cong E$,
  again since the restriction $\Mo_X(2,1,5) \to \Mo_S(2,1,5)$ is
  injective.
  Hence $F'$ is not locally free either, and in turn its
  non-locally
  free locus must be $L'$, i.e. we have:
  \begin{equation}
    \label{eq:doubledualprime}
    0 \to F' \to E \to \OO_{L'}\to 0.
  \end{equation}
  We have thus proved \eqref{2010-due}.
  Note also that, by Proposition \ref{pulcinotto}, the conic $C = L\cup L'$
  is such that $\phx(\OO_C)=\OO_{y,z}$.

  To check the second part of the statement in this case,
  we note that, since $S$ is general, the curve spanned by the
  intersection points
  of lines in $\sH^{0}_{1}(X)$ meets $S$ at a finite number
  of points. Any pair of sheaves
  $(F,F')$ such that we $F_S \cong F'_S$ with $F$ not globally
  generated determines one such point $x$. And conversely given such
  $x$ we may choose in a finite number of ways two lines $L,L'$ through $x$
  (for there are finitely many lines through $x$), and a point $z$ of
  $\phx(\OO_{L\cup L'})$. This choice determines $F$ and $F'$ as
  kernels of $\EE_Z \to \OO_L$ and $\EE_Z \to \OO_{L'}$ , so the
  set of pairs of non-globally generated sheaves $(F,F')$ having
  isomorphic restrictions to $S$ is finite.
\end{proof}

Recall that $\Mo_S(2,1,6)$ is a holomorphic symplectic
manifold with
respect to the Mukai form, see \cite{mukai:symplectic}.

\begin{thm} \label{ty}
Let $X$ be a general
prime Fano threefold of genus $7$, $S$ be a
general
hyperplane section of $X$, and let $\rho_S$ be the restriction
map
from $\Mo_X(2,1,6)$ to $\Mo_S(2,1,6)$.
The image $\rho_S(\Mo_X(2,1,6))$ is a Lagrangian
subvariety of $\Mo_S(2,1,6)$ with finitely many double points.
\end{thm}

\begin{proof}
  Recall that since $X$ is general, then  $\Mo_X(2,1,6)$ is smooth,
  by Theorem \ref{thm:lune}, part \ref{lune-B}.
  We have seen in Lemma \ref{tyurin} and Proposition \ref{iniett} that
  $\rho_S$
  is a closed embedding outside a finite subset $R$ of
  $\Mo_X(2,1,6)$.

  By the proof of Proposition \ref{iniett}, we have
  that the
  preimage of a singular point of $\rho_S(\Mo_X(2,1,6))$
  consists of
  precisely two points of $\Mo_X(2,1,6)$, hence the singular
  locus
  consists of double points.

  The image $\rho_S(\Mo_X(2,1,6)\setminus R)$ is a Lagrangian
  submanifold
  by a remark of Tyurin, see \cite[Proposition
  2.2]{tyurin:fano-cy}.
\end{proof}



\def\cprime{$'$} \def\cprime{$'$} \def\cprime{$'$}
\def\cprime{$'$}
\def\cprime{$'$} \def\cprime{$'$} \def\cprime{$'$}
\def\cprime{$'$}
\providecommand{\bysame}{\leavevmode\hbox
to3em{\hrulefill}\thinspace}
\providecommand{\MR}{\relax\ifhmode\unskip\space\fi MR }
\providecommand{\MRhref}[2]{
  \href{http://www.ams.org/mathscinet-getitem?mr=#1}{#2}
}
\providecommand{\href}[2]{#2}

\end{document}